\documentclass{amsart}

\usepackage[T1]{fontenc} 
\usepackage{lmodern} 

\usepackage{amscd, amssymb, amsmath, amsthm}
\usepackage{mathabx} 
\usepackage{enumerate, latexsym, mathrsfs}
\usepackage{xypic}
\usepackage{graphics,graphicx,psfrag, mathrsfs}
\usepackage[
		bookmarks=true, bookmarksopen=true,%
    bookmarksdepth=3,bookmarksopenlevel=2,%
    colorlinks=true,%
    linkcolor=blue,%
    citecolor=blue]{hyperref}
\newtheorem{theorem}{Theorem}[section]
\newtheorem{lemma}[theorem]{Lemma}
\newtheorem{proposition}[theorem]{Proposition}
\newtheorem{corollary}[theorem]{Corollary}

\newtheorem{condition}[theorem]{Condition}
\newtheorem{formula}[theorem]{Formula}

\theoremstyle{definition}
\newtheorem{definition}[theorem]{Definition}

\newtheorem{notation}[theorem]{Notation}

\newtheorem{example}[theorem]{Example}

\newtheorem{remark}[theorem]{Remark}

\numberwithin{equation}{section}

\newcommand{\Natural}{{\mathbb N}}
\newcommand{\Real}{{\mathbb R}}
\newcommand{\Rational}{{\mathbb Q}}
\newcommand{\Complex}{{\mathbb C}}
\newcommand{\Integral}{{\mathbb Z}}

\newcommand{\ud}{{\mathrm{d}}}

\title{Entropy versus volume via Heegaard diagrams}
     
\author[Yi Liu]{%
        Yi Liu} 
\address{%
        Beijing International Center for Mathematical Research, Peking University\\
				Beijing 100871, China P.R.} 
\email{%
    liuyi@bicmr.pku.edu.cn}

\thanks{Partially supported by NSFC Grant 11925101, 
and National Key R\&D Program of China 2020YFA0712800}

\date{%
 \today} 

\begin{document}

\begin{abstract}
	The following inequalities are established,
	improving a former inequality due to Kojima.
	For any closed arithmetic hyperbolic $3$--manifold fibering over a circle,
	the entropy of the pseudo-Anosov monodromy is 
	bounded by the hyperbolic volume of the $3$--manifold,
	up to a universal constant factor.
	For any closed hyperbolic $3$--manifold fibering over a circle with systole $\geq\varepsilon>0$,
	the entropy is bounded by the hyperbolic volume times $\log(3+1/\varepsilon)$,
	up to a universal constant factor.
	The proof relies
	on Heegaard Floer homology and hyperbolic geometry.
\end{abstract}

\maketitle

\section{Introduction}
In low dimensional topology, \emph{Entropy versus Volume} refers to 
a charming topic
about mapping classes and their mapping tori.
It seeks for primary quantitative comparison between 
surface dynamics and $3$--manifold geometry.
There have been many works on this topic
via combinatorial objects associated the surface,
such as the curve complex and the pants complex.
In this paper, we take a different approach,
via Heegaard diagrams associated to the $3$--manifold.
By this approach,
we are able combine techniques from Heegaard Floer homology
and hyperbolic geometry.

To avoid inessential technicalities,
we only discuss connected closed orientable surfaces,
typically of genus $\geq2$.
In this case, 
there are plenty of pseudo-Anosov mapping classes.
Their mapping tori are homeomorphic to
isometrically unique,
orientable closed hyperbolic $3$--manifolds,
as the geometrization theorem and the Mostow rigidity theorem tell us.

Let $S$ be a connected closed orientable surface of genus $\geq2$.
Denote by $\mathrm{Mod}(S)$ the mapping class group of $S$,
whose elements are 
the isotopy classes of orientation-preserving self-homeomorphisms of $S$.
For any pseudo-Anosov mapping class $[f]\in\mathrm{Mod}(S)$,
Kojima obtains conditional linear comparisons in two directions \cite[Theorem 1]{Kojima_entropy}:
\begin{equation}\label{v_e_Kojima}
\mathrm{Vol}(M_f)\leq \mathrm{const}_{S}\cdot \mathrm{Ent}([f]),
\end{equation}
where $\mathrm{Vol}(M_f)$ denotes the hyperbolic volume of 
the mapping torus $M_f$ as a hyperbolic $3$--manifold,
and 
$\mathrm{Ent}([f])$ denotes the mapping class entropy of $[f]$
(see Section \ref{Subsec-monodromy_entropy});
and, assuming $\mathrm{Syst}(M_f)\geq\varepsilon>0$,
\begin{equation}\label{e_v_Kojima}
\mathrm{Ent}([f])\leq \mathrm{const}_{S,\varepsilon}\cdot \mathrm{Vol}(M_f),
\end{equation}
where $\mathrm{Syst}(M_f)$ denotes the systole of $M_f$,
(that is, the hyperbolic length of the shortest geodesic).
The constants depend on data as indicated in the subscript.

Kojima's inequalities (\ref{v_e_Kojima}) and (\ref{e_v_Kojima}) are derived from
Brock's two-sided linear comparison between the Weil--Petersson translation length
and the volume \cite{Brock_wp_vs_vol}:
$$\mathrm{const}_S^{-1}\cdot\ell_{\mathtt{WP}}([f])\leq \mathrm{Vol}(M_f)\leq \mathrm{const}_S\cdot\ell_{\mathtt{WP}}([f]).$$
Although Brock's inequalities have no restriction on the systole, 
the dependence on the systole in Kojima's inequality (\ref{e_v_Kojima})
cannot be removed.
This is because of the following examples due to Long and Morton \cite{Long_Morton}:
For any genus $\geq2$,
there exists some pseudo-Anosov sequence 
$[f_1],[f_2],\cdots$ in $\mathrm{Mod}(S)$,
such that the entropy of $[f_n]$ tends to infinity, 
while the volume of $M_{f_n}$ stays uniformly bounded.
On the other hand, Kojima and McShane obtain
an effective inequality, improving Kojima's inequality (\ref{v_e_Kojima}), \cite[Theorem 1.1]{Kojima_McShane}:
$$\mathrm{Vol}(M_f)\leq 3\pi\cdot\mathrm{Ent}([f])\cdot|\chi(S)|,$$
where $\chi(S)$ denotes the Euler characteristic of $S$.
The upper bound here is very natural (if not sharp).
The linear-type dependence of the coefficient on $|\chi(S)|$ is optimal,
as is evident
by considering characteristic finite covers of $S$ and lifts of $[f]$.
Furthermore, the Kojima--McShane inequality can be refined into two steps, 
namely, the Brock--Bromberg inequality \cite{Brock--Bromberg}:
$$\mathrm{Vol}(M_f)\leq 3\cdot\sqrt{\pi/2}\cdot\ell_{\mathtt{WP}}([f])\cdot\sqrt{|\chi(S)|}$$
and the Linch inequality \cite{Linch}:
$$\ell_{\mathtt{WP}}([f])\leq \sqrt{2\pi}\cdot\mathrm{Ent}([f])\cdot\sqrt{|\chi(S)|}.$$

See also \cite{KKT_ent_vs_vol} for numerical experiments with small genera.

For our approach via Heegaard diagrams,
it is more convenient to speak of connected closed orientable $3$--manifolds
$M$ and their fibered classes $\phi\in H^1(M;\Integral)$.
This is merely change of perspective from $(S,[f])$ to 
$(M,\phi)=(M_f,\mathrm{PD}([S]))$, 
(fixing orientations of $S$ and $M$).
Accordingly, $\mathrm{Ent}([f])$ becomes what we call
the monodromy entropy of $\phi$,
denoted as $\mathrm{Ent}(\phi)$,
 (see Section \ref{Sec-entropy_fibered}).

We establish the following improvements of 
Kojima's inequality (\ref{e_v_Kojima}).

\begin{theorem}\label{main_ent_vs_vol_general}
	The following inequality holds
	for any orientable closed hyperbolic $3$--manifold $M$
	and any fibered class $\phi\in H^1(M;\Integral)$.
	$$\mathrm{Ent}(\phi)\leq 10^{20}\cdot \mathrm{Vol}(M)\cdot\log\left(3+\frac{1}{\mathrm{Syst}(M)}\right)$$
\end{theorem}

\begin{theorem}\label{main_ent_vs_vol_arithmetic}
	There exists some universal constant $C>0$,
	such that the following inequality holds
	for any arithmetic orientable closed hyperbolic $3$--manifold $M$
	and any fibered class $\phi\in H^1(M;\Integral)$.
	$$\mathrm{Ent}(\phi)\leq C\cdot \mathrm{Vol}(M)$$
\end{theorem}

Theorem \ref{main_ent_vs_vol_general} shows that the dependence on 
the genus of $S$ in (\ref{e_v_Kojima}) can be removed.
This might seem unusual at first glance,
however, 
for any fixed $M$ and varying fibered classes $\phi$,
the same phenomenon can readily be confirmed using well-known facts.
See Theorem \ref{entropy_value_set} 
for an argument based on Fried's early work 
about pseudo-Anosov flow cross-sections \cite{Fried_flow}.
In Theorem \ref{main_ent_vs_vol_general},
the function $\log(3+1/\varepsilon)$ is picked to embody the features
$\log(3+1/\varepsilon)\sim\log(1/\varepsilon)$ for $\varepsilon\to 0^+$,
and $\log(3+1/\varepsilon)\geq \log e=1$ for all $\varepsilon>0$.
The type of dependence on $\mathrm{Syst}(M)$ in Theorem \ref{main_ent_vs_vol_general} 
is actually optimal.
We demonstrate with an example arising from Long and Morton's construction
in Section \ref{Sec-example_optimal}.

Theorem \ref{main_ent_vs_vol_arithmetic} would follow 
immediately from Theorem \ref{main_ent_vs_vol_general},
if the Lehmer Conjecture on the Mahler measure of algebraic integers holds,
or at least,
if the Shortest Geodesic Conjecture holds 
for all arithmetic closed hyperbolic $3$--manifolds,
(see \cite[Section 4.4]{NR_arithmetic}).
Although these conjectures are still open,
we find some alternative ingredient 
from a recent work of Fr\k{a}czyk \cite{Fraczyk_arithmetic},
which suffices for proving Theorem \ref{main_ent_vs_vol_arithmetic}.
With the same ingredient, Fr\k{a}czyk proves 
a conjucture of Gelander
on (homotopy) simplicial triangulation for
torsion-free arithmetic lattices in $\mathrm{PSL}(2,\Complex)$
\cite[Theorem 1.5]{Fraczyk_arithmetic}.
Based on Fr\k{a}czyk's work, 
there seems to be no essential difficulty
to work out an explicit $C$ for Theorem \ref{main_ent_vs_vol_arithmetic}, 
(see Remark \ref{main_ent_vs_vol_arithmetic_remark}).


\begin{corollary}\label{reverse_Linch}
For any connected closed orientable surface $S$ of genus $\geq2$ and any pseudo-Anosov mapping class $[f]\in\mathrm{Mod}(S)$,
the following inequalities hold.
\begin{enumerate}
\item Assuming $\mathrm{Syst}(M_f)\geq \varepsilon>0$,
	$$\mathrm{Ent}([f])\leq \mathrm{const}\cdot\ell_{\mathtt{WP}}([f])\cdot\sqrt{|\chi(S)|}\cdot\log(3+1/\varepsilon).$$
\item Assuming $M_f$ to be arithmetic,
	$$\mathrm{Ent}([f])\leq \mathrm{const}\cdot\ell_{\mathtt{WP}}([f])\cdot\sqrt{|\chi(S)|}.$$
\end{enumerate}
\end{corollary}

Corollary \ref{reverse_Linch} (pointed out by Samuel Taylor)
follows immediately from Theorems \ref{main_ent_vs_vol_general} and \ref{main_ent_vs_vol_arithmetic}
and the aforementioned Brock--Bromberg inequality.
It can be viewed as conditional reverses of the Linch inequality.
It improves two known inequalities,
namely, $\mathrm{Ent}([f])\leq \mathrm{const}_{S,\varepsilon}\cdot\ell_{\mathtt{WP}}([f])$,
following from (\ref{e_v_Kojima}),
and $\mathrm{Ent}([f])\leq \mathrm{const}_{S,d}\cdot\ell_{\mathtt{WP}}([f])$,
following from 
a finiteness result regarding arithmetic surface bundles
due to Bowditch--Maclachlan--Reid \cite[Corollary 4.4]{BMR_arithmetic_surface_bundle},
where $d$ denotes any given upper bound of the degree of the invariant trace field of $M_f$.

Our combination of different methods 
is reflected in two main technical results,
as Theorems \ref{main_ent_vs_hpl} and \ref{main_vol_vs_hpl} below.
Theorem \ref{main_ent_vs_hpl} is proved by means of Heegaard Floer homology.
Theorem \ref{main_vol_vs_hpl} is proved by means of hyperbolic geometry.

The bridge between these theorems is a quantity called 
the Heegaard presentation length,
which we introduce in Section \ref{Sec-Heegaard_pl}.
In brief,
the \emph{Heegaard presentation length} $\ell_{\mathtt{He}}(M)$
of a connected closed orientable $3$--manifold $M$
is the smallest presentation length
among all finite presentations of $\pi_1(M)$
which arise from Heegaard diagrams of $M$
(Definition \ref{def_Heegaard_pl}).
This is a topological invariant of $M$,
comparable to well-known topological complexities,
such as the Matveev complexity and the Kneser complexity,
but not so much to the presentation length of $\pi_1(M)$,
(see Section \ref{Subsec-complexities}).

\begin{theorem}\label{main_ent_vs_hpl}
	Let $M$ be a connected closed orientable $3$--manifold.
	Then, for any connected finite cover $M'$ of $M$ 
	and any primitive fibered class $\phi'\in H^1(M';\Integral)$ 
	of fiber genus $\geq3$,
	the following inequality holds.
	$$\mathrm{Ent}(\phi')\leq [M':M]\cdot(\ell_{\mathtt{He}}(M)-1)\cdot\log3$$
	%
\end{theorem}

For any hyperbolic tube,
the \emph{wrist} of the tube, as we call,
refers to the hyperbolic circumference of 
any embedded, totally geodesic meridional disk.
We denote by $\mathrm{Wri}(V)$ the wrist of a hyperbolic tube $V$.
Among the three geometric quantities 
$\mathrm{Wri}(V)$, $\mathrm{Syst}(V)$, and $\mathrm{Vol}(V)$,
any two determine the third, (see Formula \ref{tube_formula}).

\begin{theorem}\label{main_vol_vs_hpl}
	Let $M$ be an orientable closed hyperbolic $3$--manifold.
	Suppose that $V_1,\cdots,V_s\subset M$ 
	are embedded, mutually disjoint, hyperbolic tubes with boundary.
	Denote by $W=M\setminus \mathrm{int}(V_1\cup\cdots\cup V_s)$
	the complementary $3$--manifold with boundary.
	
	Suppose that for some constant $0<\epsilon\leq1$,
	the compact distance $\epsilon$--neighborhood of $\partial W$ in $M$ is bicollar,
	and any point in $W$ 
	is the center of an embedded compact hyperbolic ball in $M$ of radius $\epsilon$.
	Then, the following inequality holds.
	$$\ell_{\mathtt{He}}(M)\leq 10^{22}\cdot\left(
	\epsilon^{-3}\cdot\mathrm{Vol}(W)+
	\epsilon^{-1}\cdot\sum_{i=1}^s \mathrm{Wri}(V_i)\right)$$
\end{theorem}

Theorem \ref{main_ent_vs_hpl} is proved in Section \ref{Sec-ent_vs_hpl}.
Theorem \ref{main_vol_vs_hpl} is proved in Section \ref{Sec-vol_vs_hpl}.

Theorem \ref{main_ent_vs_vol_arithmetic}
follows from Theorems \ref{main_ent_vs_hpl} and \ref{main_vol_vs_hpl},
together with available estimates for arithmetic hyperbolic $3$--manifolds,
including Fr\k{a}czyk's work.
See Section \ref{Sec-ent_vs_vol_arithmetic} for the proof of Theorem \ref{main_ent_vs_vol_arithmetic}.

Theorem \ref{main_ent_vs_vol_general}
relies on stronger intermediate results
toward the proofs of Theorems \ref{main_ent_vs_hpl} and \ref{main_vol_vs_hpl}.
See Section \ref{Sec-ent_vs_vol_general} for the proof of 
Theorem \ref{main_ent_vs_vol_general}.

\subsection*{Methods}
We explain the key ideas 
toward the proofs of Theorems \ref{main_ent_vs_hpl} and \ref{main_vol_vs_hpl}.
For the rest of the introduction,
we assume certain familiarity 
with backgrounds related to our discussion.

\subsubsection*{Entropy versus Heegaard Presentation Length}
To prove Theorem \ref{main_ent_vs_hpl},
we can reduce to the basic case $M'=M$,
thanks to a linear comparison
$\ell_{\mathtt{He}}(M')-1\leq [M':M]\cdot(\ell_{\mathtt{He}}(M)-1)$
(Corollary \ref{hpl_cover}).
It also suffices to prove for any primitive fibered class
$\phi\in H^1(M;\Integral)$, whose connected fiber we denote as $S$ and 
monodromy as $[f]\in\mathrm{Mod}(S)$.
The substantial difficulty lies in
bounding $\mathrm{Ent}(\phi)=\mathrm{Ent}([f])$ from above.
Few classical invariants do this job,
but Heegaard Floer homology does.

We start by recalling the characterization of
$\mathrm{Ent}([f])$ as the limit of
$(1/m)\cdot\log N(f^m)$, for $m\to\infty$,
where $N(f^m)$ denotes the Nielsen number of $f^m$.
If we can efficiently bound $N(f)$, 
and 
if we can similarly bound $N(f^m)$ 
by passing to $m$--cyclic covers of $M$ dual to $S$,
then we might be able to reach a good upper bound for $\mathrm{Ent}([f])$.

Under the assumption that $S$ has genus $\geq3$,
the Nielsen number $N(f)$ is bounded by
the free rank of the next-to-top term
$\mathrm{HF}^+(M,\phi,\mathrm{genus}(S)-2)$
of the plus version of Heegaard Floer homology,
with respect to $\phi$,
(see Notation \ref{phi_grading_notation}).
This fact should be well-known to experts.
However, the proof of this fact involves
several different Floer homology theories associated to $3$--manifolds or surface autormophisms.
In Appendix \ref{Sec-next_to_top_appendix},
we supply an exposition for the reader's convenience.

Moreover,
we can bound $N(f)$ by $2$ times
the free rank of the hat version 
$\widehat{\mathrm{HF}}(M,\phi,\mathrm{genus}(S)-2)$,
by applying a generalized adjunction inequality
with $U$--actions, due to Wu \cite{Wu_u_action}.
Since $\widehat{\mathrm{HF}}(M,\phi,\mathrm{genus}(S)-2)$
is a direct summand of $\widehat{\mathrm{HF}}(M)$,
it suffices to bound the free rank of $\widehat{\mathrm{HF}}(M)$.
This can be done by estimating the number of generators
in any chain complex $\widehat{\mathrm{CF}}(\Sigma,\boldsymbol{\alpha},\boldsymbol{\beta},z)$
arising from a weakly admissible pointed Heegaard diagram 
$(\Sigma,\boldsymbol{\alpha},\boldsymbol{\beta},z)$ for $M$.

The above discussion will lead to an estimate of the form
$$N(f)\leq 2k_1\cdots k_g,$$
where $g$ denotes the genus of $\Sigma$,
and $k_i$ denotes the number of intersection points on the $i$--th $\boldsymbol{\alpha}$--curve
(with all the $\boldsymbol{\beta}$--curves).
What remain unclear are two technical issues, as follows.

First, a Heegaard diagram $(\Sigma,\boldsymbol{\alpha},\boldsymbol{\beta})$
that realizes the Heegaard presentation length $\ell_{\mathtt{He}}(M)$
is typically not weakly admissible with respect to any marked point $z$. 
If we convert $(\Sigma,\boldsymbol{\alpha},\boldsymbol{\beta})$ 
into a weakly admissible $(\Sigma,\tilde{\boldsymbol{\alpha}},\boldsymbol{\beta},z)$
using the winding trick, as usual, 
we have to create many extra intersections.
Would the upper bound $2\tilde{k}_1\cdots\tilde{k}_g$ still be useful?

Secondly, $f^m$ corresponds to the $m$--cyclic cover $M'_m=M_{f^m}$.
To bound $N(f^m)$,
we might want to construct a Heegaard diagram for $M'_m$ somehow from 
$(\Sigma,\boldsymbol{\alpha},\boldsymbol{\beta})$,
but Heegaard diagrams do not naturally lift to finite covers.
How to deal with finite covers?

The second issue is relatively easy to address.
In fact, 
there is a variant of Heegaard Floer homology with multiply pointed Heegaard diagrams,
which is suitable for the finite covering setting.
For example, 
weakly admissible $l$--pointed Heegaard diagrams naturally pull back to connected $d$--fold covers,
giving rise to weakly admissible $ld$--pointed Heegaard diagrams.
See Section \ref{Subsec-mpHD} for a review.
We can make use of multiply pointed Heegaard diagrams to estimate $N(f^m)$.

The first issue also has a satisfactory resolution.
In Section \ref{Sec-winding}, 
we prove an efficient version of the winding trick (Lemma \ref{winding_efficient}),
by examining Ozsv\'{a}th and Szab\'{o}'s original procedure \cite[Section 5]{OS_hf_invariants} step by step.
With a little tricky control,
we can bound the increment of total number of intersections 
by some quadratic expression in $k=k_1+\cdots+k_g$,
whose coefficients involve $b=b_1(M)$.
Moreover, the winding only affects $b$ numbers among $k_1,\cdots,k_g$.
When passing to $m$--cyclic finite covers,
$b_1(M'_m)$ are all uniformly bounded, 
and the polynomial influence from $k$ becomes linearly negligible after taking logarithm.
Therefore, eventually,
the complication introduced during the winding trick
has no effect to our bound.

We emphasize that our efficient winding trick involves 
a special step of optimization, which only pertains to the quantitative aspect.
That step plays an indispensable role in the overall proof of 
Theorems \ref{main_ent_vs_vol_general} and \ref{main_ent_vs_vol_arithmetic};
see Remark \ref{periodic_domain_estimate_remark}.

With the above issues addressed, 
Theorem \ref{main_ent_vs_hpl} can be proved without obstacles.

\subsubsection*{Volume versus Heegaard Presentation Length}
To prove Theorem \ref{main_vol_vs_hpl},
it suffices to construct an efficient Heegaard diagram for $M$,
whose presentation length does not exceed the asserted upper bound.
If we have an efficient polyhedral cell division of $M$,
we can take a Heegaard surface as some regular neighborhood of the $1$--skeleton,
and pick out some $\boldsymbol{\alpha}$--curves and $\boldsymbol{\beta}$--curves
by selecting some $2$--cells and $1$--cells.
The presentation length of the resulting Heegaard diagram
is at most the total number of edges in the selected $2$--cells.
Therefore,
the task is to construct some efficient polyhedral cell division of $M$.

To clarify our terminology, 
we think of a \emph{polyhedral cell complex} as a cell complex
which admits some simplicial subdivision,
such that every cell is the union of finitely many simplices.
So, for example, 
the number of \emph{edges on a polygonal $2$--cell} $\mathrm{int}(D)\subset M$ 
precisely means 
the number of the preimage components of $1$--cells in $\partial D^2\cong S^1$, 
with respect to the characteristic map $D^2\to M$.

Under the assumptions of Theorem \ref{main_vol_vs_hpl},
intuitively we should be able to 
construct a polyhedral cell division of $W$,
such that the number of polyhedral $3$--cells is bounded by 
$\mathrm{Vol}(W)$, up to constant scalar,
and the number of faces on the polyhedral $3$--cells is bounded by constant.
The constants depend only on the assumed thickness $\varepsilon$ of $W$.
Then, we could extend the polyhedral cell division of $W$
to a polyhedral cell division of $M$,
by inserting a meridional $2$--cell to each $\mathrm{int}(V_i)$,
dividing $\mathrm{int}(V_i)$ into a $3$--cell,
without subdividing $\partial V_i$.
(The characteristic map restricted to the boundary maps 
the polygonal circle combinatorially to the $1$--skeleton of $\partial V_i$.)
The combinatorial path length on the $1$--skeleton of $\partial V_i$
is coarsely equal to the Riemannian arc length on $\partial V_i$,
up to constants depending on $\varepsilon$.
So, 
the number of edges on the inserted polygonal $2$--cell
should be coarsely equal to $\mathrm{Wri}(V_i)$.

Therefore, the resulting polyhedral cell division of $M$ 
should produce coarsely $\mathrm{Vol}(W)$ polygonal $2$--cells
of edge number bounded by a uniform constant, 
and another $s$ polygonal $2$--cells
of edge number bounded 
by coarsely $\mathrm{Wri}(V_1)+\cdots+\mathrm{Wri}(V_s)$ altogether.
Then, it should yield an upper bound for $\ell_{\mathtt{He}}(M)$,
which is more or less like Theorem \ref{main_vol_vs_hpl}.

The above procedure is very close to our actual construction
in Section \ref{Sec-vol_vs_hpl},
except one bothering issue to address.
For closed hyperbolic $3$--manifolds,
it is always simple to construct an efficient polyhedral cell subdivision,
for example, by the familiar Dirichlet--Voronoi division.
However, as $W$ has concave boundary, 
the familiar method does not apply directly,
and we wish to modify as simply as possible, to facilitate estimates.

In our actual construction,
we first create a Dirichlet--Voronoi division of $M$ with respect to
a relatively fine net of points in $W$ not too close to $\partial W$.
Then, 
we truncate the resulting polyhedral cell division,
leaving only the part in $W$.
By choosing the net fine enough, 
we can make sure that 
the truncation yields a polyhedral cell division of $W$.
This is also 
where we use the $\varepsilon$--bicollar assumption in Theorem \ref{main_vol_vs_hpl}.
After that, we proceed as described above,
inserting polygonal $2$--cells in $V_i$,
and complete the proof of Theorem \ref{main_vol_vs_hpl}.

Some special features of our construction deserve a comment.
First, the resulting polyhedral cell division of $M$
has only $s$ exceptional $2$--cells with generally large number of edges.
This seems to be an important point for proving Theorem \ref{main_ent_vs_vol_general},
which does not follow directly from Theorems \ref{main_ent_vs_hpl} and \ref{main_vol_vs_hpl},
(see Lemma \ref{second_division_hpl_estimate} and Section \ref{Sec-ent_vs_vol_general}).
Secondly, 
as we must build a Heegaard diagram,
we have to construct a genuine cell division of $M$.
This is why we cannot construct with an open ball cover and its nerve,
which only outputs a simplicial complex homotopy equivalent to $M$,
(compare \cite[Section 10.1]{Fraczyk_arithmetic}).

\subsection*{Organization}
In Section \ref{Sec-entropy_fibered}, we review fibered classes and monodromy entropy.
In Section \ref{Sec-Heegaard_pl}, 
we introduce Heegaard presentation length and investigate basic properties of this invariant.
In Section \ref{Sec-HF},
we review Heegaard Floer homology.
In Section \ref{Sec-winding},
we establish an efficient version of the winding trick.
Sections \ref{Sec-ent_vs_hpl}, \ref{Sec-vol_vs_hpl}, \ref{Sec-ent_vs_vol_arithmetic}, and \ref{Sec-ent_vs_vol_general}
are devoted to the proofs of 
Theorem \ref{main_ent_vs_hpl}, \ref{main_vol_vs_hpl}, \ref{main_ent_vs_vol_arithmetic}, and \ref{main_ent_vs_vol_general},
respectively.
Section \ref{Sec-example_optimal} is devoted to an example 
justifying the upper bound type in Theorem \ref{main_ent_vs_vol_general}. 
Appendix \ref{Sec-next_to_top_appendix} contains an exposition of Proposition \ref{constraint_HF} (4).

\subsection*{Acknowledgement}
The author thanks Dongtai He and Samuel Taylor for valuable comments.

\section{Monodromy entropy of fibered classes}\label{Sec-entropy_fibered}
Let $M$ be an oriented connected compact $3$--manifold. 
A cohomology class $\phi\in H^1(M;\Integral)$ 
is called a \emph{fibered class},
if $M$ admits a bundle structure fibering over a circle,
such that the fibers all represent the Poincar\'e dual of $\phi$.

More precisely, any \emph{fiber} $S\subset M$ 
(with respect to $\phi$) is an oriented, possibly disconnected, compact surface.
The orientation of $S$ is uniquely determined
the orientation of $M$ and a fixed orientation of the circle,
and the number of connected components is equal to the divisibility of $\phi$
(that is, 
the natural number by which $\phi$ is the multiple of a primitive cohomology class).

The \emph{monodromy} $f\colon S\to S$ 
(with respect to $\phi$) is an orientation-preserving self-homeomorphism, 
whose isotopy class depends only on $M$ and $S$.
It is characterized by the property that 
$M$ can be obtained from the oriented product manifold $S\times [0,1]$
by identifying the boundary component, $(x,1)\sim(f(x),0)$ for all $x\in S$,
with $S$ obtained as $S\times\{0\}$.
As $M$ is connected, $f$ acts transitively on the connected components
of $S$, and the components are all homeomorphic to each other.

Associated to any fibered class $\phi\in H^1(M;\Integral)$, 
there are two natural quantities.
The Thurston norm $\|\phi\|_{\mathtt{Th}}$ measures 
the topological complexity of $\phi$,
while the monodromy entropy $\mathrm{Ent}(\phi)$ measures 
the dynamical complexity of $\phi$.

In this preliminary section,
we review known facts about the Thurston norm and
the monodromy entropy.
We also discuss a sample theorem (Theorem \ref{entropy_value_set})
to motivate certain aspects of our main theorems.

\subsection{The Thurston norm}\label{Subsec-Thurston_norm}
Let $M$ be an oriented connected closed $3$--manifold. 
For any cohomology class $\phi\in H^1(M;\Integral)$,
the \emph{Thurston norm} $\|\phi\|_{\mathtt{Th}}$ of $\phi$ 
is defined as 
the minimum of the quantity $\chi_-(S)=\sum_i \max(-\chi(S_i),0)$,
where $S\subset M$ ranges over all the oriented closed subsurfaces
representing the Poincar\'e dual of $\phi$;
the connected components of $S$ are enumerated as $S_1,\cdots,S_k$,
and $\chi(S_i)$ denotes the Euler characteristic of each component.
With these values defined on the integral lattice
$H^1(M;\Integral)\subset H^1(M;\Real)$,
Thurston shows that 
they determine a unique seminorm on $H^1(M;\Real)$,
by first extending linearly over the rational points,
and then continuously over all the real points.
This is a norm (that is, nondegenerate)
if $M$ contains no nonseparating embedded tori or spheres.
The unit ball of the Thurston norm is the intersection
of finitely many half-spaces defined by linear inequalities
with rational coefficients,
or in other words, it is a rational polytope.
This polytope may be noncompact (in the degenerate case),
and must be symmetric about the origin.

If $\phi$ is a fibered class,
any fiber $S$ dual to $\phi$ is Thurston norm minimizing,
namely, $\|\phi\|_{\mathtt{Th}}=\chi_-(S)$.
Thurston shows that there are 
finitely many (possibly none) 
top-dimensional open faces
of the Thurston norm unit ball,
such that any integral cohomology class in the radial cones 
over these faces is a fibered class (excluding the origin),
and moreover,
any fibered class arises this way.
These open faces and cones are called 
the \emph{fibered faces} (of the Thurston norm unit ball)
and the \emph{fibered cones} (of the Thurston norm),
respectively.
They all depend only on the topology of $M$.
They emerge in $H^1(M;\Real)$ 
marking out all different ways for $M$
to fiber over a circle.

See Thurston \cite{Thurston_paper_norm} for the original introduction
and the aforementioned facts;
see also \cite[Chapter 5, Section 5.4.3]{AFW_book_group}
for a survey of results with many references.

\subsection{The monodromy entropy}\label{Subsec-monodromy_entropy}
For any fibered class $\phi\in H^1(M;\Integral)$,
remember that the monodromy $f\colon S\to S$ 
is only determined up to isotopy,
namely, as a mapping class $[f]\in\mathrm{Mod}(S)$.
Therefore, the \emph{monodromy entropy} for $(M,\phi)$ as we call
refers to the infimum of the topological entropy 
among all representatives of $[f]$.
In this paper,
we denote the monodromy entropy
for $(M,\phi)$ as $\mathrm{Ent}(\phi)$ or $\mathrm{Ent}([f])$.

The precise defining expression for the monodromy entropy
is not needed in the sequel.
We simply record it as follows, for the reader's convenience. 
\begin{equation}\label{entropy_def}
\mathrm{Ent}(\phi)=\mathrm{Ent}([f])
=\inf_f \sup_{\mathcal{U}} \varlimsup_{m\to\infty}
\frac{\log\left(\#\,\bigvee_{j=0}^{m-1}f^{-j}(\mathcal{U})\right)}{m}
\end{equation}
Here, 
$f$ ranges over all homeomorphic representatives of $[f]$;
$\mathcal{U}$ ranges over all finite open covers of $S$;
the notation $\bigvee_{j=0}^{m-1}f^{-j}(\mathcal{U})$
refers to the refined finite open cover
obtained by common intersections of (members of)
$\mathcal{U},f^{-1}(\mathcal{U}),\cdots,f^{-(m-1)}(\mathcal{U})$;
and $\#$ denotes the cardinality of a set.
See \cite[Section 2.1]{Kojima_entropy} for a brief review.

More useful to us is the following well-known characterization,
in terms of the Nielsen numbers $N(f^m)$ 
of the $m$--th iterates $f^m$, for all $m\in\Natural$:
\begin{equation}\label{entropy_def_by_characterization}
\mathrm{Ent}(\phi)=\mathrm{Ent}([f])=\lim_{m\to\infty} \frac{\log N(f^m)}{m}
\end{equation} 
In classical Nielsen theory, 
the \emph{Nielsen number} is a nonnegative integer-valued, homotopy invariant,
which can be defined for any self-map of a compact polyhedral complex.
We refer to Jiang's textbook \cite{Jiang_book} 
for a modern introduction to the general theory.

When $g\colon X\to X$ is a smooth self-map of a connected closed smooth manifold,
we say that $g$ has only non-degenerate fixed points,
if for any fixed point $p\in X$ of $g$,
the tangent map $\ud g|_p\in \mathrm{GL}(T_pX)$ does not have $1$ as an eigenvalue.
In this case,
the Nielsen number $N(g)$ can be described concretely as follows.

Note that the fixed points of $g$ is a finite subset $\mathrm{Fix}(g)\subset X$,
since non-degenerate fixed points are all isolated.
Two fixed points $p_0,p_1\in\mathrm{Fix}(X)$ are said to be \emph{Nielsen equivalent},
if there exists some path $\alpha:[0,1]\to X$ from $p_0$ and $p_1$,
such that $g\circ\alpha$ is homotopic to $\alpha$ relative to the endpoints.
The Nielsen equivalence classes in $\mathrm{Fix}(g)$ 
are called the \emph{fixed point classes} of $g$.
They form a quotient set of $\mathrm{Fix}(g)$, 
which we denote as $\mathscr{F}\mathrm{ix}(g)$.
For any fixed point $p\in\mathrm{Fix}(g)$,
the \emph{fixed point index} $\mathrm{ind}(g;p)\in\{-1,1\}$ is defined as 
the sign of the determinant of $(\mathbf{1}-\ud g)|_p\in\mathrm{End}(T_pX)$.
For any fixed point class $\mathbf{q}\in\mathscr{F}\mathrm{ix}(g)$,
the \emph{fixed point class index} $\mathrm{ind}(g;\mathbf{q})\in\Integral$
is defined as the sum of $\mathrm{ind}(g;p)$ over all $p\in\mathbf{q}$.
A fixed point class of nonzero index is called 
an \emph{essential} fixed point class.
With these notions, the Nielsen number of $g\colon X\to X$ is defined
as the number of the essential fixed point classes of $g$, namely,
\begin{equation}\label{Nielsen_number_def}
N(g)=\#\left\{\mathbf{q}\in\mathscr{F}\mathrm{ix}(g)\colon \mathrm{ind}(g;\mathbf{q})\neq0\right\}
\end{equation}
By classical Nielsen theory,
$N(g)$ is independent of the choice of the representative $g$
in its homotopy class; see \cite[Chapter I]{Jiang_book}.

Therefore, in (\ref{entropy_def_by_characterization}),
one may compute each $N(f^m)$ 
by choosing a diffeomorphic representative of $[f]^m\in\mathrm{Mod}(S)$,
with only non-degenerate fixed points.
In fact, one may obtain some generic representative $f$,
such that $f^m$ has only non-degenerate fixed points for all $m\in\Natural$.

\begin{proposition}\label{entropy_virtual}
	Let $M$ be an oriented connected compact $3$--manifold.
	The following statements hold for any fibered class $\phi\in H^1(M;\Integral)$.
	\begin{enumerate}
	\item For any $m\in\Natural$, 
	the multiple $m\phi\in H^1(M;\Integral)$ is fibered, and
	$$\mathrm{Ent}(m\phi)= \mathrm{Ent}(\phi)/m.$$
	\item For any connected finite cover $M'$ of $M$,
	the pullback $\phi'\in H^1(M';\Integral)$ is fibered, and
	$$\mathrm{Ent}(\phi')=\mathrm{Ent}(\phi).$$
	\end{enumerate}
\end{proposition}

The formulas in Proposition \ref{entropy_virtual}
are evident by (\ref{entropy_def_by_characterization}) and
our description of (\ref{Nielsen_number_def}).

\begin{example}\label{entropy_JSJ}
Let $[f]\in \mathrm{Mod}(S)$ be a mapping class of a connected closed orientable surface
of genus at least $2$.
\begin{enumerate}
\item If $[f]$ is pseudo-Anosov, there exist a pair of measured foliations
$(\mathscr{F}^{\mathtt{s}},\mu^{\mathtt{s}})$ and $(\mathscr{F}^{\mathtt{u}},\mu^{\mathtt{u}})$
on $S$, and a constant $\lambda>1$, 
such that some representative $f\colon S\to S$ has the property
$f_*(\mathscr{F}^{\mathtt{s}},\mu^{\mathtt{s}})=(\mathscr{F}^{\mathtt{s}},\lambda^{-1}\mu^{\mathtt{s}})$
and $f_*(\mathscr{F}^{\mathtt{u}},\mu^{\mathtt{u}})=(\mathscr{F}^{\mathtt{u}},\lambda\mu^{\mathtt{u}})$.
In this case, $f$ is called a \emph{pseudo-Anosov automorphism}
with \emph{stable/unstable measured foliations} $(\mathscr{F}^{\mathtt{s/u}},\mu^{\mathtt{s/u}})$ 
and \emph{stretching factor} $\lambda$.
Moreover, the following formula holds:
$$\mathrm{Ent}([f])=\log \lambda.$$
In fact, $\log\lambda$ is equal to the topological entropy of 
any pseudo-Anosov automorphism representative $f$,
which is unique up to conjugacy by isotopically trivial self-homeomorphisms of $S$.
\item
In general, $S$ can be decomposed along a collection of 
mutually disjoint, mutually non-parallel essential simple closed curves
into finitely many open components $S_1,\cdots,S_n$.
Moreover, for some representative $f\colon S\to S$ and some sufficiently divisible $k\in\Natural$,
the $k$--th iterate $f^k$ preserves each $S_i$, and the restriction of $f^k$ to each $S_i$
is isotopic to either the identity or a pseudo-Anosov automorphism of stretch factor $\lambda_i^k$.
This is essentially the content of the Nielsen--Thurston classification.
With these notations, 
$\mathrm{Ent}([f])$ is equal to the maximum among all $\log \lambda_i$,
or $0$ if there are no pseudo-Anosov components.
\end{enumerate}
See \cite[Section 2]{Kojima_entropy} and references therein for more detail;
see also \cite[Chapter 1, Section 1.10]{AFW_book_group} for the dictionary 
between the Nielsen--Thurston decomposition and the geometric decomposition.
\end{example}

\subsection{Monodromy entropy on fibered classes}\label{Subsec-value_set}
We conclude this preliminary section
with a sample theorem about
value distribution of the monodromy entropy,
as a function on the set of fibered classes for any fixed $3$--manifold
(Theorem \ref{entropy_value_set}).
Its conclusion should be well-known to experts,
and we sketch a proof for the reader's convenience.
For any fibered $3$--manifold,
Theorem \ref{entropy_value_set} 
implies a uniform upper bound of the monodromy entropy
for all fibered classes,
which depends on the topology of the $3$--manifold
in an inefficient way.

We mention this sample theorem,
so as to illustrate two helpful points
for understanding our main theorems.
First, Theorem \ref{entropy_value_set} suggests
that an upper bound of the monodromy entropy
might be independent of the Thurston norm of the fibered class.
This is indeed the case,
as confirmed by 
Theorems \ref{main_ent_vs_vol_general} and \ref{main_ent_vs_vol_arithmetic}.
Secondly,
our sample proof of Theorem \ref{entropy_value_set}
relies essentially on different flow structures associated to different fibered cones,
and on the fact that any fixed $3$--manifold has only finitely many fibered cones.
By contrast, the Heegaard diagram approach allows us to 
dispose all fibered classes simultaneously,
yielding an efficient uniform estimate as in Theorem \ref{main_ent_vs_hpl}.

Our proofs of the main theorems are logically independent of Theorem \ref{entropy_value_set},
so the reader may safely skip this part.
For the sake of generality,
we (exceptionally) allow nonempty boundary in Theorem \ref{entropy_value_set}.

\begin{theorem}\label{entropy_value_set}
Let $M$ be an oriented connected compact $3$--manifold
with empty or tori boundary. 
For any $\epsilon>0$, 
the monodromy entropy $\mathrm{Ent}(\phi)$ 
as a function on the set of fiber classes $\phi\in H^1(M;\Integral)$
take at most finitely many distinct values above $\epsilon$.
\end{theorem}

\begin{proof}
	We may assume that the set of fibered classes is nonempty and there are no disk or sphere fibers,
	for otherwise there is nothing to prove. 
	Then $M$ admits nontrivial JSJ decomposition into Seifert fibered pieces
	and hyperbolic pieces. We may further assume that there exists some hyperbolic piece.
	Since $\mathrm{Ent}(\phi)$ is the maximum of $\mathrm{Ent}(\phi_J)$,
	where $J$ ranges over the finitely many JSJ pieces of $M$,
	it suffices to show that for each $J$ that 
	the monodromy entropy of its fibered classes 
	takes at most finitely many values greater than $\epsilon$.
	This is trivial if $J$ is Seifert fibered, since $h$ is constant zero.
	See Kojima \cite{Kojima_entropy}; the bounded case is similar to
	the closed case as recalled in Example \ref{entropy_JSJ}.
	Therefore, it remains to argue for $J$ hyperbolic.
	
	Without loss of generality, we assume that $M$ is hyperbolic.	
	For each fibered face $F$ of the Thurston norm unit ball,
	we show that the monodromy entropy $\mathrm{Ent}(\phi)$ 
	takes at most finitely many possible values
	if $\phi$ ranges over the fibered classes in the fibered cone $\mathcal{C}_F$ over $F$.
	
	Since $\mathcal{C}_F$ has codimension zero in $H^1(M;\Real)$,
	there is a finite collection of rational vectors 
	$\vec{v}_1,\cdots,\vec{v}_b\in H^1(M;\Rational)$
	pointing along extreme rays on the closure of $\mathcal{C}_F$,
	and spanning $H^1(M;\Rational)$ over $\Rational$.
	Possibly after rational rescaling,
	we may assume that they span a sublattice $L$ of $H^1(M;\Rational)$ over $\Integral$,
	such that $2L$ contains $H^1(M;\Integral)$.
	It follows that the translated cone $\vec{v}+\mathcal{C}_F$ 
	of $\mathcal{C}_F$ by $\vec{v}=\vec{v}_1+\cdots+\vec{v}_b$ 
	still contains all the fibered classes $\mathcal{C}_F\cap H^1(M;\Integral)$.
	Fix such a vector $\vec{v}\in \mathcal{C}_F$.
	
	The function $\phi\mapsto 1/\mathrm{Ent}(\phi)$ extends 
	radial-linearly over $\mathcal{C}_F\cap H^1(M;\Rational)$
	and then continuously over $\mathcal{C}_F$.
	The resulting function $1/\mathrm{Ent}$ is strictly concave on $F$,
	and tends to $0$ as $\vec{x}\in F$ approaches $\partial F$.
	See Fried \cite[Theorems E and F]{Fried_flow};
	see also McMullen \cite[Corollary 5.4]{McMullen_polynomial} 
	for an alternative proof.	
	Denote by $\vec{w}\in F$ the unique maximal point of $1/\mathrm{Ent}$ on $F$.
	
	The strict concavity implies $1/\mathrm{Ent}(\vec{u})>r/\mathrm{Ent}(\vec{w})$
	for any $0<r<1$ and any $\vec{u}$	
	in the subregion $r\vec{w}+(1-r)F$ of $F$.
	One may observe this region 
	as the intersection of the cone $r\vec{w}+\mathcal{C}_F$ 
	with $F$.
	For any $x>0$,
	we obtain 
	$$1/\mathrm{Ent}(x\vec{u})=x/\mathrm{Ent}(\vec{u})>xr/\mathrm{Ent}(\vec{w}),$$
	or equivalently, 
	$$\mathrm{Ent}(x\vec{u})<\mathrm{Ent}(\vec{w})/xr,$$
	for any $\vec{u}$ in the intersection	of $r\vec{w}+\mathcal{C}_F$ with $F$.
	
	For any $\epsilon>0$, we can first choose some $0<r<1$ small enough,
	such that $r\vec{w}+\mathcal{C}_F$ contains $\vec{v}+\mathcal{C}_F$,
	then choose some $x>0$ large enough,
	such that 
	$$\mathrm{Ent}(\vec{w})/xr<\epsilon.$$
	Then, 
	we see that any fibered class $\phi\in H^1(M;\Integral)\cap \mathcal{C}_F$
	with monodromy $\mathrm{Ent}(\phi)\geq\epsilon$ 
	must have Thurston norm $\|\phi\|_{\mathtt{Th}}\leq x$.
	There are at most finitely many fibered classes in $\mathcal{C}_F$
	with Thurston norm bounded by $x$.
	They give rise to at most finitely many different values 
	of the monodromy entropy that are $\geq\epsilon$.
	This proves the asserted finiteness.
\end{proof}

\section{Heegaard presentation length}\label{Sec-Heegaard_pl}
In this section, we introduce the notion of Heegaard presentation length
for orientable connected closed $3$--manifolds.
We study its first properties and 
compare it with other complexities of $3$--manifolds.

\subsection{Presentation length and Heegaard diagrams}\label{Subsec-hpl}
In group theory, the presentation length is
a quantity measuring 
combinatorial complexity of a finitely presentable group.
For any finite presentation
$P= (u_1,\cdots,u_n; w_1,\cdots,w_m)$,
the \emph{length} of $P$ is defined to be
$$\ell(P)=\sum_{j=1}^{m} \max(0,|w_j|-2),$$
where each relator $w_j$ is a word 
in the alphabet $\{u_1^{\pm1},\cdots,u_n^{\pm1}\}$,
and $|w_j|$ denotes the word length of $w_j$.
The \emph{presentation length} of a finitely presentable group $G$
is thereby defined to be
$$\ell(G)=\min_P \ell(P),$$
where $P$ ranges over all finite presentations of $G$.
See Delzant \cite{Delzant_group};
see also \cite{Cooper_pl,Delzant_Potyagailo_complexity} 
for comparison with the hyperbolic volume for 
fundamental groups of finite-volume hyperbolic 3-manifolds.

For any connected closed orientable $3$--manifold of $M$,
a \emph{Heegaard surface} $\Sigma\subset M$ is a connected closed orientable
subsurface which 
bounds a pair of handlebodies $U_\alpha,U_\beta\subset M$
on different sides, namely, $\partial U_\alpha=\partial U_\beta=\Sigma$
and $U_\alpha\cap U_\beta=\Sigma$.
In each of the handlebodies $U_\alpha$ and $U_\beta$, 
choose a finite collection of 
mutually disjoint, properly embedded disks, 
such that cutting the handlebody along the disks 
yields a $3$--ball,
then the boundaries of these disks 
gives rise to a finite collection of mutually disjoint,
simple closed curves on $\Sigma$,
which we denote as $\boldsymbol{\alpha}$ and $\boldsymbol{\beta}$,
accordingly.
We usually fix an ordering of the $\boldsymbol{\alpha}$--curves
and the $\boldsymbol{\beta}$--curves,
denoting $\boldsymbol{\alpha}=(\alpha_1,\cdots,\alpha_g)$
and $\boldsymbol{\beta}=(\beta_1,\cdots,\beta_g)$,
where $g$ is the genus of $\Sigma$.
We also require the $\boldsymbol{\alpha}$--curves
intersect transversely with the $\boldsymbol{\beta}$--curves
on $\Sigma$ anywhere they intersect.
The triple 
$$(\Sigma,\boldsymbol{\alpha},\boldsymbol{\beta})$$
is called a \emph{Heegaard diagram} that presents $M$.

Every Heegaard diagram gives rise to a presentation of the fundamental group $\pi_1(M)$,
upon fixing a choice of orientations for the $\boldsymbol{\alpha}$--curves
and transverse orientations for the $\boldsymbol{\beta}$--curves.
The group presentation follows easily from the van Kampen theorem,
and can be read off explicitly as follows.

For each $\beta_i$, create a generator $u_i$;
for each $\alpha_j$, write down a relator $w_j$
by going once around $\alpha_j$ and recording the intersection pattern,
which is unique up to cyclic permutation.
For example, if $\alpha_j$ intersects the $\boldsymbol{\beta}$--curves
$\beta_{i_1},\cdots,\beta_{i_t}$ in order,
of signs $\epsilon_{i_1},\cdots,\epsilon_{i_t}\in\{+1,-1\}$ 
according as the orientation of $\alpha_j$ agree or disagree
with the transverse orientation of the $\boldsymbol{\beta}$--curves,
then $w_j$ is $u_{i_1}^{\epsilon_{i_1}}\cdots u_{i_t}^{\epsilon_{i_t}}$.
We refer to this presentation of $\pi_1(M)$ 
as the \emph{$U_\beta$--presentation}
associated to a Heegaard diagram $(\Sigma,\boldsymbol{\alpha},\boldsymbol{\beta})$ of $M$,
and denote as $(\mathbf{u}_{\beta};\mathbf{w}_{\alpha})$.
Therefore,
\begin{eqnarray*}
\pi_1(M)&\cong&\left\langle \mathbf{u}_{\beta}\colon \mathbf{w}_{\alpha}=\mathbf{1}\right\rangle\\
&=& \left\langle u_1,\cdots,u_g\colon w_1=\cdots=w_g=1\right\rangle.
\end{eqnarray*}

There is another presentation switching the roles
of the $\boldsymbol{\alpha}$--curves and the $\boldsymbol{\beta}$--curves,
namely, the \emph{$U_\alpha$--presentation}.
In this paper, 
we keep using the $U_\beta$--presentation unless otherwise mentioned,
and constructions are all adapted to the $U_\beta$--presentation.

\begin{definition}\label{def_Heegaard_pl}
	Let $M$ be a connected closed orientable $3$--manifold.
	The \emph{Heegaard presentation length} of $M$ is defined as
	the minimum length of presentations of $\pi_1(M)$
	that arise from Heegaard diagrams, namely,
	$$\ell_{\mathtt{He}}(M)=\min_{(\Sigma,\boldsymbol{\alpha},\boldsymbol{\beta})}\,\ell(\mathbf{u}_\beta;\mathbf{w}_\alpha),$$
	where $(\Sigma,\boldsymbol{\alpha},\boldsymbol{\beta})$
	ranges over all Heegaard diagrams that present $M$.
\end{definition}

\subsection{Simplifying Heegaard diagrams}
The estimate number $3$ in Lemma \ref{at_least_three_intersections} below
is remotely related to the coefficient $\log3$ in Theorem \ref{main_ent_vs_hpl}.

\begin{lemma}\label{at_least_three_intersections}
	Let $M$ be a connected, closed, orientable $3$--manifold.
	Suppose that $M$ does not contain
	any embedded projective plane or any embedded non-separating sphere.
	
	If $(\Sigma,\boldsymbol{\alpha},\boldsymbol{\beta})$
	is a Heegaard diagram presenting $M$ and achieving $\ell_{\mathtt{He}}(M)$,
	and if $(\Sigma,\boldsymbol{\alpha},\boldsymbol{\beta})$
	minimizes the genus of $\Sigma$ subject to the above property,
	then each $\boldsymbol{\alpha}$--curve contains 
	at least $3$ intersection points with the $\boldsymbol{\beta}$--curves.
	Moreover,
	the flipped Heegaard diagram 
	$(\Sigma,\boldsymbol{\beta},\boldsymbol{\alpha})$
	also satisfies the same property.
\end{lemma}

\begin{proof}
	We argue by ruling out all possibilities of fewer than $3$ intersection points.
	Denote by $M=U_\alpha\cup_\Sigma U_\beta$ the Heegaard splitting associated to 
	$(\Sigma,\boldsymbol{\alpha},\boldsymbol{\beta})$, as usual.
		
	Suppose that some $\boldsymbol{\alpha}$--curve, say $\alpha_1$,
	had empty intersection with the $\boldsymbol{\beta}$--curves.
	Then, $\alpha_1$ bounds disks simultaneously in both 
	$U_\alpha$ and $U_\beta$. 
	These disks make a non-separating sphere in $M$, 
	as $\alpha_i$ is non-separating on $\Sigma$.
	This contradicts the assumption on $M$.
	
	Similarly, there are no $\boldsymbol{\beta}$--curves
	having empty intersection with the $\boldsymbol{\alpha}$--curves.
	
	Suppose that some $\boldsymbol{\alpha}$--curve, say $\alpha_1$,
	had exactly one intersection point $p$ with the $\boldsymbol{\beta}$--curves.
	Without loss of generality, 
	denote by $\beta_1$ the unique $\boldsymbol{\beta}$--curve that intersects $\alpha_1$.
	Then, among the intersection points on $\beta_1$ other than $p$,
	there is some $q$ nearest to $p$, 
	namely, such that some subarc $[p,q]$ of $\beta_1$ joining $p$ and $q$ contains 
	no other intersection points.
	The $\boldsymbol{\alpha}$--curve $\alpha_i$ through $q$ is different from $\alpha_1$,
	so we can handle-slide $\alpha_i$ over $\alpha_1$ 
	by approaching along the path $[p,q]$, resulting in a new curve $\alpha'_i$ in place of $\alpha_i$.
	This replacement does not introduce new intersections, 
	and removes $q$ from the intersections.
	Similarly, we can remove all the other intersection points on $\beta_1$ one by one,
	leaving only the intersection number $p$.
	In the end, we obtain a new Heegaard diagram without introducing new intersections,
	in which $\beta_1$ and $\alpha_1$ are disjoint from all other curves,
	intersecting only at $p$.
	Destabilizing the pair $\alpha_1$ and $\beta_1$,
	the resulting Heegaard diagram still achieves the Heegaard presentation length,
	but genus has decreased by $1$
	This contradicts the genus minimality of $\Sigma$.
	
	In effect, the resulting Heegaard diagram of the above Heegaard moves 
	can also be obtained by erasing $\beta_1$, and surging on $\alpha_1$,
	(that is, cutting $\Sigma$ along $\alpha_1$ and filling up with two disks,
	and forgetting $\alpha_1$). 
	This observation makes the above procedure simpler to remember.
		
	We claim that 
	any $\boldsymbol{\beta}$--curve must also have at least $2$
	intersection points with the $\boldsymbol{\alpha}$--curves.
	In fact,
	we have shown that that every $\boldsymbol{\alpha}$--curve
	in $(\Sigma,\boldsymbol{\alpha},\boldsymbol{\beta})$ has at least $2$
	intersection points with the $\boldsymbol{\beta}$--curves.
	It follows that the Heegaard presentation length $\ell_{\mathtt{He}}(M)$
	is equal to the total number of intersections minus $2g$.
	If there were some $\boldsymbol{\beta}$--curves, say $\beta_1,\cdots,\beta_s$,
	each with exactly $1$ intersection point with the $\boldsymbol{\alpha}$--curves.
	First dealing with $\beta_1$,
	we may suppose that $\alpha_1$ intersects $\beta_1$
	at its unique intersection point $p$.
	Again, we can erase $\alpha_1$ and surge on $\beta_1$,
	then the genus of $\Sigma$ has decreased by $1$.
	However, 
	the total number of intersections has decreased by at least $2$,
	since $\alpha_1$ has at least $2$ intersection points,
	which have all gone.
	Repeat the same procedure with $\beta_2,\cdots,\beta_s$, one after another.
	In the end, the genus has decreased by $s>0$,
	whereas the total number of intersections have decreased by $t\geq 2s$.
	Therefore, 
	the resulting Heegaard diagram $(\Sigma',\boldsymbol{\alpha}',\boldsymbol{\beta}')$
	has either smaller presentation length (if $t>2s$),
	or the same presentation length but smaller genus (if $t=2s$),
	so we again reach a contradiction, proving the claim.
	
	Next, we show that
	there are no $\boldsymbol{\alpha}$--curves with exactly $2$ intersection points.
	To argue by contradiction,
	suppose that some $\boldsymbol{\alpha}$--curve, say $\alpha_1$,
	had exactly $2$ intersection point $p,q$ with the $\boldsymbol{\beta}$--curves.
	Then, $p,q$ either lie on one and the same $\boldsymbol{\beta}$--curve, say $\beta_1$,
	or lie on distinct two distinct $\boldsymbol{\beta}$--curves, say $\beta_1$ and $\beta_2$.
	
	If $p,q$ lie on distinct curves $\beta_1$ and $\beta_2$,
	we can handle-slide $\beta_2$ over $\beta_1$ along a path $[q,p]$ on $\alpha_1$,
	resulting in a new curve $\beta'_2$.
	Then the intersection number of $\beta'_2$ 
	becomes the sum of the intersection numbers of $\beta_1$ and $\beta_2$ minus $2$,
	and the intersection number of $\alpha_1$ becomes $1$.
	Erase $\beta_1$ and surge on $\alpha_1$.
	The resulting Heegaard diagram
	still achieves the Heegaard presentation length,
	as the old contribution from $\beta_1$ has been transferred to $\beta'_2$.
	However, the genus has decreased by $1$ due to destabilization.
	This contradicts the genus minimality of $\Sigma$.
	
	If $p,q$ both lie on $\beta_1$,
	this time, we consider the $U_\alpha$--presentation of $\pi_1(M)$.
	Namely,
	the generators $u_1,\cdots,u_g$ are 
	dual to the $\boldsymbol{\alpha}$--disks bounded by $\alpha_1,\cdots,\alpha_g$,
	and the relators $w_1,\cdots,w_g$ correspond to 
	the $\boldsymbol{\beta}$--curves $\beta_1,\cdots,\beta_g$.
	Then the relator $w_1$ takes the form $u_1^{\nu}$,
	where $\nu\in\{0,\pm2\}$, and any other relator $w_j$
	does not contain the letter $u_1$ or its inverse.
	It follows that $u_1$ generates a free factor of $\pi_1(M)$,
	which is either infinite cyclic (if $\nu=0$)
	or cyclic of order $2$ (if $\nu=\pm2$).
	By standard facts in $3$--manifold topology,
	$M$ has a connected summand,
	either homeomorphic to $S^1\times S^2$, or homeomorphic to 
	$P^3$.
	Then $M$ contains either a non-separating sphere, or a projective plane,
	contrary the assumption on $M$.	
	
	In summary, 
	subject to the condition of achieving the Heegaard presentation length,
	we have shown that any genus-minimizing Heegaard diagram
	must have at least $3$ intersection points on every $\boldsymbol{\alpha}$--curve.
	
	It remains to show that the flipped Heegaard diagram $(\Sigma,\boldsymbol{\beta},\boldsymbol{\alpha})$
	satisfies the same assumptions as $(\Sigma,\boldsymbol{\alpha},\boldsymbol{\beta})$ does.
	In fact, we have already seen that
	every $\boldsymbol{\beta}$--curve contains at least $2$ intersection points.
	It follows that 
	$(\Sigma,\boldsymbol{\beta},\boldsymbol{\alpha})$
	has the same presentation length as that of $(\Sigma,\boldsymbol{\alpha},\boldsymbol{\beta})$,
	both being the total number of intersections minus $2$ times the genus of $\Sigma$.
	Therefore, $(\Sigma,\boldsymbol{\beta},\boldsymbol{\alpha})$
	also presents $M$, achieves $\ell_{\mathtt{He}}(M)$, and minimizes the genus of $\Sigma$.
\end{proof}

\begin{corollary}\label{hpl_cover}
	Let $M$ be a connected, closed, orientable $3$--manifold.
	Suppose that $M$ does not contain
	any embedded projective plane or any embedded non-separating sphere.
	Then, for any connected finite cover $M'$ of $M$,
	$$\ell_{\mathtt{He}}(M')-1\leq [M':M]\cdot (\ell_{\mathtt{He}}(M)-1).$$
\end{corollary}

\begin{proof}
	Let $(\Sigma,\boldsymbol{\alpha},\boldsymbol{\beta})$
	be a Heegaard diagram as in the conclusion of Lemma \ref{at_least_three_intersections}.
	Denoting by $g$ the genus of $\Sigma$, and by $k$ the total number of intersection points
	between the $\boldsymbol{\alpha}$--curves and the $\boldsymbol{\beta}$--curves,
	we obtain $\ell_{\mathtt{He}}(M)=k-2g$ by definition.
	The preimage $\Sigma'$ in $M'$ 
	of the Heegaard surface $\Sigma$ in $M=U_\alpha\cup_\Sigma U_\beta$
	is again a Heegaard surface,
	since both $\pi_1(U_\alpha)$ and $\pi_1(U_\beta)$ surject $\pi_1(M)$.
	Denoting by $d=[M':M]$ the covering degree and $g$ the genus of $\Sigma$,
	the surface $\Sigma'$ has genus $g'=gd-d+1$,
	and each $\boldsymbol{\alpha}$--curve or $\boldsymbol{\beta}$--curve has $d$ lifts in $\Sigma'$.
	We can discard $d-1$ lifted $\boldsymbol{\alpha}$--curves 
	and $d-1$ lifted $\boldsymbol{\beta}$--curves,
	obtaining a Heegaard diagram $(\Sigma',\hat{\boldsymbol{\alpha}}',\hat{\boldsymbol{\beta}}')$
	for $M'$. 
	Denote by $L$ the length of the presentation arising from 
	$(\Sigma',\hat{\boldsymbol{\alpha}}',\hat{\boldsymbol{\beta}}')$.
	For any $\hat{\boldsymbol{\alpha}}'$--curve that contains
	at least $2$ intersection points with $\hat{\boldsymbol{\beta}}'$--curves,
	it contributes at most the same amount to $L$
	as its underlying $\boldsymbol{\alpha}$--curve contributes to $k-2g$,
	since intersections with those discarded lifted $\boldsymbol{\beta}$--curves
	no longer contribute.
	There might also be some $\hat{\boldsymbol{\alpha}}'$--curves that contain
	fewer than $2$ intersection points with $\hat{\boldsymbol{\beta}}'$--curves,
	but they only contribute $0$ to $L$,
	whereas their underlying $\boldsymbol{\alpha}$--curves contribute
	at least $3-2=1$ to $k-2g$ individually.
	Anyways, we obtain an inequality $L\leq d\cdot(k-2g)-(d-1)$,
	where 
	the subtracted $d-1$ comes from the discounted contribution 
	from the $d-1$ discarded lifted $\boldsymbol{\alpha}$--curves,
	each discounting at least $3-2=1$.
	This inequality can be rearranged into $L-1\leq d\cdot(k-2g-1)$.
	We obtain
	$$
	\ell_{\mathtt{He}}(M')-1\leq L-1 \leq	d\cdot(k-2g-1)
	= [M':M]\cdot( \ell_{\mathtt{He}}(M)-1),$$
	as desired.
\end{proof}

\subsection{Comparison with other complexities}\label{Subsec-complexities}
There is an obvious comparison
$$\ell_{\mathtt{He}}(M)\geq \ell(\pi_1(M)).$$
However, 
it seems unclear whether $\ell_{\mathtt{He}}(M)$
is bounded by a linear function of $\ell(\pi_1(M))$ from above.
There is a more satisfactory comparison 
between the Heegaard presentation length and the Matveev complexity.

In general, the Matveev complexity is a topological invariant
for compact $3$--manifolds. 
It takes non-negative integer values.
It is originally defined in terms of what is called spines,
which we refer to Matveev's treatise \cite{Matveev_book} 
for full detail.
For any connected closed irreducible $3$--manifold,
there is a convenient equivalent description,
which we recall as follows.

The Matveev complexity of a connected closed irreducible $3$--manifold $M$
is $0$ if and only if $M$ is homeomorphic to
a $3$--sphere $S^3$, or a projective $3$--space $P^3$,
or a lens space $L(3,1)$.
Otherwise, the Matveev complexity of $M$
is equal to the minimal number of tetrahedra,
such that $M$ can be obtained by gluing them
using affine isomorphisms pairing up all the faces.

According to the above description,
the Matveev complexity of $M$ is bounded by
the Kneser complexity of $M$,
that is, the minimal number of tetrahedra in need
to triangulate $M$ into a simplicial $3$--complex.
By considering the second barycentric subdivision,
one may also bound the Kneser complexity of $M$
by $(4!)^2=576$ times the Matveev complexity of $M$.
For more comparisons between other complexities of $3$--manifolds
that arise from various classical presentations,
we recommend Cha's informative exposition \cite{Cha_complexities}.

\begin{lemma}\label{compare_Matveev}
	Let $M$ be a connected closed orientable irreducible $3$--manifold.
	\begin{enumerate}
	\item
	If $M$ is homeomorphic to $S^3$ or $P^3$, then $\ell_{\mathtt{He}}(M)=0$.
	\item
	If $M$ is homeomorphic to $L(3,1)$, then $\ell_{\mathtt{He}}(M)=1$.
	\item
	In all other cases,
	$$\frac{c_{\mathtt{Ma}}(M)+3}2\leq \ell_{\mathtt{He}}(M)\leq c_{\mathtt{Ma}}(M)+1,$$
	where $c_{\mathtt{Ma}}$ denotes the Matveev complexity.
	\end{enumerate}
\end{lemma}

\begin{proof}
	The manifolds $S^3$ and $P^3$ admit Heegaard diagrams of genus $1$
	where the intersection number between the $\alpha$--curve intersects the $\beta$--curve 
	is at most $2$ points,
	so these Heegaard diagrams 
	achieve $\ell_{\mathtt{He}}(S^3)=0$ and $\ell_{\mathtt{He}}(P^3)=0$.
	
	The manifold $L(3,1)$ has a Heegaard diagram of genus $1$
	and intersection number $3$, implying $\ell_{\mathtt{He}}(L(3,1))\leq 1$.
	On the other hand, 
	any finitely presentable group $G$ with presentation length $\ell(G)=0$
	must be a free product of factors that are infinite cyclic or cyclic of order $2$,
	by simple observation.
	This implies
	$\ell_{\mathtt{He}}(L(3,1))\geq \ell(\pi_1(L(3,1))
	=\ell(\Integral/3\Integral)>0$. Hence, $\ell_{\mathtt{He}}(L(3,1))=1$.
	
	Below, we assume $M$ is a connected closed orientable irreducible $3$--manifold
	other than $S^3$, $P^3$, and $L(3,1)$.
	
	To prove the asserted upper bound, 
	suppose $c_{\mathtt{Ma}}(M)=t$, and realize $M$ by gluing $t$ tetrahedra
	by an affine face pairing. 
	As a pseudo-simplicial triangulation of $M$ 
	(into a $3$--demensional $\Delta$--complex),
	there are $4t/2=2t$ faces, so the dual $1$--skeleton 
	is a $4$--valent graph of $t$ vertices and $2t$ edges.
	A compact regular neighborhood of the dual $1$--skeleton is 
	a handlebody $U_\alpha$ of Euler characteristic $t-2t=-t$,
	and the complement in $M$ of its interior
	is a compact regular neighborhood of the $1$--skeleton of $M$,
	which is also a handlebody $U_\beta$ of the same Euler characteristic.
	Therefore, the boundary $\partial U_\alpha=\partial U_\beta=\Sigma$
	is a Heegaard surface of genus $t+1$.
	We may choose $t+1$ meridional disks in $U_\alpha$,
	each contained in a face 
	(and intersecting its dual edge transversely at a unique point),
	and choose $t+1$ meridional disks in $U_\beta$,
	each intersecting an edge transversely at a unique point
	(and contained in a dual face).
	By choosing the meridional disks without cutting 
	$U_\alpha$ or $U_\beta$ into disconnected components,
	we obtain a Heegaard diagram $(\Sigma,\boldsymbol{\alpha},\boldsymbol{\beta})$
	presenting $M$, where the $\boldsymbol{\alpha}$--curves and the $\boldsymbol{\beta}$--curves
	are the boundaries of the chosen disks.
	As the faces are triangular,
	each $\boldsymbol{\alpha}$--curve contains at most $3$ intersection points with the $\boldsymbol{\beta}$--curves.
	With this Heegaard diagram, 
	we obtain the asserted inequality
	$$\ell_{\mathtt{He}}(M)\leq (t+1)\cdot (3-2)=t+1=c_{\mathtt{Ma}}(M)+1.$$
	
	To prove the asserted lower bound, 
	suppose $M=U_\alpha\cup_\Sigma U_\beta$ where
	$(\Sigma,\boldsymbol{\alpha},\boldsymbol{\beta})$ 
	is a Heegaard diagram that achieves $\ell_{\mathtt{He}}(M)$.
	By Lemma \ref{at_least_three_intersections},
	we can require that every $\boldsymbol{\alpha}$--curve 
	intersects the $\boldsymbol{\beta}$--curves transversely 
	with at least $3$ intersection points.
	Note that $\ell_{\mathtt{He}}(M)\geq1$,
	If $\ell_{\mathtt{He}}(M)=0$,
	implying $\ell(\pi_1(M))=0$, 
	the connected closed orientable $3$--manifold $M$ can only be $S^3$ or $P^3$.
	As we have excluded these possibilities, 
	we observe $\ell_{\mathtt{He}}(M)\geq1$.
	In particular, $\Sigma$ is not a sphere.
	
	Associated to the Heegaard diagram $(\Sigma,\boldsymbol{\alpha},\boldsymbol{\beta})$,
	there is a handle decomposition of $M$,
	with a unique $0$--handle and a unique $3$--handle, 
	such that attaching the $1$--handles to the $0$--handle results in $U_\beta$,
	and the $2$--handles are attached along the $\boldsymbol{\alpha}$--curves.
	Collapsing the handles onto their cores
	yields a cell decomposition of $M$.
	Moreover, the intersection pattern of the $\boldsymbol{\alpha}$--curves
	and the $\boldsymbol{\beta}$--curves gives rise to polyhedral structures of the cells.
	To be precise, the $1$--skeleton of the cell complex is just a wedge of circles,
	dual to the decomposition of $U_\beta$ by $\boldsymbol{\beta}$--disks;
	each $2$--cell can be viewed as a copy of its corresponding $\boldsymbol{\alpha}$--disk,
	and as a polygonal disk,
	such that the vertices are the intersection points on bounding $\boldsymbol{\alpha}$--curve,
	and the edges are attached homeomorphically onto the $1$--cells;
	the $3$--cell can also be viewed as a polyhedral $3$--ball,
	whose boundary sphere is obtained by path-end compactification
	in the cell complex, 
	and hence endowed with a polyhedral disk decomposition 
	from the $2$--skeleton.
	Note that the last point makes use of the fact that $\Sigma$ is not a sphere,
	so there is at least one $2$--cell.
	
	With the above description,
	we can triangulate each $2$--cell by adding diagonals to the polygonal disks.
	Note that by our assumption on $(\Sigma,\boldsymbol{\alpha},\boldsymbol{\beta})$,
	these polygonal disks all have at least $3$ edges, and after triangulation,
	there are exactly $\ell_{\mathtt{He}}(M)$ triangular faces.
	It follows that the $3$--cell also inherits a pseudo-simplicial triangulation
	on its boundary sphere.
	
	Note that the number of triangular disks on the boundary sphere is 
	$2\cdot \ell_{\mathtt{He}}(M)\geq2$,
	since each $2$--cell appears with two lifted copies.
	Moreover, there cannot be only $2$ triangular disks,
	for otherwise $M$ would be homeomorphic to $L(3,1)$,
	which has been excluded.
	We infer that
	there must be at least $4$ triangular disks on the boundary sphere.
	In terms of the dual decomposition on the boundary sphere,
	there must be at least $4$ dual vertices,	each having valence exactly $3$,
	so the dual faces cannot all be bigons, by simple observation.
	We infer that	some vertex on the boundary sphere 
	must be contained in at least $3$ distinct triangular disks.
	
	With the above observation,
	we can further triangulate the $3$--cell
	by picking a vertex as above, 
	and constructing the tetrahedra	as cones 
	over the triangular disks that do not contain that vertex.
	The resulting pseudo-simplicial triangulation
	contains at most $2\cdot\ell_{\mathtt{He}}(M)-3$ tetrahedra.
	Therefore,
	we have constructed a pseudo-simplicial triangulation of $M$
	with at most $2\cdot\ell_{\mathtt{He}}(M)-3$ tetrahedra,
	which yields an upper bound for $c_{\mathtt{Ma}}(M)$.
	This implies the inequality
	$$\frac{c_{\mathtt{Ma}}(M)+3}2\leq \ell_{\mathtt{He}}(M),$$
	as asserted.
\end{proof}

\section{Heegaard Floer homology}\label{Sec-HF}
In this section, we review Heegaard Floer homology.
We summarize various versions of Heegaard Floer homology
for any closed connected oriented $3$--manifold
as introduced by Ozsv\'ath and Szab\'o \cite{OS_hf_invariants,OS_hf_properties}.
We collect important facts regarding constraints coming from
an embedded oriented connected closed subsurface.
We also review another variant of Heegaard Floer homology
defined using Heegaard diagrams with multiple points,
following Lee and Lipshitz \cite{Lee_Lipshitz_covering}.

\subsection{Package information}\label{Subsec-hf_package}
Let $M$ be a closed, connected, oriented $3$--manifold.
The plus, minus, and infinity versions
of the Heegaard Floer homology of $M$
are denoted as
$\mathrm{HF}^+(M)$, $\mathrm{HF}^-(M)$, 
and $\mathrm{HF}^\infty(M)$, respectively.
These are all $\Integral/2\Integral$--graded 
modules over a polynomial ring $\Integral[U]$,
where $U$ is a customary notation of the indeterminate.
Their isomorphism types as $\Integral[U]$--modules
depend only on the homeomorphism type of $M$,
and their (absolute) $\Integral/2\Integral$--gradings
can be fixed upon fixing a homological orientation of $M$
(that is, an orientation of $H_*(M;\Rational)$ as a $\Rational$--vector space).

Recall that the space of all $\mathrm{Spin}^{\mathtt{c}}$ structures on $M$
form an affine $H^2(M;\Integral)$, denoted as $\mathrm{Spin}^{\mathtt{c}}(M)$.
In other words, 
this is an additive coset of the form $\mathfrak{s}_0+H^2(M;\Integral)$,
fixing any $\mathfrak{s}_0\in\mathrm{Spin}^{\mathtt{c}}(M)$ for reference.
For each $\mathfrak{s}\in\mathrm{Spin}^{\mathtt{c}}(M)$,
there is a well-defined first Chern class 
$c_1(\mathfrak{s})\in H^2(M;\Integral)$,
such that
$c_1(\mathfrak{s}+\psi)=c_1(\mathfrak{s})+2\psi$ 
holds for all $\psi\in H^2(M;\Integral)$.
Therefore, the expression $\bar{\mathfrak{s}}=\mathfrak{s}-c_1(\mathfrak{s})$ 
determines an involution $\mathfrak{s}\mapsto \bar{\mathfrak{s}}$
on $\mathrm{Spin}^{\mathtt{c}}(M)$,
with the property $c_1(\bar{\mathfrak{s}})=-c_1(\mathfrak{s})$.
There are several equivalent models 
for representing a $\mathrm{Spin}^{\mathtt{c}}$ structure on $M$,
whose details are unnecessary for our discussion in this paper.
A model with homologous nowhere vanishing vector fields,
called Euler structures as originally introduced by Turaev,
is adopted in \cite[Section 2.6]{OS_hf_invariants}.

Each of the three flavors $\mathrm{HF}^\circ(M)$ splits 
canonically as a direct sum of $\Integral/2\Integral$--graded $\Integral[U]$--submodules
$$\mathrm{HF}^\circ(M)=\bigoplus_{\mathfrak{s}}\,\mathrm{HF}^\circ(M,\mathfrak{s})$$
where $\mathfrak{s}$ ranges over all the $\mathrm{Spin}^{\mathtt{c}}$ structures of $M$,
but there are only finitely many nonvanishing summands.
For each $\mathfrak{s}$, there is an exact triangle
of $\Integral[U]$--module homomorphisms
\begin{equation}\label{exact_triangle_mip}
\xymatrix{
\mathrm{HF}^-(M,\mathfrak{s}) \ar[r]^-{i_*} & 
\mathrm{HF}^\infty(M,\mathfrak{s}) \ar[d]^-{j_*} \\ 
& \mathrm{HF}^+(M,\mathfrak{s}) \ar[lu]^-{d_*}
}
\end{equation}
where $i_*$ and $j_*$ respect the $\Integral/2\Integral$--grading,
and $d_*$ switches the $\Integral/2\Integral$--grading.

In general, 
$\mathrm{HF}^-(M,\mathfrak{s})$ is always finitely generated
over $\Integral[U]$;
the action of $U$ on
$\mathrm{HF}^\infty(M,\mathfrak{s})$ is always invertible,
and $\mathrm{HF}^\infty(M,\mathfrak{s})$
is finitely generated over the Laurent polynomial ring $\Integral[U,U^{-1}]$;
the action of $U$ is nilpotent on any element of $\mathrm{HF}^+(M,\mathfrak{s})$,
however,
$\mathrm{HF}^+(M,\mathfrak{s})$ may be infinitely generated over $\Integral[U]$,
(only if $c_1(\mathfrak{s})\in H^2(M;\Integral)$ is torsion).
In any case, it is an easy exercise of commutative algebra
(provided the above facts) to show that 
the $\Integral[U]$--module 
$\mathrm{Im}(d_*)\cong\mathrm{HF}^+/\mathrm{Ker}(d_*)$,
often denoted as $\mathrm{HF}^+_{\mathtt{red}}(M,\mathfrak{s})$,
must be finitely generated over $\Integral$.
Moreover, 
whenever $\mathrm{HF}^+(M,\mathfrak{s})$ is finitely generated,
$d_*$ will be injective, identifying 
$\mathrm{HF}^+(M,\mathfrak{s})\cong
\mathrm{HF}^+_{\mathtt{red}}(M,\mathfrak{s})$ as 
relatively $\Integral/2\Integral$--graded,
finitely generated $\Integral$--modules.

The hat version of the Heegaard Floer homology $\widehat{\mathrm{HF}}(M)$
is a $\Integral/2\Integral$--graded module over $\Integral$,
which, again,
splits over $\mathrm{Spin}^{\mathtt{c}}$ structures $\mathfrak{s}$
as a direct sum of $\Integral/2\Integral$--graded $\Integral$--modules.
For each $\mathrm{s}$, there is an exact triangle
of $\Integral$--module homomorphisms
\begin{equation}\label{exact_triangle_pph}
\xymatrix{
\mathrm{HF}^+(M,\mathfrak{s}) \ar[r]^-{U} & 
\mathrm{HF}^+(M,\mathfrak{s}) \ar[d] \\ 
& \widehat{\mathrm{HF}}(M,\mathfrak{s}) \ar[lu]
}
\end{equation}
where the multiplication by $U$ 
and the downward arrow respect the $\Integral/2\Integral$--grading,
and the upper-leftward arrow switches the $\Integral/2\Integral$--grading.

In general, $\widehat{\mathrm{HF}}(M,\mathfrak{s})$ is always finitely generated over $\Integral$.
One may actually infer this from the above exact triangles,
by showing 
$\mathrm{Ker}(d_*)\subset \mathrm{Im}(U)$ and
$\mathrm{Ker}(d_*)\cap\mathrm{Ker}(U)=\mathrm{Im}(j_*\circ U^{-1}\circ i_*)$
in $\mathrm{HF}^+(M,\mathfrak{s})$.
Since $\mathrm{HF}^-(M,\mathfrak{s})$ is finitely generated over $\Integral[U]$,
and hence Noetherian,
one can recover $\widehat{\mathrm{HF}}(M,\mathfrak{s})$ as the composite
of three finitely generated $\Integral[U]$--modules
$\mathrm{Coker}(U)\cong\mathrm{Im}(d_*)\,/\,U\,\mathrm{Im}(d_*)$,
and $\mathrm{Ker}(U)\cap \mathrm{Ker}(d_*)$, and $d_*(\mathrm{Ker}(U))$,
all having trivial $U$ action.

All the aforementioned facts can be found in \cite{OS_hf_invariants};
in particular, see Section 4 therein for most of them.
For simplicity, we have omitted extra features such as 
natural invariance and finer grading.
See also \cite{OS_hf_properties} for more properties and applications.

\begin{example}\label{hf_examples}\,
\begin{enumerate}
\item
The standard oriented $3$--sphere $S^3$ has a standard homology orientation
and a unique $\mathrm{Spin}^{\mathtt{c}}$ structure.
There are isomorphisms of $\Integral/2\Integral$--graded $\Integral[U]$--modules
\begin{eqnarray*}
	\mathrm{HF}^+(S^3)&\cong&\Integral[U,U^{-1}]\,/\,U\,\Integral[U],\\
	\mathrm{HF}^\infty(S^3)&\cong&\Integral[U,U^{-1}],\\
	\mathrm{HF}^-(S^3)&\cong& U\,\Integral[U],
\end{eqnarray*}
all supported on the even grading.
Hence,
$$\widehat{\mathrm{HF}}(S^3)\cong\Integral.$$
\item 
The standard oriented product $3$--manifold $S^1\times S^2$ 
has a standard homology orientation
and a unique $\mathrm{Spin}^{\mathtt{c}}$ structure 
$\mathfrak{s}_0$ with $c_1(\mathfrak{s}_0)=0$.
There is an isomorphism of $\Integral/2\Integral$--graded $\Integral$--modules
$$\widehat{\mathrm{HF}}(S^1\times S^2)\cong 
H_*(S^1;\Integral),$$
supported on the direct summand at $\mathfrak{s}_0$.
\end{enumerate}
See \cite[Section 3.1]{OS_hf_properties}.
\end{example}

\subsection{Constraints from subsurfaces}

\begin{notation}\label{phi_grading_notation}
	For any $\phi\in H^1(M;\Integral)$ and any $j\in\Integral$,
	$$\mathrm{HF}^\circ(M,\phi,j)=
	\bigoplus_{\langle c_1(\mathfrak{s}),\mathrm{PD}(\phi)\rangle=2s} \mathrm{HF}^\circ(M,\mathfrak{s})$$
	where $\mathrm{HF}^\circ$ stands for any flavor of Heegaard Floer homology.
\end{notation}

\begin{proposition}\label{constraint_HF}
	Suppose the Poincar\'e dual of 
	$\phi\in H^1(M;\Integral)$
	is represented by a connected, closed, oriented, 
	embedded surface of genus $g\geq1$.
	\begin{enumerate}
	\item
	If $|j|>g-1$,
	$$\mathrm{HF}^+(M,\phi,j)=0.$$
	\item
	If $0<|j|\leq g-1$,
	$$U^{g-|j|}\,\mathrm{HF}^+(M,\phi,j)=0.$$
	\item
	Assuming $g\geq2$, 
	the following top term equality holds if and only if the surface is a fiber:
	$$\dim_\Rational\,\Rational\otimes_\Integral\mathrm{HF}^+(M,\phi,g-1)=1.$$
	\item
	Assuming $g\geq3$,
	the following next-to-top term inequality holds if the surface is a fiber:
	$$\dim_\Rational\,\Rational\otimes_\Integral\mathrm{HF}^+(M,\phi,g-2)\geq N(\phi),$$
	where $N(\phi)$ denotes the Nielsen number of the monodromy of $\phi$.
	\end{enumerate}
\end{proposition}

In Proposition \ref{constraint_HF}, 
the first statement is 
known as the adjunction inequality in Heegaard Floer homology,
asserting an upper bound for the Thurston norm $\|\phi\|_{\mathtt{Th}}$
by the topmost nonvanishing term \cite[Theorem 7.1]{OS_hf_properties};
see also \cite{OS_genus} for the stronger detection result.
The second statement is called 
the $U$--action generalization of the adjunction inequality,
due to Wu \cite[Theorem 1.2 and Remark 4.4]{Wu_u_action}.
The the third statement
is called the detection of fibering by the top term,
due to Ni \cite{Ni_hf_fibered}.

The inequality in fourth statement
can be interpreted as a sort of the Morse inequality,
once we identify the left-hand side
as the symplectic Floer homology
(also known as the fixed point Floer homology)
of the monodromy acting on the fiber.
This follows directly from several deep theorems
identifying various kinds of Floer homologies
for $3$--manifolds.
We expose the detail in 
Appendix \ref{Sec-next_to_top_appendix},
in order to avoid distraction.

\subsection{Multiply pointed Heegaard diagrams}\label{Subsec-mpHD}
Heegaard diagrams without any marked points
have been frequently used
as a way of presenting $3$--manifolds
that are connected, closed, and orientable.
Pointed Heegaard diagrams are 
the setup data for defining Heegaard Floer homology.
Strongly admissible pointed Heegaard diagrams 
with respect to individual $\mathrm{Spin}^{\mathtt{c}}$ structures
are required for constructing chain complexes
of the versions $\mathrm{HF}^+$ and $\mathrm{HF}^\infty$,
whereas a weakly admissible pointed Heegaard diagram 
can always be arranged for constructing chain complexes of
$\mathrm{HF}^-$ and $\widehat{\mathrm{HF}}$,
working simultaneously for all $\mathrm{Spin}^{\mathtt{c}}$ structures.
Multiply pointed Heegaard diagrams were originally considered
in order to define what is called link Floer homology \cite{OS_link}.
It was quickly noticed that
the more general setting is also more convenient 
when passing to finite covers \cite{Lee_Lipshitz_covering}.

We review Heegaard Floer homology with multiply pointed 
Heegaard diagrams, mostly following 
\cite[Section 2]{Lee_Lipshitz_covering}
and \cite[Section 4]{OS_link}.

\begin{definition}\label{def_mpHD}
For any integers $g\geq0$ and $l\geq1$, 
a (balanced, generic) \emph{$l$-pointed Heegaard diagram} of genus $g$
refers to a quadruple $(\Sigma,\boldsymbol{\alpha},\boldsymbol{\beta},\mathbf{z})$,
consisting of the following items.
\begin{itemize}
\item
A connected closed oriented surface $\Sigma$ of genus $g$;
\item
a $(g+l-1)$--tuple $\boldsymbol{\alpha}=(\alpha_1,\cdots,\alpha_{g+l-1})$
of mutually disjoint simple closed curves on $\Sigma$;
\item
a $(g+l-1)$--tuple $\boldsymbol{\beta}=(\beta_1,\cdots,\beta_{g+l-1})$
of curves similarly as above; and
\item
an $l$--tuple 
$\mathbf{z}=(z_1,\cdots,z_l)$
of points on $\Sigma$ 
off the $\boldsymbol{\alpha}$--curves and the $\boldsymbol{\beta}$--curves.
\end{itemize}
Moreover, these items are required to satisfy all the following conditions.
\begin{itemize}
\item The $\boldsymbol{\alpha}$--curves
decomposes $\Sigma$ into exactly $l$ planar surface components, 
such that each component contains one (and hence only one) point in $\mathbf{z}$.
\item The $\boldsymbol{\beta}$--curves 
decomposes $\Sigma$ with similar properties as above.
\item The $\boldsymbol{\alpha}$--curves and the $\boldsymbol{\beta}$--curves
intersect transversely.
\end{itemize}
A $1$--pointed Heegaard diagram is simply called a \emph{pointed Heegaard diagram}.
\end{definition}

Any $l$--pointed Heegaard diagram of genus $g$ 
presents a connected closed oriented $3$--manifold $M=U_\alpha\cup_\Sigma U_\beta$
topologically as follows.
By assumption, one obtains a handlebody $U_\alpha$ of genus $g$ by 
attaching $2$--handles to a copy of $\Sigma\times[0,1]$
on the $\Sigma\times\{1\}$ side along the $\boldsymbol{\alpha}$--curves,
and then filling up the resulting spherical boundaries with $3$--handles;
similarly, one obtains a handlebody $U_\beta$ using the $\boldsymbol{\beta}$ curves;
the $3$--manifold is obtained by gluing $U_\alpha$ and $U_\beta$
by identifying both of their boundaries with $\Sigma\times\{0\}=\Sigma$,
(oriented by orienting $U_\alpha$ with induced boundary orientation agreeing with
the orientation of $\Sigma$, and $U_\beta$ in the opposite way).

The transversality assumption makes sure 
the existence a Morse--Smale function $f$ on $M$
(with respect to an auxiliary choice of a smooth structure and a Riemannian metric),
whose upward gradient flow crosses the subsurface $\Sigma$ transversely.
Along the flow lines, 
the $\boldsymbol{\alpha}$--curves on $\Sigma$ come from
$g+l-1$ distinct index--$1$ critical points in $U_\alpha$,
and the $\boldsymbol{\beta}$--curves on $\Sigma$ go to
$g+l-1$ distinct index--$2$ critical points in $U_\beta$,
and the points in $\mathbf{z}$ come from $l$ distinct index--$0$ critical points in $U_\beta$
and go to $l$ distinct index--$3$ critical points in $U_\alpha$.
These are all the critical points.

Special Heegaard moves of $l$--pointed Heegaard diagrams
are like usual Heegaard moves, 
manipulating the curves, possibly changing the genus,
but keeping away from the marked points.
There are isotopies, handle slides, and (index 1--2) de/stabilization.
Special Heegaard moves does not change the resulting oriented $3$--manifold $M$.
They only change the topological structure of the flow 
away from the flow lines through the marked points,
and the crossing subsurface.
In the literature, 
there is another kind of move called index 0--3 de/stabilization,
which changes the number of marked points,
and is needed for certain completeness of such moves
(see \cite[Proposition 3.3]{OS_link}).
We do not need the last kind in this paper.

To define any version of the Heegaard Floer homology
associated to
an $l$--pointed Heegaard diagram $(\Sigma,\boldsymbol{\alpha},\boldsymbol{\beta},\mathbf{z})$
of genus $g$,
it is instructive to consider the $(g+l-1)$--fold symmetric product
$\mathrm{Sym}^{(g+l-1)}(\Sigma)$,
namely, the cartesian product $\Sigma^{g+l-1}$
quotient by permutations of the components.
Then the tori 
$\mathbb{T}_\alpha=\alpha_1\times\cdots\times\alpha_{g+l-1}$
and $\mathbb{T}_\beta=\beta_1\times\cdots\times\beta_{g+l-1}$
embed into $\mathrm{Sym}^{g+l-1}\,\Sigma$ under the projection of $\Sigma^{g+l-1}$,
and 
$\{z_1,\cdots,z_l\}\times\Sigma^{g+l-2}$
projects a union $\mathbb{V}_{z}$
of $l$ embedded copies of $\mathrm{Sym}^{g+l-2}(\Sigma)$
in $\mathrm{Sym}^{g+l-1}(\Sigma)$.

With the above picture in mind,
the definition of
$\widehat{\mathrm{HF}}(\Sigma,\boldsymbol{\alpha},\boldsymbol{\beta},\mathbf{z})$
fits into the framework of the Lagrangian Floer homology,
upon auxiliary choices of symplectic structures and almost complex structures
that do not affect the result.
As what it means,
the intersection points $\mathbb{T}_\alpha\cap\mathbb{T}_\beta$
should generate the chain complex 
$\widehat{\mathrm{CF}}(\Sigma,\boldsymbol{\alpha},\boldsymbol{\beta},\mathbf{z})$.
The boundary operator $\widehat{\partial}$
should count the pseudo-holomorphic Whitney disks
connecting pairs of generators without intersecting $\mathbb{V}_z$.
Moreover, suitable admissibility conditions should 
be responsible for keeping the counting results finite.

Every element $\mathbf{x}\in\mathbb{T}_\alpha\cap\mathbb{T}_\beta$ 
is explicitly a subset $\mathbf{x}\subset \Sigma$ 
consisting of $g+l-1$ distinct points,
such that each point lies in a distinct $\boldsymbol{\alpha}$--curve
and a distinct $\boldsymbol{\beta}$--curve.
All such elements $\mathbf{x}$ 
generate $\widehat{\mathrm{CF}}(\Sigma,\boldsymbol{\alpha},\boldsymbol{\beta},\mathbf{z})$
as a free $\Integral$--module.
We refer the reader to \cite[Section 4]{OS_link} 
for the precise construction of the boundary operator $\widehat{\partial}$,
since the detail is irrelevant to our application.
We only mention a commonly used sufficient condition 
for ensuring $\widehat{\partial}$ to be well-defined,
as Condition \ref{domain_condition} below.

\begin{definition}\label{def_domain}
Let $(\Sigma,\boldsymbol{\alpha},\boldsymbol{\beta},\mathbf{z})$
be an $l$--pointed Heegaard diagram of genus $g$.
Enumerate by $D_1,\cdots,D_m\subset\Sigma$ 
the closures of the components in $\Sigma$
complementary to the union of all the $\boldsymbol{\alpha}$--curves and $\boldsymbol{\beta}$--curves.
Any element of the abelian group freely generated by $D_1,\cdots,D_m$
is called a \emph{domain}
which we denote as a $\Integral$--linear combination
$$\mathcal{D}=\sum_{i=1}^m n_iD_i.$$
A \emph{positive domain} is a domain with all coefficients $n_i\geq0$.
A \emph{periodic domain} is a domain whose boundary forms
full $\boldsymbol{\alpha}$--curves and the $\boldsymbol{\beta}$--curves,
(namely, the polygonal boundary of the domain
as a $1$--chain in the union of the $\boldsymbol{\alpha}$--curves and the $\boldsymbol{\beta}$--curves
is equal to a sum of $1$--cycles). 
\end{definition}

\begin{condition}\label{domain_condition}
There are no nontrivial positive periodic domain supported away from the marked points.
Or equivalently, adopting the notations in Definition \ref{def_domain},
any periodic domain $\mathcal{P}=\sum_{i=1}^m n_iD_i$
with $n_i=0$ for every $D_i$ that contains some point in $\mathbf{z}$
has some coefficient $n_r>0$ and some coefficient $n_s<0$,
unless $\mathcal{P}=0$.
\end{condition}

For any $l$--pointed Heegaard diagram 
$(\Sigma,\boldsymbol{\alpha},\boldsymbol{\beta},\mathbf{z})$
satisfying Condition \ref{domain_condition},
we denote the induced homology of the chain complex as
$$\widehat{\mathrm{HF}}\left(\Sigma,\boldsymbol{\alpha},\boldsymbol{\beta},\mathbf{z}\right)
=H_*\left(\widehat{\mathrm{CF}}(\Sigma,\boldsymbol{\alpha},\boldsymbol{\beta},\mathbf{z}),\widehat{\partial}\right).$$
This is a finitely generate module over $\Integral$,
based our above description.
Moreover, every generator $\mathbf{x}\in\mathbb{T}_\alpha\cap\mathbb{T}_\beta$ 
of the chain complex specifies
a $\mathrm{Spin}^{\mathtt{c}}$--structure $s(\mathbf{x})$ of 
the presented $3$--manifold $M=U_\alpha\cup_\Sigma U_\beta$.
(A representative of $s(\mathbf{x})$
can be obtained as a nowhere vanishing vector field on $M$,
surgering the gradient field along the flow lines
through the $\mathbf{z}$--points and $\mathbf{x}$--points
and cancelling the critical points in pairs.)
Upon fixing a homology orientation of $M$,
$\mathbf{x}$ can also be assigned with an absolute $\Integral/2\Integral$--grading.
In fact, the boundary operator $\widehat{\partial}$ 
preserves each $\Integral$--submodule 
$\widehat{\mathrm{CF}}(\Sigma,\boldsymbol{\alpha},\boldsymbol{\beta},\mathbf{z};\mathfrak{s})$
as freely generated by 
all $\mathbf{x}\in\mathbb{T}_\alpha\cap\mathbb{T}_\beta$ with $s(\mathbf{x})=\mathfrak{s}$,
and switches the $\Integral/2\Integral$--grading.
Therefore, 
$\widehat{\mathrm{HF}}(\Sigma,\boldsymbol{\alpha},\boldsymbol{\beta},\mathbf{z})$
also carries a $\Integral/2\Integral$--grading and splits as a direct sum
of $\Integral/2\Integral$--graded submodules according to the $\mathrm{Spin}^{\mathtt{c}}$ structures.
It is also known
that special Heegaard moves induce chain homotopy equivalence
of the chain complexes respecting the extra structures,
so the isomorphism type of 
$\widehat{\mathrm{HF}}(\Sigma,\boldsymbol{\alpha},\boldsymbol{\beta},\mathbf{z})$
as a $\Integral/2\Integral$--graded $\Integral$--module
depends only on 
the orientation-preserving homeomorphism of $M$
and the fixed homology orientation.

\begin{proposition}\label{HF_hat_mpHD}
Let $(\Sigma,\boldsymbol{\alpha},\boldsymbol{\beta},\mathbf{z})$
be an $l$--pointed Heegaard diagram satisfying Condition \ref{domain_condition}.
Denote by $M=U_\alpha\cup_\Sigma U_\beta$ the connected, closed, oriented $3$--manifold 
presented by $(\Sigma,\boldsymbol{\alpha},\boldsymbol{\beta},\mathbf{z})$.
Then the following isomorphisms of $\Integral/2\Integral$--graded
$\Integral$--modules hold, upon fixing a homology orientation of $M$.
$$\widehat{\mathrm{HF}}\left(\Sigma,\boldsymbol{\alpha},\boldsymbol{\beta},\mathbf{z}\right)
\cong
\widehat{\mathrm{HF}}\left(M\#(S^1\times S^2)^{\#(l-1)}\right)
\cong
\widehat{\mathrm{HF}}(M)\otimes_\Integral H_*(S^1;\Integral)^{\otimes{(l-1)}}
$$
In particular,
$$
\mathrm{dim}_\Rational\,
\left(
\Rational\otimes_\Integral
\widehat{\mathrm{HF}}\left(\Sigma,\boldsymbol{\alpha},\boldsymbol{\beta},\mathbf{z}\right)
\right)
=
2^{l-1}\cdot
\mathrm{dim}_\Rational\,
\left(
\Rational\otimes_\Integral
\widehat{\mathrm{HF}}(M)
\right)
$$
\end{proposition}

See \cite[Theorem 4.5]{OS_link} (or \cite[Theorem 2.4]{Lee_Lipshitz_covering}) and Example \ref{hf_examples}.

\section{The winding trick}\label{Sec-winding}
The winding trick, invented by Ozsv\'ath and Szab\'o \cite[Section 5]{OS_hf_invariants},
is a general procedure
to convert any pointed Heegaard diagram into a weakly or strongly admissible one.
In this section, we establish an efficient version (Lemma \ref{winding_efficient})
of the winding trick with careful control on the number of new intersections,
such that the resulting pointed Heegaard diagram satisfies 
the periodic domain criterion for weak admissibility (Condition \ref{domain_condition}).

Our proof of Lemma \ref{winding_efficient} follows the same procedure
as used in the proof of \cite[Lemma 5.4]{OS_hf_invariants}.
We make an explicit estimate for each of the steps.
Most of the estimates are simply the first idea that one may come up with.
However, there is one tricky point of our controlling,
which matters to the proof of Theorem \ref{main_ent_vs_hpl},
and eventually, 
to Theorems \ref{main_ent_vs_vol_general} and \ref{main_ent_vs_vol_arithmetic}.
This appears in Lemma \ref{periodic_domain_estimate};
see Remark \ref{periodic_domain_estimate_remark} for a detailed comment.

We supply details of 
the verification for Condition \ref{domain_condition} (Lemma \ref{verification_domain_condition}),
which are omitted in \cite[Section 5]{OS_hf_invariants}.
These details are more important in our context,
as they demonstrate sufficiency of our efficient winding.

\begin{lemma}\label{winding_efficient}
	Let $(\Sigma,\boldsymbol{\alpha},\boldsymbol{\beta},z)$ be a pointed Heegaard diagram of genus $g$.
	Denote by $b$ the first Betti number of
	the presented $3$--manifold $M=U_\alpha\cup_\Sigma U_\beta$.
	Denote by $k$ the total number of intersections between
	the $\boldsymbol{\alpha}$--curves and the $\boldsymbol{\beta}$--curves,
	and by $o_\alpha$ and $o_\beta$ the numbers of 
	the $\boldsymbol{\alpha}$--curves and the $\boldsymbol{\beta}$--curves
	without any intersection points, respectively.
	
	Then,
	after a relabeling of the $\boldsymbol{\alpha}$--curves if necessary,
	there exist a pointed Heegaard diagram
	$(\Sigma,\tilde{\boldsymbol{\alpha}},\boldsymbol{\beta},z)$,
	obtained by isotopy of the $\boldsymbol{\alpha}$--curves 
	on $\Sigma$ keeping away from $z$,
	such that the following properties all hold.
	\begin{itemize}
	\item The pointed Heegaard diagram satisfies Condition \ref{domain_condition}.
	\item For each $i=1,\cdots,b$,
	the curve $\tilde{\alpha}_i$ remains the same as $\alpha_i$
	in nearby the intersection points of $\alpha_i$ 
	with the $\boldsymbol{\beta}$--curves,
	and 
	has at most $(k+o_\alpha)(k+o_\beta)\cdot b\cdot 2^{b+1}$ other intersection points
	with the $\boldsymbol{\beta}$--curves in total.
	\item For each $i=b+1,\cdots,g$,
	the curve $\tilde{\alpha}_i$ remains the same as $\alpha_i$.
	\end{itemize}
\end{lemma}

\begin{figure}[htb]
\centering
\resizebox{90mm}{!}{\includegraphics{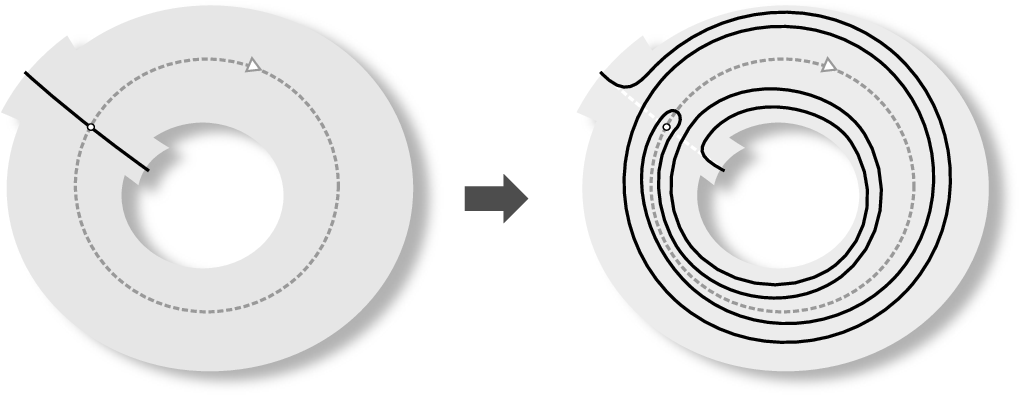}}
\caption{
An illustration of winding.
The picture on the left depicts a neighborhood of an oriented simple closed curve
on a surface, and another (unoriented) simple closed curve
that intersects the oriented curve transversely at a unique point,
only drawn as an arc nearby the intersection.
The picture on the right depicts the resulting arc
after winding along the oriented curve for $2$ full rounds (with a bit more),
starting at the point of intersection.
}\label{figWinding}
\end{figure}

The rest of this section is devoted to the proof of Lemma \ref{winding_efficient}.

We construct the asserted pointed Heegaard diagram $(\Sigma,\tilde{\boldsymbol{\alpha}},\boldsymbol{\beta},z)$
in Lemma \ref{winding_efficient} by winding a subset of the $\boldsymbol{\alpha}$--curves,
relabelled $\alpha_1,\cdots,\alpha_b$,
along disjoint oriented simple closed curves.
Each of the simple closed curves
intersects a unique $\alpha$--curve transversely at a unique point, 
missing the other $\boldsymbol{\alpha}$--curves,
and possibly intersecting $\boldsymbol{\beta}$--curves transversely.
(For each of $\alpha_1,\cdots,\alpha_b$,
we actually need two parallel simple closed curves, 
oriented in opposite directions,
in order to wind $\alpha_i$ simultaneously in two directions.)
In general, along any smooth path on a surface,
one may perform a ``finger move'',
that is,
``pushing'' the initial point along the path up to the terminal point
by a smooth isotopy supported in some small neighborhood of the path.
For our construction, we only need to push the unique intersection points 
along the oriented simple closed curves for some numbers of full rounds.
This is what we call ``winding'' of the $\boldsymbol{\alpha}$--curves;
see Figure \ref{figWinding} for an illustration.
The above sketchy description of winding 
should suffice for our exposition.
To prove Lemma \ref{winding_efficient},
the point is to construct the simple closed curves efficiently,
and control the numbers of winding rounds.

Let $(\Sigma,\boldsymbol{\alpha},\boldsymbol{\beta},z)$ be a pointed Heegaard diagram of genus $g$.
We say that a $g$--tuple $\boldsymbol{\gamma}=(\gamma_1,\cdots,\gamma_g)$
of mutually disjoint simple closed curves on $\Sigma$ is \emph{topologically dual} to $\boldsymbol{\alpha}$,
if every $\gamma_i$ intersects $\alpha_i$ transversely at exactly one point, 
and is disjoint from any other $\alpha_j$.
We also require $\gamma_i$ to keep away from $z$ and keep transverse to any $\boldsymbol{\beta}$--curve.

\begin{lemma}\label{dual_curves_estimate}
	Let $(\Sigma,\boldsymbol{\alpha},\boldsymbol{\beta},z)$ be a pointed Heegaard diagram of genus $g$.
	Denote by $k_\beta$ the total number of components
	as obtained by cutting the $\boldsymbol{\beta}$--curves by all $\boldsymbol{\alpha}$--curves.
	Then,
	there are mutually disjoint simple closed curves $\boldsymbol{\gamma}=(\gamma_1,\cdots,\gamma_g)$
	topologically dual to $\boldsymbol{\alpha}$,
	such that each $\gamma_s$ intersects 
	the $\boldsymbol{\beta}$--curves at no more than $k_\beta\cdot 2^{s-1}$ points in total.
\end{lemma}

\begin{proof}
	First cut $\Sigma$ along all the $\boldsymbol{\alpha}$--curves except $\alpha_1$,
	obtaining a connected, open surface 
	$\Sigma_1=\Sigma\setminus(\alpha_2\cup\cdots\cup\alpha_g)$.
	There is a cellular graph $\Gamma_1$ 
	dual to the combinatorial decomposition of $\Sigma_1$ by 
	the curve $\alpha_1$ together with all the $\boldsymbol{\beta}$--arcs.
	The vertices of $\Gamma_1$ are the components of
	$\Sigma\setminus(\alpha_1\cup\cdots\cup\alpha_g\cup\beta_1\cup\cdots\cup\beta_g)$,
	and the edges of $\Gamma_1$ are the components of
	$\alpha_1\setminus(\beta_1\cup\cdots\cup\beta_g)$
	and the components of all $\beta_j\setminus(\alpha_1\cup\cdots\cup\alpha_g)$,
	indicating adjacency relation between the vertices.
	We simply call the different kinds of edges $\alpha_1$--edges or $\boldsymbol{\beta}$--edges.
	
	Note that $\Gamma_1$ has at most $k_\beta$ $\boldsymbol{\beta}$--edges,
	and their union is a connected subgraph of $\Gamma_1$,
	since $\Sigma_1\setminus \alpha_1$ is connected.
	Therefore, we can find a simple path of $\boldsymbol{\beta}$--edges in $\Gamma_1$,
	such that it connects a pair of vertices that are the endpoints of some $\alpha_1$--edge.
	The union of that simple $\boldsymbol{\beta}$--edge path and that $\alpha_1$--edge 
	can be realized as a simple closed curve $\gamma_1$ in $\Sigma_1$,
	intersecting $\alpha_1$ transversely at exactly one point.
	By construction, $\gamma_1$ intersects the $\boldsymbol{\beta}$--curves transversely 
	in at most $k_\beta$ points.
	
	Next, we cutting $\Sigma$ along all the $\boldsymbol{\alpha}$--curves except $\alpha_2$,
	and also cut along $\gamma_1$, 
	obtaining $\Sigma_2=\Sigma\setminus(\alpha_1\cup\gamma_1\cup\alpha_3\cup\cdots\cup\alpha_g)$.
	By similar argument, we can find a simple closed curve $\gamma_2$ in $\Sigma_2$,
	intersecting $\alpha_2$ transversely exactly once.
	This time, 
	the number of $\boldsymbol{\beta}$--edges in the dual graph $\Gamma_2$
	is at most $k_\beta+k_\beta=2k_\beta$, so $\gamma_2$ intersects the $\boldsymbol{\beta}$--curves
	in at most $2k_\beta$ points.
	
	For the subsequent steps, we proceed similarly.
	For constructing $\gamma_s$, 
	we cut $\Sigma$ along all the $\boldsymbol{\alpha}$--curves except $\alpha_s$,
	and also cut along $\gamma_1,\cdots,\gamma_{s-1}$.
	Then the resulting $\gamma_s$ intersects $\alpha_s$ transversely,
	disjoint from any other $\boldsymbol{\alpha}$--curve or any constructed $\gamma_j$,
	and the $\boldsymbol{\beta}$--curves transversely in at most $k_\beta\cdot 2^{s-1}$ points.
	Finish while $s>g$.
\end{proof}

For any pointed Heegaard diagram $(\Sigma,\boldsymbol{\alpha},\boldsymbol{\beta},z)$ of genus $g$,
we denote the $\ell^\infty$ norm of any domain $\mathcal{D}=n_1 D_1+\cdots+ n_m D_m$ as
$\|\mathcal{D}\|_\infty=\max\{|n_1|,\cdots,|n_m|\}$.
If $\mathcal{P}$ is a periodic domain, its boundary takes the form
$$\partial\mathcal{P}=\partial_\alpha\mathcal{P}+\partial_\beta\mathcal{P},$$
where $\partial_\alpha\mathcal{P}=x_1\alpha_1+\cdots+\cdots+x_g\alpha_g$
and $\partial_\beta \mathcal{P}=y_1\beta_1+\cdots+y_g\beta_g$,
fixing auxiliary orientations of the $\boldsymbol{\alpha}$--curves and the $\boldsymbol{\beta}$--curves.
In this case,
we denote $\|\partial_\alpha\mathcal{P}\|_\infty=\max\{|x_1|,\cdots,|x_g|\}$
and $\|\partial_\beta\mathcal{P}\|_\infty=\max\{|y_1|,\cdots,|y_g|\}$.

\begin{lemma}\label{norm_comparison_periodic}
	Let $(\Sigma,\boldsymbol{\alpha},\boldsymbol{\beta},z)$ be a pointed Heegaard diagram of genus $g$.
	Denote by $k_\alpha$ the total number of components
	as obtained by cutting the $\boldsymbol{\alpha}$--curves by all $\boldsymbol{\beta}$--curves.
	If $\mathcal{P}$ is a periodic domain and if $\mathcal{P}$ has coefficient $0$
	at the region containing $z$, then
	$$(1/2)\cdot\|\partial_\beta\mathcal{P}\|_{\infty}\leq 
	\|\mathcal{P}\|_\infty\leq 
	k_\alpha\cdot\|\partial_\alpha\mathcal{P}\|_\infty.$$
\end{lemma}

\begin{proof}
	The inequality $(1/2)\cdot\|\partial_\beta\mathcal{P}\|_{\infty}\leq \|\mathcal{P}\|_\infty$
	follows immediately 
	from the fact that any $\boldsymbol{\beta}$--subarc is adjacent to at most $2$ regions in $\mathcal{P}$.
	
	To show the inequality $\|\mathcal{P}\|_\infty\leq 	k_\alpha\cdot\|\partial_\alpha\mathcal{P}\|_\infty$,
	consider the connected, open surface obtained by cutting $\Sigma$ along all the $\boldsymbol{\beta}$--curves,
	namely, $\Sigma_\beta=\Sigma\setminus(\beta_1\cup\cdots\cup\beta_g)$.
	The graph $\Gamma_\beta$ dual to the decomposition of $\Sigma_\beta$ by the $\alpha$--curves or arcs
	is connected, and has exactly $k_\alpha$ edges. 
	Starting from the vertex corresponding to the region containing $z$,
	we can reach any other vertex via a simple path of edges,
	which has length at most $k_\alpha$.
	As we cross from one region to another neighboring region,
	the coefficient difference of $\mathcal{P}$ 
	is plus or minus the coefficient of $\partial_\alpha\mathcal{P}$
	at the $\boldmath{\alpha}$--curve that is crossed.
	Therefore, the coefficient of $\mathcal{P}$ at any region is at most
	$k_\alpha\cdot\|\partial_\alpha\mathcal{P}\|_\infty$.
\end{proof}

\begin{lemma}\label{periodic_domain_estimate}
	Let $(\Sigma,\boldsymbol{\alpha},\boldsymbol{\beta},z)$ be a pointed Heegaard diagram of genus $g$.
	Denote by $b$ the first Betti number of
	the presented $3$--manifold $M=U_\alpha\cup_\Sigma U_\beta$.
	Then, 
	after a relabeling of the $\boldsymbol{\alpha}$--curves if necessary,
	there exist periodic domains $\mathcal{P}_1,\cdots,\mathcal{P}_b$,
	and some constant $R>0$,
	such that the following properties hold for all $\mathcal{P}_i$.
	\begin{itemize}
	\item The coefficient of $\mathcal{P}_i$
	at the region containing $z$ is $0$.
	\item
	$\|\partial_\alpha\mathcal{P}_i\|_\infty=R$.
	\item
	The coefficient of $\partial_\alpha\mathcal{P}_i$ at $\alpha_i$ 
	is equal to $\pm R$.
	\item 
	For each $j=1,\cdots,b$ other than $i$,
	the coefficient of $\partial_\alpha\mathcal{P}_i$ at $\alpha_j$ is $0$.
	\end{itemize}
	Indeed, 
	denoting by $k$ the total number of intersections between
	the $\boldsymbol{\alpha}$--curves and the $\boldsymbol{\beta}$--curves,	
	one may require	$0<R\leq (k/(g-b))^{g-b}$,
	unless $g=b$; in the exceptional case,
	one may replace the upper bound with $1$.
\end{lemma}

\begin{proof}
Note that any periodic domain $\mathcal{P}$ can be adjusted by adding
integral multiples of $\Sigma$ (that is, the sum of all the regions)
to make coefficient $0$ at the region containing $z$.
Subject to this normalization, 
two periodic domains $\mathcal{P}$ and $\mathcal{P}'$
are identical if and only if they have identical $\boldsymbol{\alpha}$--boundary.
For example, this is evident from Lemma \ref{norm_comparison_periodic},
applying to $\mathcal{P}-\mathcal{P}'$.

Moreover, a $\Integral$--linear combination $x_1\alpha_1+\cdots+x_g\alpha_g$
occurs as $\partial_\alpha\mathcal{P}$ for some periodic domain $\mathcal{P}$,
if and only if $x_1[\alpha_1]+\cdots+x_g[\alpha_g]=0$ holds in 
$H_1(U_\beta;\Integral)$.
In fact, the ``only if'' direction can be seen 
by capping $\mathcal{P}$ off with $\boldsymbol{\beta}$--disks
in the $\boldsymbol{\beta}$--handlebody $U_\beta$.
The ``if'' direction can be seen from the following topological recipe.
Take an immersed oriented compact surface in $U_\beta$ 
bounded by $x_1\alpha_1+\cdots+x_g\alpha_g$,
without passing through the index--$0$ critical point;
map the surface to the union of $\Sigma$ and the $\boldsymbol{\beta}$--disks
by flowing along the upward Morse trajectories;
then read off the coefficients of a solution $\mathcal{P}$ 
as the mapping degrees over individual regions.

The homological condition 
$x_1[\alpha_1]+\cdots+x_g[\alpha_g]=0$ in $H_1(U_\beta;\Integral)$ 
is equivalent to a linear system of equations 
$$AX=0,$$ 
where $X$ is the column vector transposing $(x_1,\cdots,x_g)$,
and where $A$ is the $g\times g$--matrix whose $(i,j)$--entry is the algebraic intersection number
$I([\beta_i],[\alpha_j])\in\Integral$ of $[\beta_j],[\alpha_i]\in H_1(\Sigma;\Integral)$.
The cokernel of $A$ is isomorphic to $H_1(M;\Integral)$,
so $A$ has rank $g-b$.

In the degenerate case $g=b$, the matrix $A$ is zero,
so we can take $\mathcal{P}_1,\cdots,\mathcal{P}_b$
of coefficient $0$ at the region containing $z$,
such that $\partial_\alpha\mathcal{P}_i=\alpha_i$ holds for each $\mathcal{P}_i$.
This obviously satisfies the asserted properties.

Below we assume $g>b$.
To find a fundamental set of solutions of the linear system of equations,
we may relabel the rows and the columns of $A$,
and assume the upper right $(g-b)\times (g-b)$--block $Q$ invertible.
Denote by $P$ the upper left $(g-b)\times b$--block of $A$,
so $A$ takes the block form
$$A=
\left(\begin{array}{c} A' \\ A''\end{array}\right)=
\left(\begin{array}{cc} P & Q \\ * & *\end{array}\right).$$
In addition, we assume the following maximality condition for $Q$:
\begin{itemize}
\item
Among all invertible $(g-b)\times (g-b)$--blocks $Q'$ in $A'$,
the  maximum of $|\det(Q')|$ is achieved by $|\det(Q)|>0$.
\end{itemize}

A set of fundamental solutions to $AX=0$ can be obtained as the column vectors
of the following $g\times b$--matrix,
$$S=\left(\begin{array}{c} S' \\ S''\end{array}\right),$$
where $S'=\mathrm{det}(Q)\cdot I$ is scalar of size $b\times b$, 
and $S''=-Q^*P=-\mathrm{det}(Q)\cdot Q^{-1}P$ is of size $(g-b)\times b$.
Note that the entries of $Q^*$ are $(g-b-1)\times(g-b-1)$--minors of $Q$,
so the entries of $S$ all lie in $\Integral$.
Observe that every intersection point between the $\boldsymbol{\alpha}$--curves
and the $\boldsymbol{\beta}$--curves contributes $\pm1$ 
exactly once to some entry of $A$.
It follows that the absolute values of the entries of $A$ is at most $k$ in total.
Denoting the block $Q$ as $(q_{ij})_{(g-b)\times(g-b)}$, we estimate
$$|\mathrm{det}(Q)|\leq \prod_{i=1}^{g-b} \sum_{j=1}^{g-b} |q_{ij}|
\leq \left(\frac{1}{g-b}\cdot\sum_{i=1}^{g-b}\sum_{j=1}^{g-b} |q_{ij}|\right)^{1/(g-b)}
\leq (k/(g-b))^{1/(g-b)}.$$
Any entry of $Q^*P$ is actually the determinant of 
another matrix obtained from $Q$ 
by replacing some column with a column of $P$ (by Cramer's rule),
which is again a $(g-b)\times(g-b)$--minor of $A'$, up to sign.
Therefore, entries of $Q^*P$ are all bounded
by $|\mathrm{det}(Q)|$ in absolute value (by our maximality assumption).
In particular,
for each column of $S$, the maximum among
the absolute values of the column entries is achieved
by the diagonal entry of the scalar matrix $S'=\mathrm{det}(Q)\cdot I$.

To summarize, the column vectors of $S$ represent
$b$ $\Integral$--linear combinations of the $\boldsymbol{\alpha}$--curves,
which are null-homologous in $U_\beta$.
Hence, there are $b$ periodic domains $\mathcal{P}_1,\cdots,\mathcal{P}_b$
whose $\boldsymbol{\alpha}$--boundaries are these combinations.
We can make their coefficients $0$ at the region containing $z$,
by adding integral multiples of $\Sigma$.
We take the asserted constant as $R=|\mathrm{det}(Q)|>0$,
satisfying the asserted upper bound.
We have relabeled the $\boldsymbol{\alpha}$--curves since we permuted the columns of $A$.
The three asserted properties regarding 
$\partial_\alpha\mathcal{P}$ follow from 
the maximality of $|\mathrm{det}(Q)|$, 
and the fact $S'=\mathrm{det}(Q)\cdot I$,
and the entry estimates of $S$, as we explained above.
\end{proof}

\begin{remark}\label{periodic_domain_estimate_remark}
The winding trick as in \cite[Lemma 5.4]{OS_hf_invariants} works for any invertible
$Q'$ without assuming the maximality condition as in the proof of Lemma \ref{periodic_domain_estimate}.
However, in that case,
one may have to wind too many times, 
in order to achieve Condition \ref{domain_condition}.
It would cause a replacement of the number of windings
$K$ in the sequel by something like $K'=(k+o_\alpha)\cdot (k/(g-b))^{g-b}\cdot b$.
Eventually,
this would lead to a worse upper bound for $\mathrm{Ent}(\phi)$
(comparable to $\ell_{\mathtt{He}}(M)\cdot\log(\ell_{\mathtt{He}}(M))$, for example),
which is insufficient for proving
Theorems \ref{main_ent_vs_vol_general} and \ref{main_ent_vs_vol_arithmetic}.
\end{remark}

With the above preparations, we prove Lemma \ref{winding_efficient} as follows.

Let $(\Sigma,\boldsymbol{\alpha},\boldsymbol{\beta},z)$ be a pointed Heegaard diagram of genus $g$.
Denote by $b$ the first Betti number of
the presented $3$--manifold $M=U_\alpha\cup_\Sigma U_\beta$.
If the $\boldsymbol{\alpha}$--curves intersects the $\boldsymbol{\beta}$--curves
at $k$ points in total, 
and if there are $o_\alpha$ $\boldsymbol{\alpha}$--curves 
that do not intersect with any $\boldsymbol{\beta}$--curves,
then there will be $k+o_\alpha$ components 
after cutting the $\boldsymbol{\alpha}$ by all $\boldsymbol{\beta}$--curves,
so the number $k_\alpha$ will be $k+o_\alpha$ 
when we apply Lemma \ref{norm_comparison_periodic}.
Similarly, $k_\beta$ will be $k+o_\beta$ when we apply Lemma \ref{dual_curves_estimate}.

We construct the asserted pointed Heegaard diagram
$$(\Sigma,\tilde{\boldsymbol{\alpha}},\boldsymbol{\beta},z)$$ 
as follows.

First,
we obtain $b$ periodic domains $\mathcal{P}_1,\cdots,\mathcal{P}_b$
as in Lemma \ref{periodic_domain_estimate}.
The procedure may have involved a relabeling of 
the $\boldsymbol{\alpha}$--curves.
Next, we obtain a $g$--tuple of mutually disjoint simple closed curves 
$\boldsymbol{\gamma}=(\gamma_1,\cdots,\gamma_g)$,
topologically dual the $\boldsymbol{\alpha}$--curves as in Lemma \ref{dual_curves_estimate}.
Only the first $b$ curves $\gamma_1,\cdots,\gamma_b$
will be useful to us,
so we remember that  
each of them (if $b>0$) intersects the $\boldsymbol{\beta}$--curves 
at no more than $k_\beta\cdot 2^{b-1}\leq(k+o_\beta)\cdot 2^{b-1}$ points in total.
For each $i=1,\cdots,b$ (if $b>0$),
we wind $\alpha_i$ along two nearby parallel copies of $\gamma_i$,
both running around $K$ times,
but in opposite directions, where we take
$$K=(k+o_\alpha)\cdot b.$$

To be precise (and to fix notations),
denote by $p_i\in \Sigma$ be the point where $\alpha_i$ and $\gamma_i$ intersect.
For each $\gamma_i$,
take a regular neighborhood of $\gamma_i$ 
disjoint form any other $\boldsymbol{\alpha}$--curves,
and parametrize as $\gamma_i\times[-1,+1]$,
such that 
$\gamma_i$ is the horizontal curve $\gamma_i\times\{0\}$,
and such that $\alpha_i$ and the $\boldsymbol{\beta}$--curves only intersect
the neighborhood in vertical arcs of the form $*\times[-1,\,+1]$.
Think of $\gamma_i$ as oriented from \underline{W}est to \underline{E}ast, 
and $[-1,1]$ from \underline{S}outh to \underline{N}orth.
Take two nearby points on $\gamma_i$,
named $p_i^{\mathrm{E}}$ and $p_i^{\mathrm{W}}$,
such that the short interval in $\gamma_i$ bounded by $p_i^{\mathrm{W}}$ and $p_i^{\mathrm{E}}$
contains $p_i$ at the center,
and such that no $\boldsymbol{\beta}$--curve crosses this interval;
name the points $p_i^{\mathrm{N}}=(p_i,\,+1)$, $p_i^{\mathrm{NE}}=(p_i^{\mathrm{E}},\,+1)$,
$p_i^{\mathrm{S}}=(p_i,\,-1)$, and $p_i^{\mathrm{SW}}=(p_i^{\mathrm{W}},\,-1)$;
name the curves $\gamma_i^{\mathrm{N}}=\gamma_i\times\{+1\}$
and $\gamma_i^{\mathrm{S}}=\gamma_i\times\{-1\}$.
We wind $\alpha_i$ along $\gamma_i^{\mathrm{N}}$,
starting from $p^{\mathrm{N}}_i$, 
running around $K$ times toward the $p^{\mathrm{NE}}_i$ direction,
and stopping a little ahead at $p^{\mathrm{NE}}_i$.
Similarly, 
we wind $\alpha_i$ along $\gamma_i^{\mathrm{S}}$ the same number of times,
from $p^{\mathrm{S}}_i$ toward $p^{\mathrm{SW}}_i$ and to $p^{\mathrm{SW}}_i$.
After all these windings $\alpha_1,\cdots,\alpha_b$, 
we obtain a new pointed Heegaard diagram 
$(\Sigma,\tilde{\boldsymbol{\alpha}},\boldsymbol{\beta},z)$.

The following two lemmas verify the asserted properties regarding
$(\Sigma,\tilde{\boldsymbol{\alpha}},\boldsymbol{\beta},z)$.

\begin{lemma}\label{verification_intersection}
In the pointed Heegaard diagram,
for each $i=1,\cdots,b$,
the curve $\tilde{\alpha}_i$ remains the same as $\alpha_i$
in nearby the intersection points of $\alpha_i$ 
with the $\boldsymbol{\beta}$--curve,
and 
has at most $(k+o_\alpha)(k+o_\beta)\cdot b\cdot 2^{b+1}$ other intersection points
with the $\boldsymbol{\beta}$--curves in total.
Moreover, $\tilde{\alpha}_i=\alpha_i$ for all $i=b+1,\cdots,g$.
\end{lemma}

\begin{proof}
The assertion $\tilde{\alpha}_i=\alpha_i$ for all $i=b+1,\cdots,g$
is obvious from the construction.

For each $i=1,\cdots,b$, 
$\gamma_i$ intersects the $\boldsymbol{\beta}$--curves 
at no more than $(k+o_\beta)\cdot 2^{b-1}$ points,
so the windings of $\alpha_i$ along $\gamma^{\mathrm{N}}_i$ and $\gamma^{\mathrm{S}}_i$
go across the $\boldsymbol{\beta}$--curves 
for at most $2K\cdot (k+o_\beta)\cdot 2^{b-1}=(k+o_\alpha)(k+o_\beta)\cdot b\cdot 2^b$ times altogether,
each time causing an increment of $2$ intersection points.
Therefore, there are at most $(k+o_\alpha)(k+o_\beta)\cdot b\cdot 2^{b+1}$
new intersection points with the $\boldsymbol{\beta}$--curves,
as we isotope $\alpha_i$ to its new position $\tilde{\alpha}_i$,
and the isotopy is supported away from the old intersection points of $\alpha_i$.
\end{proof}

\begin{lemma}\label{verification_domain_condition}
The pointed Heegaard diagram 
$(\Sigma,\tilde{\boldsymbol{\alpha}},\boldsymbol{\beta},z)$ satisfies Condition \ref{domain_condition}.
\end{lemma}

\begin{proof}
We need to understand 
the effect of the winding operation on periodic domains.
Every periodic domain $\mathcal{P}$ of $(\Sigma,\boldsymbol{\alpha},\boldsymbol{\beta},z)$ 
determines a periodic domain $\tilde{\mathcal{P}}$ of $(\Sigma,\tilde{\boldsymbol{\alpha}},\boldsymbol{\beta},z)$ 
by the rule that 
the coefficient of $\partial_\alpha \tilde{\mathcal{P}}$ at any $\tilde{\alpha}_i$
is equal to the coefficient of $\partial_\alpha \mathcal{P}$ at any $\alpha_i$,
for all $i=1,\cdots,g$.
During the winding isotopy, say along $\gamma^{\mathrm{N}}_i$,
every time as soon as $\alpha_i$ 
goes across a $\boldsymbol{\beta}$--curve, entering an existing region, 
a bigonal region is born.
If we keep track of a periodic domain $\mathcal{P}$ 
before and after the crossing,
$\mathcal{P}$ gains a new coefficient at the new born region,
which has to be that of the existing region
plus or minus the coefficient of $\partial_\alpha \mathcal{P}$ at $\alpha_i$.
The sign depends only the fixed orientation of $\alpha_i$ and the winding direction.
If it is plus for the winding along $\gamma^{\mathrm{N}}_i$,
then it is minus for the winding along $\gamma^{\mathrm{S}}_i$.
The pattern of regions and their coefficients in $\mathcal{P}$ stay the same
until $\alpha_i$ bumps into the next $\boldsymbol{\beta}$--curve along the winding.

From the above discussion, we see that the winding of $\alpha_i$ 
around $\gamma^{\mathrm{N}}_i$ for $K$ times has caused
a difference of the coefficient of $\tilde{\mathcal{P}}_i$ from that of $\mathcal{P}_i$,
at the region containing the point $p^{\mathrm{NE}}_i$,
by exactly $\pm (K+1)\cdot\|\partial_\alpha\mathcal{P}_i\|_\infty=\pm (K+1)R$.
So, the difference at the region containing $p^{\mathrm{SW}}_i$ is $\mp (K+1)R$,
due to the winding around $\gamma^{\mathrm{S}}_i$.
Meanwhile, for any $j=1,\cdots,b$ other than $i$,
the coefficients of $\tilde{\mathcal{P}}_j$ 
at the regions containing the points $p^{\mathrm{NE}}_i$ and $p^{\mathrm{SW}}_i$
stay invariant, since $\partial_\alpha\mathcal{P}_j$ has coefficient $0$ at $\alpha_i$.
(See Lemma \ref{periodic_domain_estimate}.)

For any periodic domain $\mathcal{P}\neq0$ with coefficient $0$ at the region containing $z$,
possibly after passing to a nonzero integral multiple,
we can write it as a unique $\Integral$--linear combination
$$\mathcal{P}=c_1\,\mathcal{P}_1+\cdots+c_b\,\mathcal{P}_b.$$
(This is because $[\partial_\alpha\mathcal{P}_1],\cdots,[\partial_\alpha\mathcal{P}_b]$
form a basis of $H_1(M;\Rational)$, and 
$[\partial_\alpha \mathcal{P}]\in H_1(M;\Rational)$ uniquely determines $\mathcal{P}$
subject to the coefficient $0$ condition at $z$,
as explained when we prove Lemma \ref{periodic_domain_estimate}.)

Suppose that the maximum among $|c_1|,\cdots,|c_b|$ is achieved by $|c_m|>0$.
By Lemmas \ref{norm_comparison_periodic} and \ref{periodic_domain_estimate},
we obtain
$\|\mathcal{P}_i\|_\infty\leq (k+o_\alpha)\cdot\|\partial_\alpha\mathcal{P}_i\|_\infty
=KR/b$. 
Therefore, the coefficients of $\mathcal{P}$
at the regions containing 
$p^{\mathrm{NE}}_m$ and $p^{\mathrm{SW}}_m$
are both bounded by 
$\|\mathcal{P}\|_\infty\leq
|c_m|\cdot(\|\mathcal{P}_1\|_\infty+\cdots+\|\mathcal{P}_b\|_\infty)
\leq |c_m|\cdot b \cdot KR/b=KR\cdot |c_m|$.
On the other hand,
the coefficient of $\tilde{\mathcal{P}}$ 
at the regions containing 
$p^{\mathrm{NE}}_m$ differs from that of $\mathcal{P}$
by exactly $\pm c_m\cdot(K+1)R$,
since the difference only comes from the $\mathcal{P}_m$ term.
At the region containing $p^{\mathrm{SW}}_m$,
the difference is the same amount of the opposite sign.

Therefore,
the coefficients of $\tilde{\mathcal{P}}$ 
at the regions containing 
$p^{\mathrm{NE}}_m$ and $p^{\mathrm{SW}}_m$
must be nonzero, and have opposite signs.

Every normalized nonzero periodic domain of $(\Sigma,\tilde{\boldsymbol{\alpha}},\boldsymbol{\beta},z)$
arises as $\tilde{\mathcal{P}}$ 
for some normalized nonzero periodic domain $\mathcal{P}$ of 
$(\Sigma,\boldsymbol{\alpha},\boldsymbol{\beta},z)$.
As $\mathcal{P}$ can be arbitrary,
we conclude that Condition \ref{domain_condition} holds for 
the pointed Heegaard diagram $(\Sigma,\tilde{\boldsymbol{\alpha}},\boldsymbol{\beta},z)$,
as asserted.
\end{proof}

By Lemmas \ref{verification_intersection} and \ref{verification_domain_condition},
the pointed Heegaard diagram $(\Sigma,\tilde{\boldsymbol{\alpha}},\boldsymbol{\beta},z)$
as we have constructed satisfies all the asserted properties 
as in the statement of Lemma \ref{winding_efficient}.

This completes the proof of Lemma \ref{winding_efficient}.

\section{Entropy versus Heegaard presentation length}\label{Sec-ent_vs_hpl}

This section is devoted to the proof of Theorem \ref{main_ent_vs_hpl}.

\begin{lemma}\label{virtual_domain_condition}
	Suppose that $(\Sigma,\boldsymbol{\alpha},\boldsymbol{\beta},\mathbf{z})$
	is an $l$--pointed Heegaard diagram of $g$,
	presenting a connected closed oriented $3$--manifold 
	$M=U_\alpha\cup_\Sigma U_\beta$.
	Then, 
	for any connected $d$--fold finite cover $M'$ of $M$,
	the preimages in $M$ of the objects
	$\Sigma,\boldsymbol{\alpha},\boldsymbol{\beta},\mathbf{z}$ in $M'$
	form an $ld$--pointed Heegaard diagram of genus $gd+d-1$,
	denoted as
	$(\Sigma',\boldsymbol{\alpha}',\boldsymbol{\beta}',\mathbf{z}')$.	
	
	Moreover,
	$(\Sigma',\boldsymbol{\alpha}',\boldsymbol{\beta}',\mathbf{z}')$
	satisfies Condition \ref{domain_condition}
	if and only if
	$(\Sigma,\boldsymbol{\alpha},\boldsymbol{\beta},\mathbf{z})$
	satisfies Condition \ref{domain_condition}.
\end{lemma}

\begin{proof}
	The preimage $\Sigma'$ of $\Sigma$ in $M'$ is again a Heegaard surface,
	(that is, connected and separating $M'$ into two handlebodies).
	This follows immediately from the fact that the inclusions of $U_\alpha$, $U_\beta$, and $\Sigma$
	into $M$ are all $\pi_1$--surjective. Moreover, the $\boldsymbol{\alpha}$--curves
	and their complementary planar subsurfaces in $\Sigma$ all lift, 
	since the $\boldsymbol{\alpha}$--curves all bound disks in $M$.
	A similar statement holds for the $\boldsymbol{\beta}$--curves.
	It follows that $(\Sigma',\boldsymbol{\alpha}',\boldsymbol{\beta}',\mathbf{z}')$
	is again a multiply pointed Heegaard diagram (see Definition \ref{def_mpHD}).
	The genus of the covering surface 
	follows from the Euler characteristic formula
	$\chi(\Sigma')=d\cdot\chi(\Sigma)$,
	and the number of marked points simply get multiplied by $d$.
	
	Every periodic domain 
	with respect to $(\Sigma',\boldsymbol{\alpha}',\boldsymbol{\beta}',\mathbf{z}')$
	pulls back to a periodic domain
	with respect to $(\Sigma,\boldsymbol{\alpha},\boldsymbol{\beta},\mathbf{z})$,
	by assigning the coefficient at region in the covering diagram 
	to be the coefficient of its projection image.
	In the other direction, every periodic domain upstairs 
	pushes forward to a periodic domain downstairs,
	by summing up the coefficients at all the lifts of each region.
	Therefore, 
	there is a nontrivial positive periodic domain upstairs
	missing all the marked points
	if and only if there is a periodic domain downstairs 
	with the same properties.
	In other words, 
	Condition \ref{domain_condition} either holds for 
	both $(\Sigma,\boldsymbol{\alpha},\boldsymbol{\beta},\mathbf{z})$ 
	and $(\Sigma',\boldsymbol{\alpha}',\boldsymbol{\beta}',\mathbf{z}')$,
	or fails for both of them.
\end{proof}

\begin{lemma}\label{entropy_estimate_with_b}
	Let $(\Sigma,\boldsymbol{\alpha},\boldsymbol{\beta})$ be a Heegaard diagram of genus $g$,
	presenting an orientable closed $3$--manifold $M=U_\alpha\cup_\Sigma U_\beta$.
	Denote by $k_i$ the number of intersection points
	of each curve $\alpha_i$ with the $\boldsymbol{\beta}$--curves in total.
	
	If $k_i\geq1$ holds for all $i=1,\cdots,g$,
	then, for any primitive fibered class $\phi\in H^1(M;\Integral)$
	of fiber genus $\geq3$,
	the following inequality holds.
	$$\mathrm{Ent}(\phi)\leq 
	\log\left(k_1\cdots k_g\right)+b\cdot\log\left(1+b\,2^{b+1} k^2\right)-\log 2,$$
	where we denote $k=k_1+\cdots+k_g$ and $b=\dim H_1(M;\Rational)$.
\end{lemma}

\begin{proof}
	Note that $M$ is irreducible 
	if has any fiber other than a sphere.
	In this case, 
	the condition $k_i\geq1$ automatically holds for all $i=1,\cdots,g$,
	because any $\boldsymbol{\alpha}$--curve without intersection points
	would give rise to a connected summand homeomorphic to $S^1\times S^2$.
	In other words, the numbers $o_\alpha$ and $o_\beta$ 
	will both be $0$, when we apply Lemma \ref{winding_efficient}.	
		
	Obtain a pointed Heegaard diagram $(\Sigma,\boldsymbol{\alpha},\boldsymbol{\beta},z)$
	by picking an arbitrary point $z$ on $\Sigma$ off the $\boldsymbol{\alpha}$--curves
	and the $\boldsymbol{\beta}$--curves.
	We apply Lemma \ref{winding_efficient}
	to obtain a new pointed Heegaard diagram 
	$(\Sigma,\tilde{\boldsymbol{\alpha}},\boldsymbol{\beta},z)$ for $M$,
	which satisfies Condition \ref{domain_condition};
	In particular, the numbers $\tilde{k}_1,\cdots,\tilde{k}_g$ of 
	intersection points on the curves $\tilde{\alpha}_1,\cdots,\tilde{\alpha}_g$ 
	satisfy 
	\begin{equation}\label{tilde_k_i_constraint}
	\begin{cases}
	\tilde{k}_i \leq k_i+b\,2^{b+1} k^2 & i=1,\cdots,b\\
	\tilde{k}_i = k_i & i=b+1,\cdots,g
	\end{cases}
	\end{equation}
	where $k=k_1+\cdots+k_g$ and $k_i\geq1$ for each $i=1,\cdots,g$.
			
	For any connected $d$--fold finite cover $M'$ of $M$,
	the covering $d$--pointed Heegaard diagram 
	$(\Sigma',\tilde{\boldsymbol{\alpha}}',\boldsymbol{\beta}',\mathbf{z}')$
	of genus $gd-d+1$
	also satisfies Condition \ref{domain_condition}
	(Lemma \ref{virtual_domain_condition}).
	Therefore,
	it is good for defining the chain complex
	$\widehat{\mathrm{CF}}(\Sigma',\tilde{\boldsymbol{\alpha}}',\boldsymbol{\beta}',\mathbf{z}')$.
	Note that the number of $\tilde{\boldsymbol{\alpha}}'$--curves is exactly $gd$,
	(obvious as each $\tilde{\boldsymbol{\alpha}}$--curve has $d$ exactly lifts).
	The chain complex
	$\widehat{\mathrm{CF}}(\Sigma',\tilde{\boldsymbol{\alpha}}',\boldsymbol{\beta}',\mathbf{z}')$
	is free generated over $\Integral$ 
	by all the $gd$--tuples $\mathbf{x}=(x_1,\cdots,x_{gd})$ of points on $\Sigma'$,
	such that each $x_i$ lies on 
	a distinct $\tilde{\boldsymbol{\alpha}}'$--curve and a distinct $\boldsymbol{\beta}'$--curve.
	Therefore, 
	the number of generators
	is at most the product $\tilde{k}'_1\cdots \tilde{k}'_{gd}$,
	where $\tilde{k}'_i$ denotes the number of intersection points on 
	the curve $\tilde{\alpha}'_i$.
	Since $\tilde{k}'_i$ depends only 
	on the $\tilde{\boldsymbol{\alpha}}$--curve on $\Sigma$
	which lifts to $\tilde{\alpha}'_i$,
	the product $\tilde{k}'_1\cdots \tilde{k}'_{gd}$ is equal to
	the product $(\tilde{k}_1\cdots \tilde{k}_g)^d$.
	Using the constraints 
	(\ref{tilde_k_i_constraint}) and the assumption $k_i\geq1$ for all $i=1,\cdots,g$,
	we estimate
	\begin{eqnarray*}
	\dim_\Rational 
	\left(\Rational\otimes_\Integral 
	\widehat{\mathrm{HF}}\left(\Sigma',\tilde{\boldsymbol{\alpha}}',\boldsymbol{\beta}',\mathbf{z}'\right)
	\right)
	&\leq&
	\dim_\Rational 
	\left(\Rational\otimes_\Integral 
	\widehat{\mathrm{CF}}\left(\Sigma',\tilde{\boldsymbol{\alpha}}',\boldsymbol{\beta}',\mathbf{z}'\right)
	\right)\\
	&=& \left(\tilde{k}_1\cdots\tilde{k}_g\right)^d\\
	&\leq& 
	\left(\prod_{i=1}^b(k_i+b\,2^{b+1}k^2)\cdot \prod_{i=b+1}^{g} k_i\right)^d\\
	&\leq& 
	\left(\prod_{i=1}^b \left(k_i\cdot(1+b\,2^{b+1}k^2)\right)\cdot \prod_{i=b+1}^{g} k_i\right)^d\\
	&=&
	\left(k_1\cdots k_g\right)^d\cdot\left(1+b\,2^{b+1}k^2\right)^{bd}\\
	\end{eqnarray*}
	By Proposition \ref{HF_hat_mpHD}, we obtain
	\begin{eqnarray*}
	\dim_\Rational 
	\left(\Rational\otimes_\Integral 
	\widehat{\mathrm{HF}}\left(M'\right)
	\right)
	&=&
	\frac{1}{2^{d-1}}\cdot\dim_\Rational 
	\left(\Rational\otimes_\Integral 
	\widehat{\mathrm{HF}}\left(\Sigma',\tilde{\boldsymbol{\alpha}}',\boldsymbol{\beta}',\mathbf{z}'\right)
	\right)\\
	&\leq&
	\left(k_1\cdots k_g\right)^{d}\cdot\left(1+b\,2^{b+1}k^2\right)^{bd}\cdot(1/2)^{d-1}.
	\end{eqnarray*}
	
	We apply the above estimate to the $m$--fold cyclic covers $M'_m$ of $M$
	dual to the given primitive fibered class $\phi\in H^1(M;\Integral)$,
	for all $m\in\Natural$.
	One may think of $M'_m$ as the mapping torus of the iterate $f^m$ of 
	the monodromy $f\colon S\to S$ associated to $(M,\phi)$.
	Denote by $\phi'_m\in H^1(M'_m;\Integral)$ the primitive class
	obtained as	the pull-back of $\phi$ divided by $m$,
	whose monodromy can be identified with $f^m\colon S\to S$.
	Then, we estimate the Nielsen number of $f^m$ for each $m\in\Natural$:
	\begin{eqnarray*}
	N(f^m)&\leq & 
	\dim_\Rational 
	\left(\Rational\otimes_\Integral 
	\mathrm{HF}^+\left(M',\phi'_m,\,\mathrm{genus}(S)-2\right)
	\right)\\
	&\leq&
	2\cdot\dim_\Rational 
	\left(\Rational\otimes_\Integral 
	\widehat{\mathrm{HF}}\left(M',\phi'_m,\,\mathrm{genus}(S)-2\right)
	\right)\\
	&\leq&
	2\cdot\dim_\Rational 
	\left(\Rational\otimes_\Integral 
	\widehat{\mathrm{HF}}\left(M'\right)
	\right)\\
	&\leq&
	2\cdot\left(k_1\cdots k_g\right)^{m}\cdot\left(1+b\,2^{b+1}k^2\right)^{bm}\cdot(1/2)^{m-1}\\
	&=&
	\left(k_1\cdots k_g\right)^{m}\cdot\left(1+b\,2^{b+1}k^2\right)^{bm}\cdot(1/2)^{m-2}
	\end{eqnarray*}
	Here,
	the first step is direct application
	of the next-to-top term estimate in Proposition \ref{constraint_HF} (4),
	valid under our fiber genus assumption on $\phi$.
	The second step is an easy consequence of 
	the $U$--action adjunction inequality in Proposition \ref{constraint_HF} (2).
	To be precise,
	$L=\mathrm{HF}^+(M',\phi'_m,\,\mathrm{genus}(S)-2)$
	is annihilated by $U^2$,
	according to the $U$--action adjunction inequality;
	the cokernel $L/UL$ injects 
	$\widehat{\mathrm{HF}}(M',\phi'_m,\,\mathrm{genus}(S)-2)$,
	by the exact triangle (\ref{exact_triangle_pph});
	the cokernel $L/UL$ surjects the image $UL=UL/U^2L$ under 
	the natural homomorphism $U\colon L/UL\to UL/U^2L$.
	These facts imply that the (free) rank of $L$, 
	equal to the rank sum of
	$L/UL$ and $UL$,
	is at most the rank of
	$\widehat{\mathrm{HF}}(M',\phi'_m,\,\mathrm{genus}(S)-2)$	times $2$,
	justifying the second step.
	The rest steps are obvious.
	
	Finally, we obtain the estimate for the monodromy entropy of $\phi$:
	\begin{eqnarray*}
	\mathrm{Ent}(\phi)&=&
	\lim_{m\to\infty}\,\frac{1}{m}\cdot\log N(f^m)\\
	&\leq&
	\lim_{m\to\infty}\,\frac{1}{m}\cdot
	\log\left\{\left(k_1\cdots k_g\right)^{m}\cdot\left(1+b\,2^{b+1}k^2\right)^{bm}\cdot(1/2)^{m-2}\right\}\\
	&=&
	\log\left(k_1\cdots k_g\right)+b\cdot\log\left(1+b\,2^{b+1}k^2\right)-\log 2,
	\end{eqnarray*}
	as desired.
\end{proof}

\begin{lemma}\label{entropy_estimate_without_b_fine}
	Let $(\Sigma,\boldsymbol{\alpha},\boldsymbol{\beta})$ be a Heegaard diagram of genus $g$,
	presenting an orientable closed $3$--manifold $M=U_\alpha\cup U_\beta$.
	Denote by $k_i$ the number of intersection points
	of each curve $\alpha_i$ with the $\boldsymbol{\beta}$--curves in total.
	
	If $k_i\geq1$ holds for all $i=1,\cdots,g$,
	then, for any primitive fibered class $\phi\in H^1(M;\Integral)$ of fiber genus $\geq3$,
	the following inequality holds.
	$$\mathrm{Ent}(\phi)\leq \log(k_1\cdots k_g)-\log k_{\mathrm{min}},$$
	where we denote $k_{\mathrm{min}}=\min(k_1,\cdots,k_g)$.
\end{lemma}

\begin{proof}
	For the same reason as
	in the proof of Lemma \ref{entropy_estimate_with_b},
	we may assume $M$ irreducible.	
	Obtain a pointed Heegaard diagram $(\Sigma,\boldsymbol{\alpha},\boldsymbol{\beta},z)$
	by fixing a base point.	
	
	For each $m\in \Natural$,
	denote by $M'_m$ the $m$--fold cyclic cover of $M$
	dual to the given primitive fibered class $\phi\in H^1(M;\Integral)$.
	Denote by $\phi'_m\in H^1(M'_m;\Integral)$ the primitive class
	obtained as	the pull-back of $\phi$ divided by $m$.	
	
	We obtain the covering $m$--pointed Heegaard diagram 
	$(\Sigma'_m,\boldsymbol{\alpha}'_m,\boldsymbol{\beta}'_m,\mathbf{z}'_m)$
	with respect to each $M'_m$.
	Since $\Sigma'_m$ has genus $gm-m+1$,
	we can obtain another pointed Heegaard diagram
	$(\Sigma'_m,\hat{\boldsymbol{\alpha}}'_m,\hat{\boldsymbol{\beta}}'_m,z'_m)$,
	by discarding $m-1$ curves from $\boldsymbol{\alpha}'_m$,
	and $m-1$ curves from $\boldsymbol{\beta}'_m$, and $m-1$ points from $\mathbf{z}'_m$,
	upon some suitable choice.
	The pointed Heegaard diagram 
	$(\Sigma'_m,\hat{\boldsymbol{\alpha}}'_m,\hat{\boldsymbol{\beta}}'_m,z'_m)$
	still presents $M'_m$.
	
	Every $M'_m$ is again irreducible.
	In particular,
	Lemma \ref{entropy_estimate_with_b}
	will still work for
	$(\Sigma'_m,\hat{\boldsymbol{\alpha}}'_m,\hat{\boldsymbol{\beta}}'_m,z'_m)$.
	Some $\widehat{\boldsymbol{\alpha}}'_m$--curves
	may have fewer intersections than their projection images do,
	and the intersection points on the discarded 
	$\widehat{\boldsymbol{\alpha}}'_m$--curves will no longer contribute.
	This makes $m\cdot\log(k_1\cdots k_g)-(m-1)\cdot\log k_{\mathrm{min}}$
	an upper bound for the first term on the right-hand side.
	Note also that the first Betti numbers of $M'_m$ are uniformly bounded,
	by some $B\geq1$ no more than $1$ plus the genus of the fiber surface associated to $(M,\phi)$.
	
	Applying Lemma \ref{entropy_estimate_with_b},
	we estimate the monodromy entropy of $\phi$:
	\begin{eqnarray*}
	\mathrm{Ent}(\phi)&=& \mathrm{Ent}(\phi'_m)/m\\
	&\leq&
	\frac{1}{m}\cdot
	\left\{
	m\cdot\log\left(k_1\cdots k_g\right)-(m-1)\cdot\log k_{\mathrm{min}}+
	O(\log m)
	\right\}.\\
	\end{eqnarray*}
	where the remainder term $O(\log m)$
	is explicitly
	$B\cdot\log(1+B\,2^{B+1}k^2m^2)-\log 2$,
	growing only logarithmically fast as $m$ tends to infinity.
	
	Passing to limit as $m$ tends to infinity,
	we obtain an improved inequality
	$$\mathrm{Ent}(\phi)\leq \log\left(k_1\cdots k_g\right)-\log k_{\mathrm{min}},$$
	as desired.
\end{proof}

\begin{corollary}\label{entropy_estimate_without_b_fine_corollary}
	In Lemma \ref{entropy_estimate_without_b_fine_corollary},
	assuming instead that $\phi$ is primitive fibered of fiber genus $2$,
	$$\mathrm{Ent}(\phi)\leq 2\cdot(\log(k_1\cdots k_g)-\log k_{\mathrm{min}}).$$
\end{corollary}

\begin{proof}
	In this case, identify $M$ as a mapping torus $M_f$ for some $[f]\in\mathrm{Mod}(S)$
	where $S$ has genus $2$. Take any connected double cover $\tilde{S}$ of $S$,
	then some iterate $f^m$ of $f$ admits a lift $[\tilde{f}]\in\mathrm{Mod}(\tilde{S})$.
	For example, one may take any $f^m$ 
	that induces the trivial automorphism on $H_1(S;\Integral/2\Integral)$.
	Denote by $M'$ the mapping torus $\tilde{f}$, which naturally covers $M$ of degree $2m$.
	Obtain a Heegaard diagram
	$(\Sigma',\hat{\boldsymbol{\alpha}}',\hat{\boldsymbol{\beta}}',z')$ for $M'$
	from $(\Sigma,\boldsymbol{\alpha},\boldsymbol{\beta})$,
	similarly as in the proof of Lemma \ref{entropy_estimate_without_b_fine_corollary}.
	Observe the pullback $\phi'\in H^1(M';\Integral)$ is dual to $m[\tilde{S}]\in H_2(M';\Integral)$,
	and $\tilde{S}$ has genus $3$.
	Then, by Lemma \ref{entropy_estimate_without_b_fine_corollary} and Proposition \ref{entropy_virtual}, we derive
	\begin{eqnarray*}
	\mathrm{Ent}(\phi)&=& \mathrm{Ent}(\phi')/m\\
	&\leq&
	\frac{1}{m}\cdot
	\left\{
	(2m\cdot\log\left(k_1\cdots k_g\right)-(m-1)\cdot\log k_{\mathrm{min}})-\log k_{\mathrm{min}}
	\right\}\\
	&=&
	2\cdot\left(\log\left(k_1\cdots k_g\right)-\log k_{\mathrm{min}}\right),
	\end{eqnarray*}
	as desired.
\end{proof}

\begin{lemma}\label{entropy_estimate_without_b}
	Adopt the same assumptions and notations
	as in Lemma \ref{entropy_estimate_without_b_fine}.
	If $k_i\geq3$ holds for all $i=1,\cdots,g$,
	then, for any primitive fibered class $\phi\in H^1(M;\Integral)$
	of fiber genus $\geq3$,
	the following inequality holds.
	$$\mathrm{Ent}(\phi)\leq (k-2g-1)\cdot\log 3,$$
	where we denote $k=k_1+\cdots+k_g$.
\end{lemma}

\begin{proof}
	Applying Lemma \ref{entropy_estimate_without_b_fine} with $k_{\mathrm{min}}\geq3$, 
	we estimate
	\begin{eqnarray*}
	\mathrm{Ent}(\phi)&\leq& \log\left(k_1\cdots k_g\right)-\log k_{\mathrm{min}}\\
	&\leq& g\cdot \log\left(\frac{k_1+\cdots+ k_g}g\right)-\log3\\
	&=& g\cdot \log\left(\frac{k}{g}\right)-\log3\\
	&=& (k-2g)\cdot F\left(\frac{g}{k-2g}\right)-\log3,
	\end{eqnarray*}
	where
	$$F(x)=x\cdot\log\left(2+x^{-1}\right).$$
	Observe $k-2g= (k_1-2)+\cdots+(k_g-2)\geq g$, 
	and hence, $0<g/(k-2g)\leq 1$.
	For $0<x\leq 1$, we can easily estimate the derivative
	$$F'(x)=\log(2+x^{-1})-\frac{x^{-1}}{2+x^{-1}}>\log3-1>0.$$
	This means $F(x)\leq F(1)=\log3$ for all $0<x\leq 1$.
	Therefore, we obtain
	$$\mathrm{Ent}(\phi)\leq (k-2g)\cdot\log 3-\log3=(k-2g-1)\cdot\log3,$$
	as desired.
\end{proof}

With the above preparation, we finish the proof of Theorem \ref{main_ent_vs_hpl}
as follows.

Let $M$ be a connected closed orientable $3$--manifold.
Suppose that $M'$ is any connected finite cover of $M$,
and $\phi'\in H^1(M';\Integral)$ is a primitive fibered class of fiber genus $\geq3$.
Hence, $M'$ and $M$ are both irreducible.

Obtain a Heegaard diagram $(\Sigma',\boldsymbol{\alpha}',\boldsymbol{\beta}')$
that presents $M'$ and achieves $\ell_{\mathtt{He}}(M')$,
such that each $\boldsymbol{\alpha}'$--curves contains
at least $3$ intersection points with the $\boldsymbol{\beta}'$--curves,
by Lemma \ref{at_least_three_intersections}.
In particular,
the Heegaard presentation length
$\ell_{\mathtt{He}}(M')$ is equal to $k'-2g'$,
where $k'$ denotes the total number of intersections and $g'$ denotes the genus of $\Sigma'$
(see Definition \ref{def_Heegaard_pl}).
By Lemma \ref{entropy_estimate_without_b} and Corollary \ref{hpl_cover},
we obtain the comparison
$$\mathrm{Ent}(\phi')\leq (\ell_{\mathtt{He}}(M')-1)\cdot\log3
\leq [M':M]\cdot(\ell_{\mathtt{He}}(M)-1)\cdot\log3,$$
as desired.

This completes the proof of Theorem \ref{main_ent_vs_hpl}.

\section{Volume versus Heegaard presentation length}\label{Sec-vol_vs_hpl}

This section is devoted to the proof of Theorem \ref{main_vol_vs_hpl}.

We recall the following basic formulas in hyperbolic geometry.
These formulas can be obtained by easy calculation in usual models of hyperbolic geometry.
We refer to Ratcliffe's textbook \cite{Ratcliffe_book}
for standard facts about hyperbolic manifolds.

\begin{formula}\label{tube_formula}
The (unoriented) isometric shape of 
any hyperbolic tube is uniquely determined by its depth $r\in(0,+\infty)$, 
and its systole $l\in(0,+\infty)$, and 
its (unoriented) monodromic angle $\varphi\in [0,\pi]$.
A model $\mathbf{V}(r,l,\varphi)$ of the hyperbolic tube
can be obtained as the distance--$r$ neighborhood
of a geodesic line in $\mathbb{H}^3$ quotient by a loxodromic isometry
translating along the geodesic line of distance $l$ and rotating about the geodesic line
of angle $\pm\varphi$.
Topologically, $\mathbf{V}(r,l,\varphi)$ is a compact $3$--manifold with boundary.
\begin{itemize}
\item
Hyperbolic tube volume: 
$$\mathrm{Vol}\left(\mathbf{V}(r,l,\varphi)\right)=\pi\cdot l \cdot\sinh^2(r)$$
\item
Hyperbolic tube wrist:
$$\mathrm{Wri}\left(\mathbf{V}(r,l,\varphi)\right)=2\pi\cdot\sinh(r)$$
\end{itemize}
\end{formula}

\begin{formula}\label{ball_formula}
The isometric shape of 
any hyperbolic ball is uniquely determined by its radius $r\in(0,+\infty)$.
A model $\mathbf{B}(r)$ of the hyperbolic ball
can be obtained as the distance--$r$ neighborhood
of a point in $\mathbb{H}^3$.
Topologically, $\mathbf{B}(r)$ is a compact $3$--manifold with boundary.
\begin{itemize}
\item
Hyperbolic ball volume: 
$$\mathrm{Vol}\left(\mathbf{B}(r)\right)=\pi\cdot(\sinh(2r)-2r)$$
\end{itemize}
\end{formula}

For some sufficiently small constant $\delta>0$ 
to be determined and depending only on $\epsilon$,
we construct as follows.

\emph{Step 1}. Take a maximal set $\mathscr{P}$ of distinct points in $W$,
subject to the following conditions: 
\begin{itemize}
\item The points in $\mathscr{P}$ are mutually apart of distance greater than $\delta$,
and are all apart from $\partial W$ of distance greater than $\epsilon/2$.
\end{itemize}
Construct the Dirichlet--Voronoi division of $M$ with respect to $\mathscr{P}$,
denoted as
\begin{equation}\label{first_division}
M=\frac{\bigsqcup_{p\in\mathscr{P}}\mathcal{D}(p)}{\mbox{side pairing}}.
\end{equation}
As what this means, 
for each $p\in\mathscr{P}$, 
there is an open, convex polyhedral $3$--cell 
$\mathrm{int}(\mathcal{D}(p))\subset M$,
which consists of all the points $q\in M$,
such that
$x=p$ is the unique point that minimizes 
the distance $\mathrm{dist}_M(q,x)$ among all $x\in\mathscr{P}$.
Let $\mathcal{D}(p)$ be the path-end compactification of $\mathrm{int}(\mathcal{D}(p))$,
which is an (abstract) compact, convex polyhedral $3$--cell
with boundary.
The boundary of $\mathcal{D}(p)$ comprises finitely many
totally geodesic, convex polygonal faces.
There is a unique characteristic map $\mathcal{D}(p)\to M$,
which is the extension by continuity of the inclusion of $\mathrm{int}(\mathcal{D}(p))$.
This way, $M$ is obtained naturally as the disjoint union of all $\mathcal{D}(p)$
by side pairing, which refers to the equivalence relation pairing up
(points on) faces of all $\mathcal{D}(p)$ with identical images in $M$
under the characteristic map.
 
\emph{Step 2}.
Possibly after generic small perturbation of $\mathscr{P}$
subject to the distance requirements,
we may assume that the cells in $M$ are all transverse to $\partial W$.
We obtain a decomposition of $W$ by truncating 
the interior of the tubes $V_i$ from the $3$--cells $\mathcal{D}(p)$.
Namely, we obtain
\begin{equation}\label{first_division_truncated}
W=\frac{\bigsqcup_{p\in\mathscr{P}} \mathcal{D}_W(p)}{\mbox{side pairing}},
\end{equation}
where we denote $\mathcal{D}_W(p)=\mathcal{D}(p)\cap W$, for all $p\in\mathscr{P}$.
Here and below, we often abuse the notation for intersection with abstract regions,
so, for instance, $\mathcal{D}(p)\cap W$ actually 
means the preimage of $W$ with respect to the characteristic map $\mathcal{D}(p)\to M$.

In general, the resulting regions $\mathcal{D}_W(p)$ 
and the patterns on $\partial\mathcal{D}_W(p)$ could be quite complicated.
For example, after the truncation,
some $\mathcal{D}(p)$ might become disconnected,
and some polygonal $2$--cells on $\partial\mathcal{D}_W(p)$ 
might become topological annuli,
rather than disks.
The following Lemmas \ref{estimate_visible} and \ref{cell_intersection}
help to rule out the bothering complication,
when $\delta$ is sufficiently small.

A \emph{cylindrical tube} $U$ in $\mathbb{H}^3$
refers to the distance $r$--neighborhood with boundary 
of a geodesic line $\gamma$, where we call $r$ the \emph{radius} of $U$
and $\gamma$ the \emph{axis} of $U$.
For any cylindrical tube $U\subset \mathbb{H}^3$, 
and any point $p\in\mathbb{H}^3$ not in $U$, 
we say that a point $q\in\partial U$ 
is \emph{visible} at $p$ 
if the geodesic segment $[p,q]\subset\mathbb{H}^3$ 
intersects $U$ only at $q$.

\begin{lemma}\label{estimate_visible}
	Suppose $0<\rho\leq1$. 
	Let $U\subset\mathbb{H}^3$ be a cylindrical tube of radius greater than $\rho$.
	If $p\in\mathbb{H}^3$ is at least distance $\rho/2$ from $U$,
	and $q\in\partial U$ is at most distance 
	$\rho\cdot\mathrm{arsinh}(1/\sqrt{3})\approx \rho\times 0.549$ from $p$,
	then $q$ is visible at $p$.
\end{lemma}

\begin{proof}
	We give a proof by means of elementary hyperbolic geometry, as follows.

	Let $p^*\in\partial V$ be the nearest point to $p$ on $\partial V$.
	Let $z\in V$ be the point on the geodesic ray $pp^*$ of distance $\rho/2$ to $p^*$.
	Denote by $\Sigma_z\subset\mathbb{H}^3$ the sphere centered at $z$ of radius $\rho/2$,
	so $\Sigma_z$	is contained in $V$ and tangent to $\partial V$ at $p^*$.
	Denote by $\Pi\subset\mathbb{H}^3$ the geodesic plane tangent to $\partial V$ at $p^*$.
	
	The geodesic lines passing through $p$ and tangent to $\Sigma_z$
	form a cone, which cut out a circle on $\Pi$.
	This circle lies on a unique sphere $\Sigma_p$ centered at $p$,
	and bounds a cap $D\subset\Sigma_p$ on the side of $\Pi$ that contains $z$
	Since $\Sigma_z$ is contained in $V$ and $\Pi$ separates $V$ from $p$,
	any geodesic ray emanating from $p$ and crossing the interior of $D$ 
	must intersects $\partial V$ two points, 
	one visible at $p$,	and the other invisible.
	Moreover, the invisible one lies beyond the hyperplane 
	passing through $z$ and perpendicular to $[p,z]$,
	so it lies outside the sphere $\Sigma_p$.
	Therefore, any point $q\in\partial V$ inside $\Sigma_p$
	is contained in the region confined by the cap $D$ and the plane $\Pi$,
	and is visible at $p$.	
	
	To complete the proof,
	it suffices to show that $\Sigma_p$ has radius $r_p\geq\rho\cdot\mathrm{arsinh}(1/\sqrt{3})$.
	Note that $p^*$, $\Sigma_z$, and $\Pi$ are fixed as $p$ varies on the geodesic ray $p^*p$,
	keeping at least distance $\rho/2$ from $V$,
	so $r_p$ is minimized when $[p,p^*]$ has length $\rho/2$.
	
	Below, we assume $\mathrm{Length}([p,p^*])=\rho/2$.
	In this case, consider any geodesic line passing through $p$ and tangent to $\Sigma_z$
	at a point $y$, and intersecting $\Sigma_p$ at a point $x$. 
	We observe that the geodesic segments
	$[p,p^*]$, $[p^*,z]$, and $[z,y]$ all have length $\rho/2$,
	and $[x,p^*]\perp [p,z]$, and $[z,y]\perp[x,y]$.
	Therefore,
	the geodesic triangles $\triangle(p,p^*,x)$, $\triangle(z,p^*,x)$, and $\triangle(z,y,x)$
	are all congruent, implying that the angle of $\triangle(p,p^*,x)$ at $x$ is $\pi/3$.
	By the hyperbolic sine law,
	$r_p=\mathrm{Length}([p,x])$ satisfies 
	$\sinh(r_p)/\sinh(\rho/2)=1/\sin(\pi/3)=2/\sqrt{3}$.
	Using the assumption $\rho\leq1$,
	we obtain
	$\sinh(r_p)=\sinh(\rho/2)\cdot2/\sqrt{3}\geq \rho/\sqrt{3}$,
	so $r_p\geq \mathrm{arsinh}(\rho/\sqrt{3})\geq \rho\cdot\mathrm{arsinh}(1/\sqrt{3})$,
	as asserted.
\end{proof}

\begin{lemma}\label{cell_intersection}
	Suppose $0<\epsilon\leq1$ and $0<\delta\leq\epsilon\times 10^{-1}$.
	With the notations $M,W,V_i,\mathscr{P},\mathcal{D}(p)$ as above,
	the following statements hold for any $p\in\mathscr{P}$
	where $\mathcal{D}(p)$ is not contained in $W$.
	\begin{enumerate}
	\item
	The intersection $\mathcal{D}(p)\cap\partial W$ is a topological disk
	properly embedded in $\mathcal{D}(p)$.
	\item
	The intersection $\partial \mathcal{D}(p)\cap\partial W$
	is a simple closed curve on $\partial \mathcal{D}(p)$
	that subdivides the polygonal $2$--cells into polygonal $2$--cells.
	\end{enumerate}
\end{lemma}

\begin{proof}
	If $\mathcal{D}(p)$ is not contained in $W$,
	$\mathcal{D}_W(p)$ must lie within distance $2\delta+(\epsilon/2)$ from $p$.
	This is because any point in $W$ of distance $>2\delta$ from $p$ and
	of distance $>\epsilon/2$ to $\partial W$
	is strictly closer to some different point in $\mathscr{P}$.
	Passing to the universal cover $\mathbb{H}^3$ of $M$
	and a lift $\mathcal{D}(\tilde{p})$ of $\mathcal{D}(p)$,
	there is a unique preimage component $\tilde{V}_i$ of $V_i$ that intersects 
	$\mathrm{int}(\mathcal{D}(\tilde{p}))$.
	Therefore, $\mathcal{D}(p)$ lifts homeomorphically to
	$\mathcal{D}(\tilde{p})$,
	so
	$\mathcal{D}(p)\cap\partial W=
	\mathcal{D}(p)\cap\partial V_i$ 
	lifts to $\mathcal{D}(\tilde{p})\cap \partial \tilde{V}_i$.
	
	Note that both $\tilde{V}_i$ and $\mathcal{D}(\tilde{p})$ are convex.
	By transversality,
	the intersection $\mathcal{D}(\tilde{p})\cap \tilde{V}_i$ 
	is a convex topological ball with boundary.
	The intersection $\partial \mathcal{D}(\tilde{p})\cap \partial \tilde{V}_i$
	is an embedded closed $1$--submanifold on 
	$\partial (\mathcal{D}(\tilde{p})\cap \tilde{V}_i)$,
	separating the topological sphere boundary
	into two planar parts, namely,
	the $\tilde{V}_i$--part boundary $\mathcal{D}(\tilde{p})\cap \partial \tilde{V}_i$
	and the $\mathcal{D}(\tilde{p})$--part boundary
	$\partial\mathcal{D}(\tilde{p})\cap\tilde{V}_i$.
	If either one of them is contractible, then both parts are topological disks.
	
	With this picture in mind,
	the following claim implies the first assertion:
	\textit{The intersection $\mathcal{D}(\tilde{p})\cap \partial \tilde{V}_i$
	is contractible.}	
	
	To this end, fix any point $\tilde{x}\in
	\mathcal{D}(\tilde{p})\cap \partial\tilde{V}_i$.
	For any $\tilde{y}\in\mathcal{D}(\tilde{p})\cap \partial\tilde{V}_i$,
	the geodesic segment $[\tilde{x},\tilde{y}]$ in $\mathbb{H}^3$
	is contained in $\tilde{V}_i$ and also in $\mathcal{D}(\tilde{p})$,
	by convexity.
	Since both $\tilde{x}$ and $\tilde{y}$ lie within distance 
	$2\delta+(\epsilon/2)<\epsilon\cdot\mathrm{arsinh}(1/\sqrt{3})$ from $\tilde{p}$,
	and $\tilde{p}$ lies at least distance $\epsilon$ from $\partial\tilde{V}_i$,
	both $\tilde{x}$ and $\tilde{y}$ are visible from $\tilde{p}$ (Lemma \ref{estimate_visible}).
	For any constant $t\in[0,1]$, 
	denote by $\tilde{y}_t\in[\tilde{x}_0,\tilde{y}]$ 
	the point that divides $[\tilde{x},\tilde{y}]$ into subsegments $[\tilde{x},\tilde{y}_t]$
	and $[\tilde{y}_t,\tilde{y}]$ of length ratio $t:(1-t)$.
	Then the segment $[\tilde{y}_t,\tilde{p}]$ intersects $\partial\tilde{V}_i$
	at a unique point $v(\tilde{y},t)$.
	The visibility of $\tilde{x}$ and $\tilde{y}$ implies
	$v(\tilde{y},0)=\tilde{y}$ and $v(\tilde{y},1)=\tilde{x}$.
	Therefore,
	the assignment $v\colon (\mathcal{D}(\tilde{p})\cap \partial\tilde{V}_i)\times[0,1]\to 
	(\mathcal{D}(\tilde{p})\cap \partial\tilde{V}_i)$ defines a deformation retraction
	of $\mathcal{D}(\tilde{p})\cap \partial\tilde{V}_i$ onto the point $\tilde{x}_0$.
	This shows that $\mathcal{D}(\tilde{p})\cap \partial\tilde{V}_i$
	is contractible, as claimed.
	This proves the first assertion.
	
	The first assertion implies that
	$\partial \mathcal{D}(p)\cap\partial W$
	is a simple closed curve on $\partial \mathcal{D}(p)$.
	The subdivision property in the second assertion would fail
	precisely when $\partial \mathcal{D}(p)\cap\partial W$
	is contained in the interior of some polygonal $2$--cell
	on $\partial \mathcal{D}(p)$.
	However, we observe that 
	$\partial W$ the union of all the topological disks  
	$\mathcal{D}(p)\cap\partial W$,
	where $p$ ranges over all the points in $\mathscr{P}$
	with $\mathcal{D}(p)$ not contained in $W$.
	It follows that the union of all the simple closed curves
	$\partial\mathcal{D}(p)\cap\partial W$ is a $1$--skeleton of $\partial W$,
	which must be connected.
	Therefore, we can rule out the possibility that 
	$\partial\mathcal{D}(p)\cap\partial W$ is contained in the interior
	of a polygonal $2$--cell on $\partial\mathcal{D}(p)$.
	This proves the second assertion.
\end{proof}

Under the assumption that $\delta$ is smaller than $\epsilon\times 10^{-1}$,
we see that the truncated regions $\mathcal{D}_W(p)$
are all polyhedral $3$--cells, by Lemma \ref{cell_intersection}.
Hence, 
the truncation of the polyhedral cell division (\ref{first_division}) of $M$
results in an (exact) polyhedral cell division (\ref{first_division_truncated}) of $W$.

\begin{lemma}\label{cell_estimates}
	Suppose $0<\epsilon\leq1$ and $0<\delta\leq\epsilon\times 10^{-1}$.
	The following estimates hold 
	for the polyhedral cell division (\ref{first_division_truncated}).
	\begin{enumerate}
	\item
	The cardinality of $\mathscr{P}$ is bounded by $\delta^{-3}\cdot\mathrm{Vol}(W)$.
	\item 
	If $\partial\mathcal{D}_W(p)$ contains no $2$--cells on $\partial W$,
	the number of $2$--cells on $\partial\mathcal{D}_W(p)$ is bounded by $10^3$.
	\item
	If $\mathcal{D}_W(p)$ contains some $2$--cell on $\partial W$,
	the number of $2$--cells on $\partial\mathcal{D}_W(p)\setminus \partial W$
	is bounded by $(\epsilon/\delta)^{6}\times 10^3$.
	Moreover, the number of $1$--cells in
	$\partial\mathcal{D}_W(p)\setminus \partial W$ is also 
	bounded by $(\epsilon/\delta)^6\times 10^3$.
	\item 
	For each boundary component $\partial V_i$ of $W$,
	there exists some simple closed curve on $\partial V_i$
	which bounds a totally geodesic meridional disk in $V_i$,
	misses the $0$--cells on $\partial V_i$,
	and intersects transversely with the $1$--cells on $\partial V_i$.
	Moreover,
	the number of intersections of any such curve 
	with the $1$--cells in $\partial V_i$
	is bounded by $\delta^{-1}\cdot(\epsilon/\delta)^8\cdot\mathrm{Wri}(V_i)\times 10^4$
	in total.
	\end{enumerate}	
\end{lemma}

\begin{proof}
	In estimation below involving hyperbolic ball volume, 
	we often apply the inequalities 
	$$\frac{4\pi r^3}{3}<\mathrm{Vol}(\mathbf{B}(r))<2\pi r\cdot\sinh^2(r),$$
	without further explanation.
	Here, the lower bound is the volume comparison with Euclidean balls of radius $r$;
	the upper bound is the volume comparison with 
	any hyperbolic tube of depth $r$ and systole $2r$;
	see Formulas \ref{tube_formula} and \ref{ball_formula}.
	For $0<r<2$, 
	we often apply $\sinh(r)<(\sinh(2)/2)\cdot r$, where $\sinh(2)/2\approx 1.813$.
	
	The cardinality of $\mathscr{P}$ is bounded by
	$\lfloor\mathrm{Vol}(W)\,/\,\mathrm{Vol}(\mathbf{B}(2\delta))\rfloor$.
	So, we can bound the cardinality of $\mathscr{P}$ 
	by $\delta^{-3}\cdot\mathrm{Vol}(W)$.
	
	If $\partial\mathcal{D}_W(p)$ contains no $2$--cells on $\partial W$,
	$\mathcal{D}_W(p)$ is just $\mathcal{D}(p)$
	in the original Dirichlet--Voronoi division of $M$.
	Passing to the universal cover $\mathbb{H}^3$ and any lift $\mathcal{D}(\tilde{p})$ of $\mathcal{D}(p)$,
	the region $\mathcal{D}(\tilde{p})$ is contained in the $2\delta$--neighborhood of $\tilde{p}$,
	so the faces lie on the mid-perpendicular hyperplanes separating the geodesic segments $[\tilde{q},\tilde{p}]$
	where $\tilde{q}\in\widetilde{\mathscr{P}}$ lie within distance $4\delta$ from $\tilde{p}$.
	Since the $\delta$--neighborhoods of all $\tilde{q}$ are mutually disjoint in $\mathbb{H}^3$,
	the number of faces on $\partial\mathcal{D}(p)$ is at most
	$\lfloor\mathrm{Vol}(\mathbf{B}(5\delta))\,/\,\mathrm{Vol}(\mathbf{B}(\delta))\rfloor-1$.
	In this case,
	we can bound the number of $2$--cells on $\partial\mathcal{D}_W(p)$
	by $10^3$.
	
	If $\mathcal{D}_W(p)$ contains some $2$--cell on $\partial W$,
	any lifted region $\mathcal{D}_{\tilde{W}}(\tilde{p})$ in $\mathbb{H}^3$
	is the intersection of the lifted Dirichlet--Voronoi region $\mathcal{D}(\tilde{p})$
	with the preimage $\tilde{W}$ of $W$.
	The lifted region $\mathcal{D}_{\tilde{W}}(\tilde{p})$
	is contained in	the $(2\delta+(\epsilon/2))$--neighborhood of $\tilde{p}$,
	and 
	intersects some unique cylindrical tube $\tilde{V}_i$, 
	which covers some tube $V_i\subset M$.
	Similarly as above,
	the number of faces on $\partial\mathcal{D}(\tilde{p})$ 
	that intersects $\tilde{W}$ is at most
	$\lfloor\mathrm{Vol}(\mathbf{B}(5\delta+\epsilon))\,/\,\mathrm{Vol}(\mathbf{B}(\delta))\rfloor-1$,
	so the number of edges on $\partial\mathcal{D}(\tilde{p})$
	that intersects $\tilde{W}$ is at most
	$(\lfloor\mathrm{Vol}(\mathbf{B}(5\delta+\epsilon))\,/\,\mathrm{Vol}(\mathbf{B}(\delta))\rfloor-1)^2/2$.
	Since each edge on $\partial\mathcal{D}(\tilde{p})$ is a geodesic segment,
	having at most $2$ points of intersections with the convex set boundary $\partial \tilde{V}_i$,
	the number of $1$--cells on the truncated region $\partial\mathcal{D}_{\tilde{W}}(\tilde{p})$
	which are contained in $\partial\tilde{V}_i$ 
	is at most $(\lfloor\mathrm{Vol}(\mathbf{B}(5\delta+\epsilon))\,/\,\mathrm{Vol}(\mathbf{B}(\delta))\rfloor-1)^2$.
	This leads to an upper bound
	$(\lfloor\mathrm{Vol}(\mathbf{B}(5\delta+\epsilon))\,/\,\mathrm{Vol}(\mathbf{B}(\delta))\rfloor-1)
	+(\lfloor\mathrm{Vol}(\mathbf{B}(5\delta+\epsilon))\,/\,\mathrm{Vol}(\mathbf{B}(\delta))\rfloor-1)^2$
	for the number of $2$--cells on $\partial\mathcal{D}_{\tilde{W}}(\tilde{p})\setminus\partial\tilde{V}_i$.
	In this case, we can bound the number of $2$--cells on $\partial\mathcal{D}_{W}(p)\setminus\partial W$,
	and also the number of $1$--cells in $\partial\mathcal{D}_{W}(p)\cap\partial W$,
	by $(\epsilon/\delta)^{6}\times 10^3$.
	
	For any boundary component $\partial V_i$ of $W$,
	consider a totally geodesic meridional disk in $F_i$ in $\partial V_i$.
	Except finitely many choices in bad positions,
	$\partial F_i$ sits in general position with respect to the $1$--skeleton on $\partial V_i$,
	missing all the $0$--cells and 
	intersecting transversely with all the $1$--cells.
		
	To estimate the number of intersections of $\partial F_i$ with the $1$--cells on $\partial V_i$,
	we look at any lift $\partial \tilde{F}_i\subset\partial \tilde{V}_i$ in $\mathbb{H}^3$.
	The number of $1$--cells on $\partial F_i$ is equal to the number of $1$--cells on $\partial \tilde{F}_i$,
	and also equal to the number of $0$--cells on $\partial \tilde{F}_i$.
	Each $0$--cell on $\partial \tilde{F}_i$ arises from
	the intersection on $\partial \tilde{V}_i$ 
	of $\partial \tilde{F}_i$ with some face of $\mathcal{D}(\tilde{p})$.
	Since both $\tilde{F}_i$ and any face of $\mathcal{D}(\tilde{p})$ are totally geodesic,
	the above intersection is contained the intersection of $\partial \tilde{V}_i$
	with a geodesic line.
	Therefore, any face of $\mathcal{D}(\tilde{p})$, and 
	hence any $2$--cell on $\partial\mathcal{D}_{\tilde{W}}(\tilde{p})\setminus\partial\tilde{V}_i$,
	may only intersect $\partial\tilde{F}_i$ in at most $2$ points.
	
	On the other hand,
	any $\mathcal{D}_{\tilde{W}}(\tilde{p})$ that intersects with $\partial\tilde{F}_i$
	is contained within distance $(4\delta+\epsilon/2)$ from $\partial\tilde{F}_i$.
	So, the number of all such $\mathcal{D}_{\tilde{W}}(\tilde{p})$
	can be bounded by the volume of the $(4\delta+\epsilon/2)$--neighborhood
	of $\partial\tilde{F}_i$ divided by $\mathrm{Vol}(\mathbf{B}(\delta))$,
	and the number of all $2$--cells on $\mathcal{D}_{\tilde{W}}(\tilde{p})\setminus\partial\tilde{V}_i$
	is bounded by $(\epsilon/\delta)^{6}\times 10^3$ times the above number.
	Moreover, the volume of the $(4\delta+\epsilon/2)$--neighborhood of $\partial\tilde{F}_i$
	can be bounded by the product of
	the area of a hyperbolic disk of radius $(4\delta+\epsilon/2)$
	and the circumference of a hyperbolic disk of radius $r_i+4\delta+\epsilon/2$,
	where $r_i$ denotes the radius of $F_i$.
	This product can be calculated as
	$4\pi\cdot\sinh^2(2\delta+\epsilon/4)$ times $2\pi\cdot\sinh(r_i+4\delta+\epsilon/2)$,
	which can be bounded by
	$8\pi^2\cdot \sinh^2(\epsilon/2)\cdot \exp(\epsilon)\cdot \sinh(r_i)
	=4\pi\cdot \sinh^2(\epsilon/2)\cdot \exp(\epsilon)\cdot \mathrm{Wri}(V_i)$.
	
	After simplification, 
	we can bound the number of intersections of $\partial F_i$
	with the $1$--cells on $\partial V_i$
	by $\delta^{-1}\cdot(\epsilon/\delta)^8\cdot\mathrm{Wri}(V_i)\times 10^4$,
	as asserted.
\end{proof}

\emph{Step 3}. 
We extend the polyhedral cell division (\ref{first_division_truncated}) of $W$ 
to be a new polyhedral cell division (\ref{second_division}) of $M$,
by dividing each tube $V_i$ using only one extra $2$--cell.
The procedure turns out to be quite efficient,
based on the following topological lemma.

\begin{lemma}\label{tube_division}
	Suppose that a topological tube $V$ has a prescribed polygonal cell division
	on $\partial V$, such that each $2$--cell has at most $r$ edges.
	Suppose that some simple closed curve on $\partial V$
	bounds a meridional disk in $V$, misses the $0$--cells on $\partial V$,
	and intersects the $1$--cells transversely at exactly $n$ points.
	
	Then, there exists a polyhedral cell division of $V$ which agrees
	with the prescribed division on $\partial V$, 
	such that there are no additional $1$--cells,
	and there is exactly one $2$--cell 
	and one $3$--cell in $\mathrm{int}(V)$.
	Moreover, the additional polygonal $2$--cell	has at most $nr$ edges.
\end{lemma}

\begin{proof}
	Denote by $\Gamma$ the $1$--skeleton on $\partial V$.
	Denote by $\mu$ the simple closed curve on $\partial V$ as assumed.
	The assumptions guarantee that
	$\mu$ can be isotoped on $\partial V$
	limiting to a cellular loop $\gamma\colon S^1\to \Gamma$ 
	consisting of at most $nr$ $1$--cells.
	
	To see this, one may first isotope $\mu$ without increasing the number of intersections
	with $\Gamma$, ending up with a simple closed curve $\mu'$ on $\partial V$ 
	which intersects $\Gamma$ only within small neighborhoods
	of the vertices. Then, for each $2$--cell,
	isotope the arcs of $\mu'$ in that $2$--cell relative to their end points
	into small collar neighborhood of the polygonal boundary,
	(working on the arcs one after another, in the ``outmost first'' order).
	After that, isotope the deformed $\mu'$ within the small neighborhood of $\Gamma$,
	limiting to a unique cellular loop $\gamma$,
	such that the intersection points all become the nearby vertices,
	and where the arcs all become embedded cellular subarcs 
	of the polygonal boundary	(possibly single vertices).
	Since each open arc in $\mu\setminus \Gamma$ stays in the original $2$--cell all the time,
	in the end, each deformed arc in $\gamma$ comprises at most $r$ $1$--cells.
	Since $\mu\setminus\Gamma$ consists of at most $n$ open arcs,
	$\gamma$ comprises at most $nr$ $1$--cells, as claimed.
	
	To describe the polyhedral decomposition of $V$, 
	we may denote the above deformation as $\gamma_t\colon S^1\to \partial V$,
	for all $t\in [0,1]$, such that $\gamma_0=\gamma$, and $\gamma_1=\mu$.
	This is an isotopy except at $t=0$.
	It gives rise to a map of an annulus
	$S^1\times[0,1]\to \partial V\times[0,1]\colon (u,t)\mapsto (\gamma_t(u),t)$.
	Fill the boundary torus $\partial V\times\{1\}$ with a solid torus with meridian $\mu\times\{1\}$.
	Extend the above map over a disk bounded by $S^1\times\{1\}$.
	Identify the boundary torus $\partial V\times\{0\}$ with $\partial V$ with 
	the prescribed division.
	Then the filled up solid torus can be identified with $V$,
	and the filled up map is a characteristic map of a polygonal $2$--cell in $\mathrm{int}(V)$,
	whose boundary maps as $\gamma$. The complement of the $2$--cell in $\mathrm{int}(V)$ is a $3$--cell.
	This is our polyhedral cell division of $V$ as asserted.	
\end{proof}

Continue from the polyhedral cell division (\ref{first_division_truncated}) of $W$.
We extend the polygonal cell division on $\partial V_i$
to a polyhedral cell division of $V_i$,
using Lemma \ref{tube_division}.
Topologically, the construction adds one polygonal $2$--cell in each $V_i$ 
to the existing $2$--skeleton of $W$,
which divides $\mathrm{int}(V_i)$ into a $3$--cell $\mathrm{int}\mathcal{E}_i$.
Again by path-end compactification,
we obtain an abstract compact polyhedral $3$--cell $\mathcal{E}_i$,
together with a characteristic map $\mathcal{E}_i\to V_i$.
Therefore, 
we obtain a new polyhedral cell decomposition of $M$,
denoted as
\begin{equation}\label{second_division}
M=
\frac{
\left(\bigsqcup_{p\in\mathscr{P}}\mathcal{D}_W(p)\right)\,\sqcup\,
\left(\bigsqcup_{i=1}^{s}\mathcal{E}_i\right)
}{\mbox{side pairing}}.
\end{equation}

\begin{lemma}\label{cell_estimates_more}
	Suppose $0<\epsilon\leq1$ and $0<\delta\leq\epsilon\times 10^{-1}$.
	The following estimates hold 
	for the polyhedral cell division (\ref{second_division}).
	\begin{enumerate}
	\item
	The estimates in Lemma \ref{cell_estimates} regarding $\mathscr{P}$
	and all $\partial\mathcal{D}_W(p)$.
	\item
	In each $V_i$, 
	there is a unique polygonal $2$--cell in $V_i$,
	resulting from $\mathcal{E}_i$ side pairing with itself.
	The number of edges of this polygonal $2$--cell is bounded by
	$\delta^{-1}\cdot(\epsilon/\delta)^{14}\cdot\mathrm{Wri}(V_i)\times 10^7$.	
	\end{enumerate}	
\end{lemma}

\begin{proof}
	The estimates in Lemma \ref{cell_estimates} regarding $\mathscr{P}$
	and all $\partial\mathcal{D}_W(p)$ have not changed
	since (\ref{second_division}) only extends 
	the polyhedral cell division (\ref{first_division_truncated}) on $W$.
	The asserted bound for the edge number of the extra polygonal $2$--cell in $V_i$
	follows directly from Lemmas \ref{tube_division} and \ref{cell_estimates}.	
\end{proof}

\emph{Step 4}. 
We obtain a Heegaard diagram 
\begin{equation}\label{second_division_hd}
(\Sigma,\boldsymbol{\alpha},\boldsymbol{\beta})
\end{equation} 
from the polyhedral cell division (\ref{second_division}) in Step 3.
This procedure works in general for any polyhedral cell division,
and involves some choices of $1$--cells and $2$--cells 
for constructing the $\boldsymbol{\alpha}$--curves
and the $\boldsymbol{\beta}$--curves.
It goes as follows:
Take the Heegaard surface $\Sigma\subset M$ 
to be the boundary of some regular neighborhood of the $1$--skeleton in $M$;
denote by $U_\beta\subset M$ the handlebody bounded by $\Sigma$
which contains the $1$--skeleton,
and by $U_\alpha\subset M$ the other handlebody bounded by $\Sigma$,
which contains the dual $1$--skeleton 
(with the dual $0$--cells in the $3$--cells and the dual $1$--cells in the $2$--cells);
choose a maximal subset of $1$--cells, such that the rest of the $1$--skeleton is a spanning tree,
and similarly, choose a maximal subset of dual $1$--cells,
such that the rest of the dual $1$--skeleton is a spanning tree;
denote by $\alpha_1,\cdots,\alpha_g$ 
the simple closed curves on $\Sigma$
which are the intersection of the chosen $2$--cells with $\Sigma$;
denote by $\beta_1,\cdots,\beta_g$
the simple closed curves on $\Sigma$
which bound meridional disks 
each intersecting an (only) distinct chosen $1$--cell 
transversely at a unique point.
In the above construction,
observe that $M\setminus(\mathrm{int}(F_1\cup\cdots\cup F_s)$
is connected, so we can require in addition
that $\alpha_i=F_i\cap \Sigma$ for $i=1,\cdots,s$.

The Heegaard diagram (\ref{second_division_hd})
is the final output of our construction.
Its quantitive features are summarized as follows.

\begin{lemma}\label{second_division_hpl_estimate}
	Adopt the notations 
	$M=W\cup(V_1\cup\cdots\cup V_s)$ and $\epsilon$
	and assumptions as in Theorem \ref{main_vol_vs_hpl}.
	
	Suppose $0<\epsilon\leq1$ and $0<\delta\leq\epsilon\times 10^{-1}$.
	Then, $M$ admits a Heegaard diagram $(\Sigma,\boldsymbol{\alpha},\boldsymbol{\beta})$
	of genus $g\geq s$ with the following properties.
	\begin{itemize}
	\item For each $i=1,\cdots,s$,
	the curve $\alpha_i$ contains at most 
	$\delta^{-1}\cdot(\epsilon/\delta)^{14}\cdot\mathrm{Wri}(V_i)\times 10^7$
	intersection points with the $\boldsymbol{\beta}$--curves.
	\item For each $i=s+1,\cdots,g$,
	the curve $\alpha_i$ contains at most $(\epsilon/\delta)^6\times 10^3$
	intersection points with the $\boldsymbol{\beta}$--curves.
	\item The genus $g$ is bounded by
	$s+\delta^{-3}\cdot(\epsilon/\delta)^6\cdot\mathrm{Vol}(W)\times 10^3$.
	\end{itemize}
\end{lemma}

\begin{proof} 
	It suffices to check that the Heegaard diagram (\ref{second_division_hd})
	satisfy the asserted properties.	
	Note that when constructing (\ref{second_division_hd}),
	we have picked out 
	some $1$--cells and $2$--cells from the polyhedral cell division (\ref{second_division}),
	so the upper bounds in Lemmas \ref{cell_estimates} and \ref{cell_estimates_more}) 
	remain valid below, although they overestimate.
	
	For each $i=1,\cdots,s$, the asserted bound 
	$\delta^{-1}\cdot(\epsilon/\delta)^{14}\cdot\mathrm{Wri}(V_i)\times 10^7$ 
	for intersection points on $\alpha_i$
	follows immediately from Lemma \ref{cell_estimates_more}.
	
	For each $i=s+1,\cdots,s$,
	$\alpha_i$ comes from a polygonal $2$--cell $c$ on some $\partial\mathcal{D}_W(p)$,
	the asserted bound $(\epsilon/\delta)^6\times 10^3$ follows from Lemma \ref{cell_estimates},
	by considering three possible cases.
	If $\mathcal{D}_W(p)$ does not contain any $2$--cell on $\partial W$,
	the number of edges on $\partial c$ is at most the number of $2$--cells on $\partial\mathcal{D}_W(p)$,
	since $\mathcal{D}_W(p)=\mathcal{D}(p)$ is convex with totally geodesic faces.
	In this case, the number of edges on $\partial c$ is at most $10^3$ (Lemma \ref{cell_estimates}).
	If $\mathcal{D}_W(p)$ does not contains some $2$--cell on $\partial W$,
	and if $c$ is not contained in $\partial W$, then $\partial c$ contains at most $1$--cell on $\partial W$,
	and any other $1$--cell in $\partial c$ comes from intersecting $c$ with 
	another $2$--cell in $\mathcal{D}_W(p)\setminus \partial W$ (see Lemma \ref{cell_intersection}).
	In this case, the number of edges on $\partial c$ is at most $(\epsilon/\delta)^6\times 10^3$ (Lemma \ref{cell_estimates}).
	If $\mathcal{D}_W(p)$ does not contains some $2$--cell on $\partial W$,
	and if $c$ is contained in $\partial W$, 
	then $c$ is $\partial\mathcal{D}(p)\cap\partial W$.
	In this case, the number of edges on $\partial c$ is, again,
	at most $(\epsilon/\delta)^6\times 10^3$ (Lemma \ref{cell_estimates}).
	
	To bound the genus of $\Sigma$, it suffices to estimate the number of $2$--cells in (\ref{second_division}).
	Except the $s$ extra $2$--cells in $V_1,\cdots,V_s$,
	any other $2$--cell is a face of some $\mathcal{D}_W(p)$.
	Therefore, the number of other $2$--cells is easily bounded
	by the cardinality of $\mathscr{P}$ times 
	an upper bound of face numbers for each $\mathcal{D}_W(p)$.
	This yields the asserted upper bound
	$s+\delta^{-3}\cdot(\epsilon/\delta)^6\cdot\mathrm{Vol}(W)\times 10^3$ (Lemma \ref{cell_estimates}).
\end{proof}

With the above preparation, we finish the proof of Theorem \ref{main_vol_vs_hpl} as follows.

Let $M$ be an orientable closed hyperbolic $3$--manifold.
Suppose that $V_1,\cdots,V_s\subset M$ 
are embedded, mutually disjoint, hyperbolic tubes with boundary.
Denote by $W=M\setminus \mathrm{int}(V_1\cup\cdots\cup V_s)$
the complementary $3$--manifold with boundary.
	
For any constant $0<\epsilon\leq1$, 
we obtain a Heegaard diagram $(\Sigma,\boldsymbol{\alpha},\boldsymbol{\beta})$ presenting $M$,
as constructed in (\ref{second_division_hd}) with respect to 
$$\delta=\epsilon\times10^{-1}.$$
The presentation length $\ell(\mathbf{u}_\beta;\mathbf{w}_\alpha)$
of $(\Sigma,\boldsymbol{\alpha},\boldsymbol{\beta})$
is an upper bound for the Heegaard presentation length $\ell_{\mathtt{He}}(M)$
(Definition \ref{def_Heegaard_pl}).
Moreover, $\ell(\mathbf{u}_\beta;\mathbf{w}_\alpha)$
is obviously bounded by the total number of intersections
between the $\boldsymbol{\alpha}$--curves and the $\boldsymbol{\beta}$--curves.
By Lemma \ref{second_division_hpl_estimate}, we obtain
\begin{eqnarray*}
\ell_{\mathtt{He}}(M) &\leq&
\delta^{-3}\cdot(\epsilon/\delta)^{12}\cdot\mathrm{Vol}(W)\times 10^6
+
\delta^{-1}\cdot(\epsilon/\delta)^{14}\cdot\sum_{i=1}^s \mathrm{Wri}(V_i)\times 10^7\\
&\leq&
10^{22}\cdot\left(
\epsilon^{-3}\cdot\mathrm{Vol}(W)+
\epsilon^{-1}\cdot\sum_{i=1}^s \mathrm{Wri}(V_i)
\right),
\end{eqnarray*}
as desired.

This completes the proof of Theorem \ref{main_vol_vs_hpl}.

\section{Entropy versus volume: With arithmeticity}\label{Sec-ent_vs_vol_arithmetic}

	%

This section is devoted to the proof of Theorem \ref{main_ent_vs_vol_arithmetic}.

All we need from arithmeticity
has been encapsulated into Lemma \ref{short_tubes_wrist_sum_arithmetic},
so the rest of the proof can be read 
without knowing 
what an arithmetic hyperbolic $3$--manifold precisely means.
For backgrounds on arithmetic hyperbolic $3$--manifolds,
we refer to Neumann and Reid's survey \cite{NR_arithmetic};
for interesting properties of
arithmetic hyperbolic surface bundles, see also \cite{BMR_arithmetic_surface_bundle}.

For any orientable closed hyperbolic $3$--manifold $M$ and any constant $\mu>0$,
the $\mu$--\emph{thin part} of $M$ consists of all the points
at which the injectivity radius is less than $\mu/2$.
If $\mu$ is a Margulis number for $M$,
by definition, the $\mu$--thin part of $M$
is a disjoint union of open hyperbolic tubes,
called the $\mu$--\emph{Margulis tubes}.
For example,
$0.104$ is known to be a uniform Margulis number 
for all complete hyperbolic $3$--manifolds,
due to Meyerhoff \cite[Section 9]{Meyerhoff_margulis_constant}.

\begin{lemma}\label{short_tubes}
	Let $M$ be any closed orientable hyperbolic $3$--manifold.
	Let $\mu>0$ be any Margulis number for $M$.
	Denote by $V_1,\cdots,V_m\subset M$
	the closures of the $\mu/2$--Margulis tubes.
	Note that all $V_i$ are embedded, mutually disjoint hyperbolic tubes.	
	Then, the following statements hold.
	\begin{enumerate}
	\item 
	If some $V_i$ has systole at most $\mu/4$,
	then $V_i$ has volume at least $(4\pi/3)\cdot(\mu/8)^3$.
	\item
	If some subcollection $V_{i_1},\cdots,V_{i_s}$ all have systole at most $\mu/4$,
	then 
	the compact distance $\mu/8$--neighborhood of $\partial(V_{i_1}\cup\cdots\cup V_{i_s})$
	in $M$ is bicollar.
	\end{enumerate}
\end{lemma}

\begin{proof}
	If some $V_i$ has systole at most $\mu/4$, then the depth of $V_i$ 
	(that is, the distance from the boundary to the core geodesic) is greater than $\mu/8$,
	since the injectivity radius on $\partial V_i$ is at least $\mu/4$ in $M$.
	For any point $x\in V_i$ of distance $\mu/8$ from $\partial V_i$,
	the injectivity radius at $x$ is greater than $\mu/8$.
	Therefore, $V_i$ contains 
	an embedded hyperbolic ball centered at $x$ of radius $\mu/8$.
	It follows that $V_i$ has volume at least 
	$\mathrm{Vol}(\mathbf{B}(\mu/8))>(4\pi/3)\cdot(\mu/8)^3$.
	
	If some subcollection $V_{i_1},\cdots,V_{i_s}$ all have systole at most $\mu/4$,
	then the above argument shows that the compact distance $\mu/8$--neighborhood
	of $\partial(V_{i_1}\cup\cdots\cup V_{i_s})$ in $V_{i_1}\cup\cdots\cup V_{i_s}$ is collar.
	On the other hand, each $V_i$ is contained in a $\mu$--Margulis tube
	of distance at least $\mu/4$ from the boundary, and 
	all the $\mu$--Margulis tubes have mutually disjoint interior,
	so, the compact distance $\mu/8$--neighborhood
	of $\partial(V_{i_1}\cup\cdots\cup V_{i_s})$
	in $M\setminus(\mathrm{int}V_{i_1}\cup\cdots\cup V_{i_s})$
	is also collar.	Therefore, the compact distance $\mu/8$--neighborhood
	of $\partial(V_{i_1}\cup\cdots\cup V_{i_s})$ in $M$ is bicollar.
\end{proof}

\begin{lemma}\label{short_tubes_wrist_sum}
	Let $M$ be any closed orientable hyperbolic $3$--manifold.
	Let $\mu>0$ be any Margulis number for $M$.
	Denote by $V_1,\cdots,V_s\subset M$
	the closures of the $\mu/2$--Margulis tubes with systole at most $\mu/4$.
	Then, the following comparison holds.
	$$\sum_{i=1}^s\mathrm{Wri}(V_i)\leq 
	\sqrt{3}\cdot(\mu/8)^{-3/2}\cdot\frac{1}{\sqrt{\mathrm{Syst}(M)}}\cdot\sum_{i=1}^s\mathrm{Vol}(V_i).$$
\end{lemma}

\begin{proof}
	By Formula \ref{tube_formula} and Lemma \ref{short_tubes}, we estimate
	\begin{eqnarray*}
	\sum_{i=1}^s \mathrm{Wri}(V_i)
	&=&
	\sum_{i=1}^s \sqrt{\frac{4\pi\cdot\mathrm{Vol}(V_i)}{\mathrm{Syst}(V_i)}}\\
	&\leq&
	\sum_{i=1}^s \sqrt{\frac{4\pi\cdot\mathrm{Vol}(V_i)}{\mathrm{Syst}(M)}}\\
	&\leq&
	\frac{1}{\sqrt{\mathrm{Syst}(M)}}\cdot
	\sum_{i=1}^s 
	\sqrt{
	4\pi\cdot\mathrm{Vol}(V_i)\cdot
	\frac{\mathrm{Vol}(V_i)}{(4\pi/3)\cdot(\mu/8)^3}
	}\\
	&\leq&
	\sqrt{3}\cdot(\mu/8)^{-3/2}\cdot
	\frac{1}{\sqrt{\mathrm{Syst}(M)}}
	\cdot\sum_{i=1}^s \mathrm{Vol}(V_i),
	\end{eqnarray*}
	as desired.
\end{proof}

\begin{lemma}\label{short_tubes_wrist_sum_arithmetic}
	There exists some constant $D(\mu)>0$ depending only on $\mu$,
	such that the following statement holds.
		
	Let $M$ be any arithmetic closed orientable hyperbolic $3$--manifold.
	Let $\mu>0$ be any Margulis number for $M$.
	Denote by $V_1,\cdots,V_s\subset M$
	the closures of the $\mu/2$--Margulis tubes with systole at most $\mu/4$.
	The following comparison holds.
	$$\sum_{i=1}^s\mathrm{Wri}(V_i)\leq D(\mu)\cdot \mathrm{Vol}(M).$$
\end{lemma}

\begin{proof}
	Denote by ${\mathbf{k}}$ the invariant trace field of $M$.
	In the arithmetic case,
	we recall that ${\mathbf{k}}$ is a finite extension of $\Rational$
	with exactly one complex place (besides real places).
	The field isomorphism type of $\mathbf{k}$
	depends only on the commensurability class of $M$.
	Upon fixing an orientation of $M$,
	the holonomy representation $\pi_1(M)\to\mathrm{PSL}(2,\Complex)$
	(up to group conjugation)
	determines a complex embedding $\mathbf{k}\subset \Complex$.
	The complex conjugate embedding corresponds to the reversed orientation.
	See \cite[Section 2]{NR_arithmetic}.
	
	Denote by $d_{\mathbf{k}}$ the degree of $\mathbf{k}$ over $\Rational$,
	and $\Delta_{\mathbf{k}}$ 
	the absolute value of the discriminant of $\mathbf{k}$ over $\Rational$.
	
	For any closed geodesic $\gamma$ in $M$,
	the hyperbolic length of $\gamma$ takes the form $2\log |u|$ for some
	$u\in\Complex$ of modulus $|u|>1$, 
	such that $u^2$ is an algebraic integer in some quadratic extension of $\mathbf{k}$.
	Recall that the Mahler measure of any algebraic integer $\alpha$ over $\Rational$ 
	is defined as
	$\mathbb{M}(\alpha)=\prod_{i=1}^m\mathrm{max}(|\sigma_i(\alpha)|,1),$
	where $\sigma_1(\alpha),\cdots,\sigma_m(\alpha)$
	denote all the Galois conjugates of $\alpha$ in $\Complex$.
	Therefore, the length of $\gamma$ is equal to $2\log\mathbb{M}(u)$ 
	if $u$ is real,	or $\log\mathbb{M}(u)$ otherwise.
	The former occurs when $\gamma$ is purely hyperbolic,
	and the latter occurs when $\gamma$ is strictly loxodromic.
	See \cite[Section 4.4]{NR_arithmetic} or \cite[Section 3.1]{BMR_arithmetic_surface_bundle}.
	
	Since $M$ is closed and arithmetic, there exists 
	some explicit universal constant $c_1>0$,
	such that the following estimate holds:
	\begin{equation}\label{systole_estimate_arithmetic}
	\mathrm{Syst}(M)\geq c_1\cdot \left(\frac{\log\log d_{\mathbf{k}}}{\log d_{\mathbf{k}}}\right)^3.
	\end{equation}
	This follows immediately from Dobrowolski's celebrated lower bound for 
	the Mahler measure of algebraic integers with bounded degree \cite{Dobrowolski_bound} 
	(see also \cite[Theorem 3.5]{BMR_arithmetic_surface_bundle}),
	and the length formula recalled above.

	Since $V_1,\cdots,V_s$ are all contained in the $\mu/2$--thin part of $M$,
	there exists some constant $c_2(\mu)>0$ depending only on $\mu$,
	such that the following estimate holds:
	\begin{equation}\label{thin_part_volume_estimate_arithmetic}
	\sum_{i=1}^{s}\mathrm{Vol}(V_i)\leq c_2(\mu)\cdot \mathrm{Vol}(M)\cdot\Delta_{\mathbf{k}}^{-4/9}
	\end{equation}
	This follows from a remarkable recent result due to Fr\k{a}czyk \cite[Theorem 3.1]{Fraczyk_arithmetic}.
	In fact, Fr\k{a}czyk proves a much stronger inequality,
	bounding the volume of 
	the $(\rho+\eta\,d_{\mathbf{k}})$--thin part of
	any (closed or cusped) arithmetic hyperbolic $3$--manifold $M$,
	where $\eta>0$ is some universal constant
	and $\rho>0$ can be any constant;
	the upper bound takes the same form,
	only replacing $c_2(\mu)$ with some constant depending on $\rho$;
	moreover, if $M$ is congruence, 
	then $\mathrm{Vol}(M)$ can be replaced with $\mathrm{Vol}(M)^{11/12}$.

	Since ${\mathbf{k}}$ has exactly one complex place,
	there exists some explicit universal constant $c_3>0$,
	such that the following estimate holds:
	\begin{equation}\label{discriminant_estimate_arithmetic}
	\Delta_{\mathbf{k}}\geq c_3\cdot 60^{d_{\mathbf{k}}}.
	\end{equation}
	This follows from a well-known discriminant lower bound
	due to Odlyzko \cite{Odlyzko_bound}.
	In fact, Odlyzko proves 
	$\Delta_{\mathbf{k}}^{1/d_{\mathbf{k}}}
	=60^{r_1/d_{\mathbf{k}}}\cdot 22^{2r_2/d_{\mathbf{k}}}+o(1)$,
	for $d_{\mathbf{k}}=r_1+2r_2$ tending to infinity,
	where $r_1$ and $r_2$ denote the numbers of real and complex places of $\mathbf{k}$,
	as usual. In our case, we apply to $r_2=1$.
	For proving Lemma \ref{short_tubes_wrist_sum_arithmetic},
	we could also use Minkowski's lower bound 
	$\Delta_{\mathbf{k}}\geq 
	(\pi/4)^{2r_2}\cdot d_{\mathbf{k}}^{2d_{\mathbf{k}}}/ (d_{\mathbf{k}}!)^2$,
	which would make it 
	easier to extract an explicit universal coefficient like $c_3$.

	For all $d_{\mathbf{k}}\geq3$, 
	it is elementary to estimate, 
	with some universal constant $c_4>0$,
	$$60^{-4d_{\mathbf{k}}/9}\cdot\left(\frac{\log d_{\mathbf{k}}}{\log\log d_{\mathbf{k}}}\right)^{3/2}\leq c_4.$$
	
	With the above facts, we apply Lemma \ref{short_tubes_wrist_sum_arithmetic} to estimate
	\begin{eqnarray*}
	\sum_{i=1}^s\mathrm{Wri}(V_i)
	&\leq&
	\sqrt{3}\cdot(\mu/8)^{-3/2}\cdot\frac{1}{\sqrt{\mathrm{Syst}(M)}}\cdot\sum_{i=1}^s\mathrm{Vol}(V_i)\\
	&\leq&
	D(\mu)\cdot\mathrm{Vol}(M),	
	\end{eqnarray*}
	where we can set 
	$$D(\mu)=\sqrt{3}\cdot c_1c_3c_4\cdot (\mu/8)^{-3/2}\cdot c_2(\mu),$$
	as asserted.
\end{proof}

With the above preparation, we finish the proof of Theorem \ref{main_ent_vs_vol_arithmetic}
as follows.

Let $M$ be an arithmetic orientable closed hyperbolic $3$--manifold.
Suppose that $\phi\in H^1(M;\Integral)$ be a fibered class.
For any connected finite cover $M'$ of $M$, 
the pullback $\phi'\in H^1(M';\Integral)$ of $\phi$ to $M'$
is some positive integral multiple $m'\psi'$ 
of a primitive fibered class $\psi'\in H^1(M';\Integral)$.
We can choose some $M'$ of covering degree $2m$,
such that $\psi'$ is dual to a fiber of genus $\geq3$.
(For example, see the proof of Corollary \ref{entropy_estimate_without_b_fine_corollary}.)

%

Let $V_1,\cdots,V_s$ the hyperbolic tubes
as in Lemma \ref{short_tubes_wrist_sum},
with respect to $\mu$ and $M$.
Denote by $W=M\setminus\mathrm{int}(V_1\cup\cdots\cup V_s)$
the complementary $3$--manifold with boundary.
By Lemma \ref{short_tubes},
$V_1,\cdots,V_s$ and $W$ satisfy the assumption of Theorem \ref{main_vol_vs_hpl}
with respect to $\epsilon=\mu/8$.
Applying Theorems \ref{main_ent_vs_hpl} and \ref{main_vol_vs_hpl} 
and Lemma \ref{short_tubes_wrist_sum},
we estimate
\begin{eqnarray*}
\mathrm{Ent}(\phi)&=&\mathrm{Ent}(\psi')/m'\\
&\leq& \mathrm{Ent}(\psi)\\
&\leq& 2\cdot(\ell_{\mathtt{He}}(M)-1)\cdot\log3\\
&\leq& 
2\cdot(\log3)\cdot
10^{22}\cdot\left(
\epsilon^{-3}\cdot\mathrm{Vol}(W)+
\epsilon^{-1}\cdot\sum_{i=1}^s \mathrm{Wri}(V_i)
\right)\\
&\leq&
2\cdot(\log3)\cdot
10^{22}\cdot\left(
\epsilon^{-3}\cdot\mathrm{Vol}(M)+
\epsilon^{-1}\cdot D(8\epsilon)\cdot\mathrm{Vol}(M)
\right)\\
&\leq&
C\cdot\mathrm{Vol}(M_f).
\end{eqnarray*}
For example,
we can set
$$C=10^{23}\cdot \left((\mu/8)^{-3}+(\mu/8)^{-1}\cdot D(\mu)\right).$$
where $D$ is 
as declared in Lemma \ref{short_tubes_wrist_sum_arithmetic},
and $\mu$ is set to be $0.104$.

This completes the proof of Theorem \ref{main_ent_vs_vol_arithmetic}.

\begin{remark}\label{main_ent_vs_vol_arithmetic_remark}\,
\begin{enumerate}
\item
In our proof of Theorem \ref{main_ent_vs_vol_arithmetic} as above,
the only implicit part to construct the asserted constant $C$
comes from the function $D(\mu)$ in Lemma \ref{short_tubes_wrist_sum_arithmetic}.
From the proof of Lemma \ref{short_tubes_wrist_sum_arithmetic},
it is evident that the universal constants $c_1$, $c_3$, and $c_4$ 
can be made explicit.
It seems possible to work out an explicit expression 
of the function $c_2(\mu)$,
by following all the steps of Fr\k{a}czyk's argument in \cite{Fraczyk_arithmetic},
which is constructive in principle.
Then, in principle,
Theorem \ref{main_ent_vs_vol_arithmetic}
can be made efficient by writing down an explicit universal constant $C$.
\item
It is known that there exist arithmetic closed hyperbolic surface bundles
of fiber genus $2$ \cite{Reid_non-Haken}.
Therefore, 
in our proof of Theorem \ref{main_ent_vs_vol_arithmetic},
we cannot apply Theorem \ref{main_ent_vs_hpl} directly to $M$ in general.
This justifies our trick of passing to $M'$.
However, one may conjecture that
Theorem \ref{main_ent_vs_hpl} holds for the case with fiber genus $2$ as well.
\end{enumerate}
\end{remark}

\section{Entropy versus volume: With systole}\label{Sec-ent_vs_vol_general}


This section is devoted to the proof of Theorem \ref{main_ent_vs_vol_general}.

The idea of the proof is similar to the arithmetic case (Section \ref{Sec-ent_vs_vol_arithmetic}).
The main difference is that 
we use the stronger intermediate results, Lemma \ref{entropy_estimate_without_b_fine}
(plus Corollary \ref{entropy_estimate_without_b_fine_corollary}) 
and Lemma \ref{second_division_hpl_estimate},
in place of Theorem \ref{main_ent_vs_hpl} and Theorem \ref{main_vol_vs_hpl},
respectively.

Let $M$ be an orientable closed hyperbolic $3$--manifold.
Suppose that $\phi\in H^1(M;\Integral)$ be a fibered class.

%

Let $V_1,\cdots,V_s$ the hyperbolic tubes
as in Lemma \ref{short_tubes_wrist_sum},
with respect to $\mu=0.104$ and $M$.
Denote by $W=M\setminus\mathrm{int}(V_1\cup\cdots\cup V_s)$
the complementary $3$--manifold with boundary.
By Lemma \ref{short_tubes},
$V_1,\cdots,V_s$ and $W$ satisfy the assumption of Theorem \ref{main_vol_vs_hpl}
with respect to $\epsilon=\mu/8$.
Set $\delta=\epsilon\times 10^{-1}$.

We present $M$ with a Heegaard diagram $(\Sigma,\boldsymbol{\alpha},\boldsymbol{\beta})$
of genus $g\geq s$, as provided by Lemma \ref{second_division_hpl_estimate}.
For each $i=1,\cdots,g$,
denote by $k_i$ the number of intersections of $\alpha_i$ 
with the $\boldsymbol{\beta}$--curves in total.
The irreducibility of $M$ implies $k_i\geq1$ for all $i=1,\cdots,g$.
By Lemma \ref{entropy_estimate_without_b_fine}
and Corollary \ref{entropy_estimate_without_b_fine_corollary},
we can estimate
\begin{eqnarray*}
\mathrm{Ent}(\phi) 
&\leq& 
2\cdot(\log(k_1\cdots k_g)-\log k_{\mathrm{min}})\\
&\leq& 
2\cdot\log(k_1\cdots k_g)\\
&=& 
2\cdot\left(\sum_{i=s+1}^g \log k_i + \sum_{i=1}^s \log k_i\right)\\
&\leq& 
2\cdot\left(
(g-s)\cdot \log\left(10^3\cdot(\epsilon/\delta)^6\right) 
+
\sum_{i=1}^s \log \left(10^{7}\cdot \delta^{-1}\cdot (\epsilon/\delta)^{14}\cdot\mathrm{Wri}(V_i)\right)
\right)
\\
&=& 
2\cdot\left(
(g-s)\cdot \log 10^9
+ s\cdot \log\left(10^{22}\cdot (\mu/8)^{-1}\right)
+ \sum_{i=1}^s \log \mathrm{Wri}(V_i)
\right)
\\
&\leq&
2\cdot\left(30\,(g-s)+60\,s+ \sum_{i=1}^s \log \mathrm{Wri}(V_i)\right).
\end{eqnarray*}
By Lemma \ref{entropy_estimate_without_b_fine}, we estimate
$$g-s
\leq 10^3\cdot\delta^{-3}\cdot (\epsilon/\delta)^6\cdot\mathrm{Vol}(W)
\leq 10^{18}\cdot\mathrm{Vol}(W).$$
By Lemma \ref{short_tubes}, we estimate
$$s\leq \frac{1}{(4\pi/3)\cdot(\mu/8)^3}\cdot \mathrm{Vol}(M_f)
\leq 10^3\cdot \mathrm{Vol}(M_f).$$
By Formula \ref{tube_formula}, we estimate
\begin{eqnarray*}
	\sum_{i=1}^s \log\mathrm{Wri}(V_i)
	&=&
	\sum_{i=1}^s \log\sqrt{\frac{4\pi\cdot\mathrm{Vol}(V_i)}{\mathrm{Syst}(V_i)}}\\
	&\leq&
	\sum_{i=1}^s \log\sqrt{\frac{4\pi\cdot\mathrm{Vol}(V_i)}{\mathrm{Syst}(M)}}\\
	&=&
	\frac{s}2\cdot\log\left(\frac{1}{\mathrm{Syst}(M)}\right)+\frac{1}{2}\cdot\sum_{i=1}^s \log\left(4\pi\cdot\mathrm{Vol}(V_i)\right)\\
	&\leq&
	\frac{s}2\cdot\log\left(3+\frac{1}{\mathrm{Syst}(M)}\right)+2\pi\cdot\sum_{i=1}^s \mathrm{Vol}(V_i).
\end{eqnarray*}
Putting the above estimates together, we can easily bound
\begin{eqnarray*}
\mathrm{Ent}(\phi) 
&\leq& 
62\cdot 10^{18}\cdot \mathrm{Vol}(M)+s\cdot\log\left(3+\frac{1}{\mathrm{Syst}(M)}\right)\\
&\leq& 
62\cdot 10^{18}\cdot \mathrm{Vol}(M)+10^3\cdot\mathrm{Vol}(M)\cdot\log\left(3+\frac{1}{\mathrm{Syst}(M)}\right)\\
&\leq&
10^{20}\cdot\mathrm{Vol}(M)\cdot\log\left(3+\frac{1}{\mathrm{Syst}(M)}\right),
\end{eqnarray*}
as desired.

This completes the proof of Theorem \ref{main_ent_vs_vol_general}.

\section{Example}\label{Sec-example_optimal}
We conclude our discussion with an example,
justifying the form of the upper bound 
in Theorem \ref{main_ent_vs_vol_general}.
Indeed, for any connected closed oriented surface $S$ of genus $\geq2$,
we construct a sequence of pseudo-Anosov mapping classes $[f_n]\in\mathrm{Mod}(S)$,
with the following properties. 
For some constant $K>0$ and for all but finitely many $n$,
$$\mathrm{Ent}([f_n])> K^{-1}\cdot \mathrm{Vol}(M_n)\cdot \log\left(3+\frac{1}{\mathrm{Syst}(M_n)}\right),$$
and
$$\lim_{n\to\infty} \mathrm{Syst}(M_n)=0.$$
Here, $M_n$ denotes the mapping torus of $[f_n]$.
Our construction is as follows.

Let $S$ be any connected closed oriented surface of genus $g\geq2$.
Fix a symplectic basis $\xi_1,\eta_1,\cdots,\xi_g,\eta_g\in H_1(S;\Integral)$.
Namely, the algebraic intersection pairing takes the form
$\langle \xi_i, \xi_j\rangle=0$, and $\langle \eta_i, \eta_j\rangle=0$,
and $\langle \xi_i, \eta_j\rangle=\delta_{ij}$.

Pick a pair of simple closed curves $x$ and $y$ on $S$,
fixing orientations, such that $\xi_1=[x]$ and $\eta_1=[y]$.
Moreover, we assume that $x$ and $y$ together fill $S$,
(that is,
after isotoping $x$ and $y$ to any trasverse position
minimizing the number of intersections,
each complementary component of $x\cup y$ is an open disk).

Denote by $T_x,T_y\in\mathrm{Mod}(S)$ 
the (right-hand) Dehn twists along $x$ and $y$,
respectively.
We construct a sequence of mapping classes $[f_n]\in\mathrm{Mod}(S)$, for all $n\in\Natural$, 
as the composites
\begin{equation}\label{ent_fastest_sequence}
[f_n]=T_x^nT_y^{-1}.
\end{equation}
In particular, this is an instance of Penner's construction,
so the mapping classes $[f_n]$ are all pseudo-Anosov.
It is also a special case of families appeared in Long and Morton \cite{Long_Morton}.

The induced linear automorphism $(f_n)_*\colon H_1(S;\Rational)\to H_1(S;\Rational)$
is represented 
over the fixed symplectic basis
as a square matrix of size $2g$ (acting on column coordinate vectors):
$$(f_n)_*=\left(
\begin{array}{cccc}
{\begin{array}{cc} n+1& n \\ 1& 1\end{array}} & & & \\
&{\begin{array}{cc} 1& 0\\ 0& 1\end{array}}& & \\
& & \ddots & \\
& & & {\begin{array}{cc} 1& 0\\ 0& 1\end{array}}\\
\end{array}
\right)$$ 
The matrix $(f_n)_*$ has an eigenvalue $1$ of multiplicity $(2g-2)$,
and another pair of positive simple eigenvalues $(n+2\pm\sqrt{n^2+4n})/2$.
Hence, $(f_n)_*$ has spectral radius $(n+2+\sqrt{n^2+4n})/2$,
whose logarithm is a lower bound for $\mathrm{Ent}([f_n])$.
This yields 
\begin{equation}\label{ent_log_growth}
\mathrm{Ent}([f_n])\geq \log\left(\frac{n+2+\sqrt{n^2+4n}}{2}\right)>\log n.
\end{equation}

For all $n\in\Natural$,
denote by $M_n$ the mapping torus of $[f_n]$.
Topologically, $M_n$ is homeomorphic to 
the product $3$--manifold $S\times(\Real/\Integral)$
by doing a $1/n$--surgery along $x\times [1/3]$
and a $(-1)$--surgery along $y\times[2/3]$.
To be precise, we think of any simple closed curve $z\times[t]$
on any slice $S\times[t]$ 
as framed by the forward normal vectors to $S\times[t]$,
so a $p/q$--surgery along an oriented $z\times[t]$ 
means removing from $S\times(\Real/\Integral)$
a tube with core $c\times[t]$ and refilling with another tube,
such that the slope on the removed tube
parellel to $p$ times the meridian plus $q$ times the longitude
bounds a disk in the refilled tube;
the longitude is oriented according to the orientation of $c\times[t]$,
and the meridian is oriented compatibly.

Geometrically, the sequence of closed hyperbolic $3$--manifolds $M_n$
converges in the Gromov--Hausdorff sense
to a finite-volume hyperbolic $3$--manifold $M_\infty$ with one cusp,
fixing a Margulis number $\mu>0$ for all $M_n$
and a base point for each $M_n$ in the $\mu$--thick part.
The manifold $M_\infty$ is homeomorphic
to $S\times(\Real/\Integral)$
with $(-1)$--surgery along $y\times[2/3]$
and with $x\times[1/3]$ drilled.
Moreover, it is known that 
\begin{equation}\label{vol_decrease_filling}
\mathrm{Vol}(M_n)<\mathrm{Vol}(M_\infty)
\end{equation}
holds for all $n$.
For all but finitely many $n$,
there are $\mu$--Margulis tubes $V_n\subset M_n$ 
containing the surged core curves,
and they converge to 
the $\mu$--Margulis horocusp $V_\infty\subset M_\infty$.
It follow that the wrist of $V_n$ grows asymptotically as
$$\mathrm{Wri}(V_n)\sim n\cdot w_\infty$$
for $n$ tending to infinity, where $w_\infty>0$ 
denotes the length of the Euclidean geodesic meridian on $\partial V_\infty$.
Formula \ref{tube_formula} implies
$$\mathrm{Syst}(M_n)=\mathrm{Syst}(V_n)
=\frac{4\pi\cdot\mathrm{Vol}(V_n)}{\mathrm{Wri}(V_n)^2}
\sim n^{-2}\cdot \mathrm{Vol}(V_\infty)\cdot\frac{4\pi}{w_\infty^2},$$
for $n$ tending to infinity.
Therefore, we obtain
\begin{equation}\label{syst_log_growth}
\log\left(3+\frac{1}{\mathrm{Syst}(M_n)}\right)=2\,\log n + o(\log n),
\end{equation}
for $n$ tending to infinity.

By (\ref{ent_log_growth}), (\ref{vol_decrease_filling}), and (\ref{syst_log_growth}),
the sequence (\ref{ent_fastest_sequence}) satisfies the asserted properties,
where we can take $K=1+2\cdot\mathrm{Vol}(M_\infty)$.

\appendix

\section{Next-to-top rank versus Nielsen number}\label{Sec-next_to_top_appendix}
This appendix section supplies an exposition of Proposition \ref{constraint_HF} (4).
The asserted inequality regarding the next-to-top term in Heegaard Floer homology
is obtained through identifications with 
certain versions of the monopole Floer homology, 
the periodic Floer homology,
and the symplectic Floer homology.
Each of the identifications holds under certain monotonicity condition.
We review relevant known facts, 
and explain how to derive the asserted inequality.
Our exposition expands an outline appeared in \cite[Section 1.2]{Cotton-Clay_sfh}.

For a similar argument in terms of knot Floer homology,
see \cite[Theorem 1.2]{Ni_hfk_fp} and references therein.
To compare that result with Proposition \ref{constraint_HF} (4),
construct a surface bundle as the $0$--surgery along a fibered knot in a $3$--manifold
(pointed out by Dongtai He).

\subsection{Monopole Floer homology}\label{Subsec-hm}
Monopole Floer homology is a kind of Floer homology
for oriented closed smooth $3$--manifolds.
It is originally developed from the study of Seiberg--Witten equation
(also known as the monopole equation) in $4$--dimensional differential topology.
For a thorough exposition on this topic,
we refer to the book of Kronheimer and Mrowka \cite{KM_book}.
Below, 
we mention a few subtle points that one needs to know,
in order to understand the connection with other Floer homologies.

From a bird's eye view,
there are three flavors $\widecheck{\mathrm{HM}}_*$ (the ``to'' version), 
$\widehat{\mathrm{HM}}_*$ (the ``from'' version), 
and $\overline{\mathrm{HM}}_*$ (the ``bar'' version),
in many ways
like the Heegaard Floer homology versions $\mathrm{HF}^+$, $\mathrm{HF}^-$, and $\mathrm{HF}^\infty$.
Out of a similar exact triangle as (\ref{exact_triangle_mip}),
one may extract another reduced version $\mathrm{HM}_*$,
which is like $\mathrm{HF}^+_{\mathrm{red}}$.
In \cite{KM_book},
Kronheimer and Mrowka also consider negative completion of these versions
(over the formal power series ring $\Integral[[U]]$).
The completed versions are more suitable 
for discussion regarding extra product structures and dualities.
Those completed versions are denoted with a subscript $\bullet$ instead of $*$.
Both the usual or completed versions split into direct summands
with respect to the $\mathrm{Spin}^{\mathtt{c}}$ structures.
See \cite[Chapter I, Section 3]{KM_book} for detailed summary.

More generally,
there are monopole Floer homology with perturbations,
as introduced in \cite[Chapter VIII]{KM_book}.
The perturbation data appears as a closed differential $2$--form
added to the monopole equation,
satisfying certain conditions and leading to a perturbed chain complex.
Moreover,
the resulting homology depends only on the de Rham cohomology class.
As special cases,
the above usual or completed versions are 
monopole Floer homology with exact perturbations.
However, non-exact perturbations are relevant to our subsequent discussion.

We focus on the perturbed ``to'' version of monopole Floer homology,
as it suffices for our application.
Let $M$ be a connected, closed, oriented $3$--manifold,
and $\mathfrak{s}$ be a $\mathrm{Spin}^{\mathtt{c}}$ structure on $M$.

We say that a cohomology class $c\in H^2(M;\Real)$ is \emph{balanced},
with respect to $\mathfrak{s}$,
if $c=-2\pi^2\,c_1(\mathfrak{s})$ (as real cohomology classes),
or \emph{positively monotone} if $c=2\pi^2\,(t-1)\,c_1(\mathfrak{s})$
for some $t>0$, 
or \emph{negatively monotone} if $c=2\pi^2\,(t-1)\,c_1(\mathfrak{s})$
for some $t<0$
\cite[Definition 29.1.1]{KM_book}.
In particular, the \emph{exact} class $c=0$ is positively monotone.
In all these cases,
the \emph{monopole Floer homology}
$\widecheck{\mathrm{HM}}_*(M,\mathfrak{s},c)$
(with $\Integral$ coefficients) can be defined,
as a $\Integral/2\Integral$--graded module over $\Integral[U]$,
where $U$ is a fixed indeterminant.
In the exact case $c=0$, 
we simply denote $\widecheck{\mathrm{HM}}_*(M,\mathfrak{s})$.

\begin{lemma}\label{perturbed_hm_to}
		Let $M$ be a connected, closed, oriented $3$--manifold,
		and $\mathfrak{s}$ be any $\mathrm{Spin}^{\mathtt{c}}$ structure on $M$.
		If $c_1(\mathfrak{s})\in H^2(M;\Integral)$ is not torsion,
		and if $c\in H^2(M;\Real)$ is balanced or positively monotone with respect to $\mathfrak{s}$,
		then there exists an isomorphism of $\Integral/2\Integral$--graded $\Integral[U]$--modules
		$$\widecheck{\mathrm{HM}}_*(M,\mathfrak{s},c)\cong
		\widecheck{\mathrm{HM}}_*(M,\mathfrak{s}).$$
\end{lemma}

This is a special case of \cite[Theorems 31.1.1 and 31.1.2]{KM_book}.
In fact, the similar conclusion holds for 
$\widecheck{\mathrm{HM}}_\bullet$, $\widehat{\mathrm{HM}}_\bullet$, and $\overline{\mathrm{HM}}_\bullet$,
only over $\Integral[[U]]$.
However,
$\widecheck{\mathrm{HM}}_\bullet$ is always identical to
$\widecheck{\mathrm{HM}}_*$,
since any generator of their chain complexes
is annihilated by some sufficiently large power of $U$.

\begin{lemma}\label{hm_and_hf}
		Let $M$ be a connected, closed, oriented $3$--manifold,
		and $\mathfrak{s}$ be any $\mathrm{Spin}^{\mathtt{c}}$ structure on $M$.
		Then, there exists an isomorphism of $\Integral/2\Integral$--graded $\Integral[U]$--modules
		$$\widecheck{\mathrm{HM}}_*(M,\mathfrak{s},c_{\mathtt{b}})\cong
		\mathrm{HF}^+(M,\mathfrak{s}),$$
		where $c_{\mathtt{b}}=-2\pi\,c_1(\mathfrak{s})$ 
		denotes the balanced class
		in $H^2(M;\Real)$ with respect to $\mathfrak{s}$.
\end{lemma}

This is consequence of deep works due to Kutluhan, Lee, and Taubes 
\cite{KLT_hf_hm_i,KLT_hf_hm_ii,KLT_hf_hm_iii,KLT_hf_hm_iv,KLT_hf_hm_v}.
In the same series of papers,
they also prove the isomorphisms
$\widehat{\mathrm{HM}}_*(M,\mathfrak{s},c_{\mathtt{b}})\cong\mathrm{HF}^-(M,\mathfrak{s})$
and 
$\overline{\mathrm{HM}}_*(M,\mathfrak{s},c_{\mathtt{b}})\cong\widehat{\mathrm{HF}}(M,\mathfrak{s})$,
fitting into parallel exact triangles.
See \cite[Main Theorem]{KLT_hf_hm_i}.

\subsection{Periodic Floer homology}\label{Subsec-hp}
Periodic Floer homology is a kind of Floer homology
for mapping classes of oriented closed surfaces.
Upon an auxiliary choice of a smooth area form
and a generic area-preserving representative of the mapping class, 
the chain complex is generated by finite collection of periodic orbits with multiplicity.
Passing to the mapping torus, one may interprete the generators
as collections of periodic trajectories.
The product of $\Real$ with the mapping torus is equipped with 
a naturally induced symplectic structure.
The differential operator is defined by certain counting
pseudo-holomorphic curves in the product symplectic $4$--manifold
connecting between different generators drawn on the $-\infty$ and $+\infty$ ends,
upon a choice of a tame almost complex structure.
In this setting, there is a notion of monotonicity,
which serves as an admissibility condition for ensuring finiteness of the counting.
See \cite[Section 2]{Hutchings_Sullivan_hp}
for a more detailed review of periodic Floer homology;
see also \cite[Section 1.1]{Lee_Taubes_hp_swf} and Remark \ref{hp_and_hm_remark}.

We need some notations to recall relevant facts to our discussion.
Let $S$ be a connected, closed, orientable, smooth surface, 
equipped with an area form $\omega_S$ and the induced orientation.
Let 
$$f\colon S\to S$$ 
be an area-preserving diffeomorphism,
such that $f^m$ has only non-degenerate fixed points for all $m\in\Natural$,
(see Section \ref{Subsec-monodromy_entropy}).
Denote by $M_f$ the mapping torus of $f$.
For consistency of this paper, we adopt the dynamical convention,
namely, $M_f$ is the quotient of $S\times\Real$ by the infinite cyclic group action
$(x,r)\mapsto(f^{-1}(x),r+1)$,
(compare Remark \ref{hp_and_hm_remark}).
Under the setting $(S,\omega_S,f)$,
there is a distinguished second cohomology class $[w_f]\in H^2(M_f;\Real)$,
with the property
$$\langle [w_f],[S]\rangle=\int_S\omega_S.$$
As a de Rham cohomology class, $[w_f]$ is represented by 
a closed $2$--form $w_f$ on $M_f$,
obtained as the pull-back of $\omega_S$ via the factor projection $S\times\Real\to S$ 
descending to $M_f$.

There is a distinguished $\mathrm{Spin}^{\mathtt{c}}$ structure on $M_f$,
which we denote as $\mathfrak{s}_\theta$.
This distinguished $\mathrm{Spin}^{\mathtt{c}}$ structure
is represented by the velocity field of the (forward) suspension flow 
$\theta_t\colon M_f\to M_f$,
(namely, $\theta_t[x,r]=[x,r+t]$ for all $t\in \Real$ and $[x,r]\in M_f$).
Therefore, any other $\mathrm{Spin}^{\mathtt{c}}$ structure on $M_f$
can be written as 
$\mathfrak{s}_{\Gamma}=\mathfrak{s}_\theta+\mathrm{PD}(\Gamma)$ 
for some $\Gamma\in H_1(M_f;\Integral)$,
where $\mathrm{PD}(\Gamma)\in H^2(M_f;\Integral)$ 
denotes the Poincar\'{e} dual of $\Gamma$.
The first Chern class $c_1(\mathfrak{s}_\theta)\in H^2(M_f;\Integral)$
satisfies the relation 
$\langle c_1(\mathfrak{s}_\theta), [S]\rangle=\chi(S)$,
implying
\begin{equation}\label{s_Gamma_chern}
\langle c_1(\mathfrak{s}_\Gamma), [S]\rangle=\chi(S)+2\,\langle \mathrm{PD}(\Gamma),[S]\rangle
\end{equation}
for all $\Gamma\in H_1(M_f;\Integral)$.

For any $\Gamma\in H_1(M_f;\Integral)$,
we say that $\Gamma$ is \emph{positively monotone}
with respect to $[w_f]$, 
if $[w_f]=-\tau\,c_1(\mathfrak{s}_\Gamma)$ holds in $H^2(M_f;\Real)$
for some $\tau>0$.
In this case, 
the \emph{periodic Floer homology}
$\mathrm{HP}(f,S,\omega_S,\Gamma)$
(with $\Integral$ coefficients) can be defined,
as a $\Integral/2\Integral$--graded module over $\Integral$.
We also simply denote $\mathrm{HP}(f,\Gamma)$ 
when $(S,\omega_S)$ is assumed in the context.

\begin{lemma}\label{hp_and_hm}
		Let $S$ be a connected, closed, orientable, smooth surface, 
		equipped with an area form $\omega_S$.
		Let $f\colon S\to S$ be an area-preserving diffeomorphism,
		such that $f^m$ has only non-degenerate fixed points for all $m\in\Natural$.
		Then, for any $\Gamma\in H_1(M_f;\Integral)$ that is positively monotone
		with respect to $[w_f]$,
		there exists an isomorphism of $\Integral/2\Integral$--graded $\Integral$--modules
		$$\mathrm{HP}(f,\Gamma)\cong
		\widecheck{\mathrm{HM}}_*(M_f,\mathfrak{s}_\Gamma).$$
\end{lemma}

This is a weaker statement of a deep theorem
due to Lee and Taubes \cite[Theorem 1.1]{Lee_Taubes_hp_swf}.
Note that the isomorphism of \cite[Theorem 1.1]{Lee_Taubes_hp_swf}
is formulated in terms of monopole Floer cohomology 
adopting the topological convention for mapping tori.
In the statement of Lemma \ref{hp_and_hm},
we have reformulated in terms of monopole Floer homology
adopting the dynamical convention for mapping tori.
See Remark \ref{hp_and_hm_remark} for details.

\begin{remark}\label{hp_and_hm_remark}
The mapping torus in \cite{Lee_Taubes_hp_swf}
is constructed as the quotient of $S\times\Real$
by the infinite cyclic group action $(x,r)\mapsto (f(x),r+1)$.
Denoting it temporarily as $M^!_f$,
we still orient $M^!_f$, and obtain $\theta^!_t$,
and $w^!_f$ using the same objects on $S\times\Real$ as we did with $M_f$.
Note $\langle c_1(\mathfrak{s}_{\theta^!}),[S]\rangle = \chi(S)$.
There is an orientation-reversing isomorphism $\sigma\colon M_f\to M^!_f$, 
defined as $\sigma([x,r])=[x,-r]$.
We observe $[w^!_f]=\sigma^*[w_f]$, and 
$c_1(\mathfrak{s}_{\Gamma^!})=\sigma^*c_1(\mathfrak{s}_{\sigma_*\Gamma^!})$.
With these notations, our positive monotonicity condition is equivalent 
to saying that $\Gamma^!$ is positively monotone with respect to $\sigma^*[w_f]$
in the sense of \cite[Definition 1.1]{Lee_Taubes_hp_swf},
if and only if $\sigma_*\Gamma^!$ is positively monotone with respect to $[w_f]$
in our terms.
In this case,
\cite[Theorem 1.1]{Lee_Taubes_hp_swf} asserts an isomorphism
$$\mathrm{HP}(f,\sigma_*\Gamma^!)\cong
\mathrm{HM}^{-*}(M^!_f,\mathfrak{s}_{\Gamma^!},c_+),$$
where $c_+\in H^2(M^!_f;\Real)$ is any positively monotone perturbation class 
with respect to $\mathfrak{s}_{\Gamma^!}$.
In particular, this applies to the exact pertrubation case $c_+=0$,
and the right-hand side becomes
$$\mathrm{HM}^{-*}(M^!_f,\mathfrak{s}_{\Gamma^!})
\cong\mathrm{HM}_*(M_f,\mathfrak{s}_{\sigma_*\Gamma^!})
\cong\widecheck{\mathrm{HM}}_*(M_f,\mathfrak{s}_{\sigma_*\Gamma^!}).$$
Here, the first isomorphism is by duality \cite[Corollary 22.5.11]{KM_book};
the second isomorphism follows from the fact that
$c_1(\mathfrak{s}_{\sigma_*\Gamma^!})\in H^2(M_f;\Integral)$ is not torsion,
since the positive monotonicity 
implies $\langle c_1(\mathfrak{s}_{\sigma_*\Gamma^!}),[S]\rangle <0$,
(see \cite[Theorems 31.1.1 and 31.5.1]{KM_book}).
Rewriting $\sigma_*\Gamma^!$ as $\Gamma$,
we derive the statement of Lemma \ref{hp_and_hm} 
from the above isomorphisms.
\end{remark}

\subsection{Symplectic Floer homology}
Periodic Floer homology generalizes what is called symplectic Floer homology,
recovering the latter as the ``$1$--periodic'' part.
Their connection is summarized below.
We refer to \cite{Seidel_sfh} for Seidel's original definition of symplectic Floer homology;
see also \cite[Section 2.1]{Cotton-Clay_sfh} for a review.

Let $S$ be a connected, closed, orientable, smooth surface, 
equipped with an area form $\omega_S$ and the induced orientation.
Let 
$$f\colon S\to S$$ 
be an area-preserving diffeomorphism with only non-degenerate fixed points.
The symplectic Floer homology of $(f,S,\omega_S)$ 
is constructed following the same procedure as
used in constructing the periodic Floer homology,
except for two major differences.
First, the chain complex of the symplectic Floer homology
is only generated by the fixed points of $f$.
Secondly, the differential operator is defined
under a weaker monotonicity condition,
which has no requirement on a prescribed homology class.

To be precise, denote by $\mathrm{Fix}(f)\subset S$ the fixed point set of $f$.
The chain complex $\mathrm{CF}(f)=\mathrm{CF}(f,S,\omega_S)$
is freely generated by $\mathrm{Fix}(f)$ over $\Integral$.
We endow $\mathrm{CF}(f)$ with a $\Integral/2\Integral$--grading,
by assigning a degree $\epsilon(p)\in\Integral/2\Integral$ of each generator
$p\in\mathrm{Fix}(f)$, such that $(-1)^{\epsilon(p)}$ is the fixed point index
$\mathrm{ind}(f;p)$ of $f$ at $p$.
Furthermore, 
for any fixed point class $\mathbf{q}\in\mathscr{F}\mathrm{ix}(f)$,
we denote by $\mathrm{CF}(f,\mathbf{q})$ the $\Integral/2\Integral$--graded $\Integral$--submodule
of $\mathrm{CF}(f)$ freely generated by all $p\in\mathbf{q}$.
So, $\mathrm{CF}(f)$ splits as the direct sum of all  
$\mathrm{CF}(f,\mathbf{q})$. 
(See Section \ref{Subsec-monodromy_entropy}.)

We say that $f$ is \emph{monotone}
if $[w_f]=-\tau\,c_1(\mathfrak{s}_\theta)$ holds for some $\tau\in\Real$,
where $[w_f],c_1(\mathfrak{s}_\theta)\in H^2(M_f;\Real)$
are the same as in Section \ref{Subsec-hp};
(see \cite[Definition 1.1]{Cotton-Clay_sfh}).
In this case,
the differential operator $\partial\colon \mathrm{CF}(f)\to \mathrm{CF}(f)$
is well-defined,
roughly speaking, 
by counting pseudo-holomorphic cylinders in $M_f\times\Real$
connecting between $1$--periodic trajectories on $M_f\times\{-\infty\}$ and $M_f\times\{+\infty\}$.
By the mapping torus characterization of Nielsen equivalence,
the $1$--periodic trajectories of the suspension flow in $M_f$ 
correspond bijectively with the fixed points of $f$,
and their free homotopy classes in $M_f$ correspond bijectively with
the Nielsen equivalence classes of the fixed points of $f$,
\cite[Chapter I, Theorem 1.10]{Jiang_book}.
In particular,
for any generator $p\in \mathrm{Fix}(f)$,
the coefficient of $\partial p$ is nonzero at $q\in\mathrm{Fix}(f)$
only if $q$ is Nielsen equivalent to $p$, (see Section \ref{Subsec-monodromy_entropy}).
Moreover, $\partial$ switches the $\Integral/2\Integral$--grading.
So, the \emph{symplectic Floer homology} 
$\mathrm{HF}(f)=\mathrm{HF}(f,S,\omega_S)$ 
is defined as the homology of the chain complex $(\mathrm{CF}(f),\partial)$.
It is a $\Integral/2\Integral$--graded $\Integral$--module, and splits into 
a direct sum of $\Integral/2\Integral$--graded $\Integral$--submodules:
$$\mathrm{HF}(f)=\bigoplus_{\mathbf{q}\,\in\,\mathscr{F}\mathrm{ix}(f)}\,\mathrm{HF}(f,\mathbf{q}),$$
where $\mathrm{HF}(f,\mathbf{q})$ denotes the homology of
the chain subcomplex $(\mathrm{CF}(f,\mathbf{q}),\partial)$.

\begin{lemma}\label{sfh_and_hp}
		Let $S$ be a connected, closed, orientable, smooth surface, 
		equipped with an area form $\omega_S$.
		Let $f\colon S\to S$ be an area-preserving diffeomorphism
		with only non-degenerate fixed points.
		If $f$ is monotone,
		then for any $\Gamma\in H_1(M_f;\Integral)$
		with $\langle \mathrm{PD}(\Gamma),[S]\rangle=1$,
		the same construction as in the positively monotone case
		is valid for the setting $(f,\Gamma)$,
		resulting in the same periodic Floer homology 
		$\mathrm{HP}(f,\Gamma)$ up to natural isomorphism.
		In this case, the following identification 
		of $\Integral/2\Integral$--graded $\Integral$--modules holds.
		$$\mathrm{HF}(f)=\bigoplus_{\langle \mathrm{PD}(\Gamma),[S]\rangle=1}\, \mathrm{HP}(f,\Gamma)$$
\end{lemma}

This is elaborated in \cite[Appendix 7.2]{Lee_Taubes_hp_swf}.

\begin{lemma}\label{sfh_and_N}
		Let $S$ be a connected, closed, orientable, smooth surface, 
		equipped with an area form $\omega_S$.
		Let $f\colon S\to S$ be an area-preserving diffeomorphism
		with only non-degenerate fixed points.
		If $f$ is monotone, then the following inequality holds.
		$$\mathrm{dim}_\Rational\,\Rational\otimes_\Integral \mathrm{HF}(f)\geq N(f),$$
		where $N(f)$ denotes the Nielsen number of $f$.
\end{lemma}

This is analogous to the usual Morse inequality.
In fact, the Euler characteristic of $\mathrm{HF}(f,\mathbf{q})$
(that is, the free rank of $\mathrm{HF}_{\mathrm{even}}(f,\mathbf{q})$ minus
the free rank of $\mathrm{HF}_{\mathrm{odd}}(f,\mathbf{q})$) 
can be identified with the fixed point class index of $f$ at $\mathbf{q}$
(see Section \ref{Subsec-monodromy_entropy}):
$$\chi(\mathrm{HF}(f,\mathbf{q}))=\chi(\mathrm{CF}(f,\mathbf{q}))=\mathrm{ind}(f;\mathbf{q}).$$
The absolute value of the Euler characteristic 
is a lower bound for the free rank of $\mathrm{HF}(f,\mathbf{q})$.
So, the asserted inequality in Lemma \ref{sfh_and_N} follows immediately from 
the defining expression (\ref{Nielsen_number_def}) of the Nielsen number $N(f)$.

\begin{remark}\label{sfh_and_N_remark}
With coefficients modulo $2$,
Cotton-Clay \cite{Cotton-Clay_sfh} 
shows that
the symplectic Floer homology can be algorithmically computed,
based on the Nielsen--Thurston normal forms 
due to Jiang and Guo \cite{Jiang_Guo_fixed_points}. 
\end{remark}

\subsection{Summary}\label{Subsec-next_to_top}
Linking up the above recalled facts, 
we prove the statement (4)
of Proposition \ref{constraint_HF} as follows.

Let $M$ be a connected, closed, oriented $3$--manifold.
Suppose that $\phi\in H^1(M;\Integral)$ is a primitive fibered class.
Denote by $S\subset M$ 
an embedded connected closed oriented subsurface of genus $g\geq3$,
representing $\mathrm{PD}(\phi)\in H_2(M;\Integral)$.
Denote by $[f]\in\mathrm{Mod}(S)$ the monodromy of $(M,\phi)$.
We can rewrite $M$ as the mapping torus $M_f$.

For any $\Gamma\in H_1(M_f;\Integral)$,
the condition $\langle \mathrm{PD}(\Gamma),[S]\rangle=1$
holds if and only if 
$\langle c_1(\mathfrak{s}_\Gamma),[S]\rangle=2g-4$,
by (\ref{s_Gamma_chern}).
Moreover, 
for $[f]$ to admit an area-preserving diffeomorphic representative
with only non-degenerate fixed points,
such that $\Gamma$ is positively monotone with respect to $(f,S,\omega_S)$,
the condition $g-2>0$ is sufficient and necessary.
In fact,
the necessity follows directly from the fact $\langle [w_f],[S]\rangle>0$;
the sufficiency can be shown by perfoming Hamiltonian isotopy within $[f]$, 
(see \cite[Section 1.1]{Lee_Taubes_hp_swf}).
Therefore, the aforementioned lemmas are all applicable under our assumption $g\geq3$.

Applying Lemmas \ref{perturbed_hm_to}, \ref{hm_and_hf}, \ref{hp_and_hm}, and \ref{sfh_and_hp},
we deduce
$$\mathrm{HF}(f)\cong\mathrm{HF}^+(M_f,\mathrm{PD}([S]),g-2)
=\mathrm{HF}^+(M,\phi,g-2).$$

Applying Lemma \ref{sfh_and_N},
we conclude
$$\dim_\Rational\,\Rational\otimes_\Integral\mathrm{HF}^+(M,\phi,g-2)\geq N(f).$$

This establishes the asserted inequality in Proposition \ref{constraint_HF} (4).

\bibliographystyle{amsalpha}


\end{document}